\documentclass[11pt]{article}
	
\usepackage{amsmath, amsthm, amsfonts, amsxtra, fullpage, color, hyperref, url}
\usepackage{graphicx}
\usepackage[inline]{enumitem}
\numberwithin{equation}{section}

\definecolor{refkey}{gray}{.5}   
\definecolor{labelkey}{gray}{.5} 

\newcommand{\K}{\mathbf{K}}
\newcommand{\E}{\mathbf{E}}
\newcommand{\N}{\mathbb{N}}
\newcommand{\T}{\mathbf{T}}
\newcommand{\Rmelonref}{R^{\reflect, n}_{0 \to T}}
\newcommand{\Rmelonabs}{R^{\absorb, n}_{0 \to T}}

\newcommand{\Prob}{\mathbb{P}}
\newcommand{\bigO}{\mathcal{O}}
\newcommand{\intZ}{\mathbb{Z}}
\newcommand{\realR}{\mathbb{R}}
\newcommand{\compC}{\mathbb{C}}
\newcommand{\NIBMT}{\ensuremath{\text{NIBM}_{0 \to T}}}
\newcommand{\NIBMref}{\ensuremath{\text{NIBM}^{\reflect}_{0 \to T}}}
\newcommand{\NIBMabs}{\ensuremath{\text{NIBM}^{\absorb}_{0 \to T}}}
\newcommand{\lHopital}{l'H\^{o}pital}
\newcommand{\Backlund}{B\"{a}cklund}
\newcommand{\Painleve}{Painlev\'{e}}
\newcommand{\W}{\smash{\overset{\circ}{W}}}
\newcommand{\sg}{\sigma}

\newcommand{\z}{\zeta}
\newcommand{\de}{\delta}
\renewcommand{\d}{\partial}
\newcommand{\Sg}{\Sigma}
\newcommand{\al}{\alpha}
\newcommand{\1}{{\bf 1}}
\newcommand{\B}{{\bf B}}
\newcommand{\A}{\mathcal{A}}
\newcommand{\lcal}{\mathcal{L}}
\newcommand{\Veto}{Vet\H{o}}

\newtheorem{thm}{Theorem}[section]

\newtheorem{lem}[thm]{Lemma}
\newtheorem{prop}[thm]{Proposition}
\newtheorem{conj}[thm]{Conjecture}
\theoremstyle{remark}
\newtheorem{rmk}{Remark}[section]

\DeclareMathOperator{\reflect}{ref}
\DeclareMathOperator{\absorb}{abs}
\DeclareMathOperator{\even}{\,even}
\DeclareMathOperator{\odd}{\,odd}
\DeclareMathOperator{\tac}{tac}
\DeclareMathOperator{\Pearcey}{P}
\DeclareMathOperator{\diag}{diag}
\DeclareMathOperator{\xxx}{xxx}
\DeclareMathOperator{\Ai}{Ai}
\DeclareMathOperator{\ext}{ext}

\title{Nonintersecting Brownian bridges between reflecting or absorbing walls}

\author{
Karl Liechty \thanks{Department of Mathematical Sciences, DePaul University, Chicago, IL, 60614 USA
\href{mailto:kliechty@depaul.edu}{\nolinkurl{kliechty@depaul.edu}}}
 \and
Dong Wang\thanks{Department of Mathematics, National University of Singapore, Singapore, 119076, \href{mailto:matwd@nus.edu.sg}{\nolinkurl{matwd@nus.edu.sg}}}
}

\begin{document}

\maketitle

\begin{abstract}
We study a model of nonintersecting Brownian bridges on an interval with either absorbing or reflecting walls at the boundaries, focusing on the point in space-time at which the particles meet the wall. These processes are determinantal, and in different scaling limits when the particles approach the reflecting (resp. absorbing) walls we obtain hard-edge limiting kernels which are the even (resp. odd) parts of the Pearcey and tacnode kernels. We also show that in the single time case, our hard-edge tacnode kernels are equivalent to the ones studied by Delvaux \cite{Delvaux13a}, defined in terms of a $4\times 4$ Lax pair for the inhomogeneous \Painleve\ II equation (PII). As a technical ingredient in the proof, we construct a Schlesinger transform for the $4 \times 4$ Lax pair in \cite{Delvaux13a} which preserves the Hastings--McLeod solutions to PII.
\end{abstract}

\section{Introduction}

Consider a Brownian motion on the interval $[0,\pi]$ with either an absorbing or reflecting boundary condition at the endpoints of the interval. By the reflection principle \cite[Section X.5]{Feller71}, if both boundaries are reflecting the transition probability is
\begin{equation}\label{eq:reflecting_tp}
  P^{\reflect}_{\sigma}(x, y; t) =  \frac{1}{\sqrt{2\pi t} \sigma} \sum^{\infty}_{k = -\infty} e^{-\frac{(y - x + 2k\pi)^2}{2t \sigma^2}} + e^{-\frac{(y + x + 2k\pi)^2}{2t \sigma^2}}, \end{equation}
and if both boundaries are absorbing it is\footnote{The function \eqref{eq:reflecting_tp} is a genuine probability density on $[0,\pi]$, meaning that for any $0\le x\le \pi$ the integral with respect to $y$ over $[0,\pi]$ is 1. The function \eqref{eq:absorbing_tp} has total integral less than 1 because there is some probability that the particle is absorbed by one of the walls.} 
\begin{equation}\label{eq:absorbing_tp}
  P^{\absorb}_{\sigma}(x, y; t) = \frac{1}{\sqrt{2\pi t} \sigma} \sum^{\infty}_{k = -\infty} e^{-\frac{(y - x + 2k\pi)^2}{2t \sigma^2}} - e^{-\frac{(y + x + 2k\pi)^2}{2t \sigma^2}}.
\end{equation}
In each case, $t$ is time and $\sg>0$ is the diffusion parameter. The model we consider is that of $n\in \N$ such Brownian motions on the interval $[0,\pi]$ which are conditioned not to intersect, which we denote
  \begin{equation}\label{eq:nonintersect}
  X_1(t)<X_2(t)< \dots<X_n(t), 
  \end{equation}
and we fix the diffusion parameter to be equal to the reciprocal square root of the number of Brownian paths, $\sg=n^{-1/2}$. In the absorbing boundary condition case, we also condition so that no particle is absorbed by either wall. Under these conditions we let the particles evolve over the time interval $t\in [0,T]$, and we furthermore we fix the starting and ending points for all particles to be at $0$, i.e.,
\begin{equation}\label{eq:confluent_ends}
  X_1(0)=X_1(T)=X_2(0)=X_2(T)=\dots=X_n(0)=X_n(T) =0.
 \end{equation}
Even though the conditions \eqref{eq:nonintersect} and \eqref{eq:confluent_ends} seem to be contradictory, it is well known that the model of nonintersecting paths with confluent starting/ending points is well defined as a limit of a model in which the starting and ending points of the nonintersecting paths are close to but not equal to one another, see Section \ref{algebraic}.
 
When the boundary conditions for the Brownian motion on $[0,\pi]$ are reflecting (resp. absorbing) we denote the ensemble of nonintersecting paths conditioned as described above as $\NIBMref$ (resp. $\NIBMabs$). These ensembles are \emph{determinantal processes}, which means that the correlation functions are described by a determinant involving a certain \emph{extended kernel} function. More specifically, fix $m$ times $0< t_1 < t_2 < \cdots <t_m<T$, and to each time $t_i$, fix $k_i$ points in the interval $(0,\pi)$, $0 \le x_1^{(i)} <x_2^{(i)}<\cdots<x_{k_i}^{(i)} < \pi$.  The multi-time correlation function is then defined as
\begin{multline} \label{eq:corr_function_defn}
  \Rmelonref(x^{(1)}_1, \dotsc, x^{(1)}_{k_1}; \dotsc; x^{(m)}_1, \dotsc, x^{(m)}_{k_m}; t_1, \dotsc, t_m) := \\
  \lim_{\Delta x \to 0} \frac{1}{(\Delta x)^{k_1 + \dotsb + k_m}} \Prob \left( \text{there is a particle in $[x^{(i)}_j, x^{(i)}_j + \Delta x)$ for $j = 1, \dotsc, k_i$ at time $t_i$} \right),
\end{multline}
for the $\NIBMref$ model, and similarly for the $\NIBMabs$ model, for which we denote the correlation function as $\Rmelonabs$.
Then there exist kernel functions $K^{\reflect}_{t_i, t_j}(x,y)$ and $K^{\absorb}_{t_i, t_j}(x,y)$ such that ($\xxx$ stands for either $\reflect$ or $\absorb$)
\begin{equation} \label{eq:defn_Rmelon_reflect}
  R^{\xxx, n}_{0 \to T}(x^{(1)}_1, \dotsc, x^{(1)}_{k_1}; \dotsc; x^{(m)}_1, \dotsc, x^{(m)}_{k_m}; t_1, \dotsc, t_m) = \det \left( K^{\xxx}_{t_i, t_j}\left(x^{(i)}_{l_i}, x^{(j)}_{l'_j}\right) \right)_{\substack{i, j = 1, \dotsc, m \\ l_i = 1, \dotsc, k_i \\ l'_j = 1, \dotsc, k_j}}\,.
\end{equation}

In an earlier paper \cite{Liechty-Wang14-2} of the current authors, we studied a very similar model of nonintersecting Brownian motions on the unit circle. That model can be also considered as one of nonintersecting paths on the interval, but with  periodic rather than absorbing or reflecting boundary conditions. For that model it was shown that the correlation kernel could be expressed in terms of a system of orthogonal polynomials with respect to a discrete Gaussian weight, and that it converged to the \emph{tacnode kernel} and the \emph{Pearcey kernel} in different scaling limits. The result was a formula for the tacnode kernel which involved certain solutions to the Flaschka--Newell Lax pair for the Painlev\'{e} II equation. The nonintersecting Brownian paths on the interval with absorbing or reflecting boundaries also have a correlation kernel expressed in terms of discrete Gaussian orthogonal polynomials, and in this paper we show that it converges under proper scaling limits to the even (resp. odd) part of the tacnode and Pearcey kernels in the case of reflecting (resp. absorbing) boundary conditions.

The dependence of the models $\NIBMref$ and $\NIBMabs$ on the total time $T$ is described heuristically as follows. Since all particles are forced to begin and end at the lower wall, if the total time $T$ is small then the particles will not have time to approach the upper wall at $x=\pi$, and the model is very close to a model of nonintersecting Brownian bridges on the half-line $[0,\infty)$. On the other hand, if $T$ is large enough then the particles will reach the barrier at $x=\pi$ around some critical time $t=t^{c}<T/2$ and the asymptotic limiting density of particles fills the entire interval $[0,\pi]$ throughout the time interval $[t^c, T-t^c]$ before beginning to return to the origin. These two cases are separated by a critical total time $T_c$, and it can be shown as in \cite{Liechty-Wang14-2} that this critical value is $T_c=\pi^2/2$. Therefore, as in \cite{Liechty-Wang14-2} we separate the $\NIBMref$ and $\NIBMabs$ models into three cases: subcritical ($T<\pi^2/2$), critical ($T=\pi^2/2$), and supercritical $T>\pi^2/2$, see Figure \ref{fig:Global_picture}. Note that this global picture is the same for both $\NIBMref$ and $\NIBMabs$.

\begin{figure}[ht]
    \centering
    \includegraphics[scale=0.5]{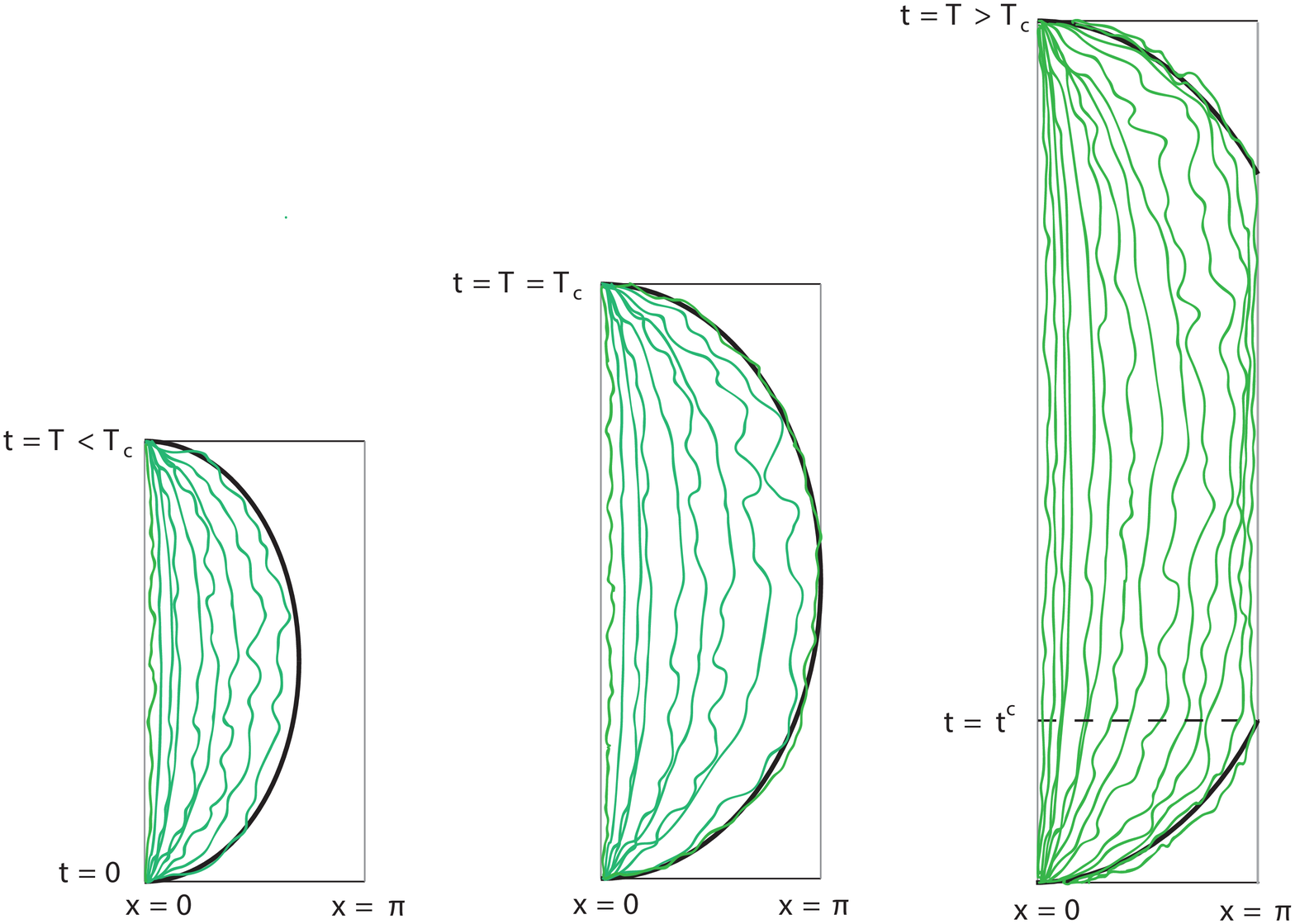}
    \caption{Typical configurations of paths in $\NIBMref$ and $\NIBMabs$ in the subcritical (left), critical (center), and supercritical (right) regimes. Time is on the vertical axis and space on the horizontal axis. As $n\to\infty$ the hull of the paths fills out the regions bounded by the thick curves. In the supercritical case, the time $t^c$ at which the particles reach the upper wall is marked. These figures are schematic and do not distinguish between reflecting walls, for which the particles may touch the walls, and absorbing walls for which they cannot.}
    \label{fig:Global_picture}
  \end{figure}

We remark that the nonintersecting Brownian motions on an interval with periodic, reflecting, and absorbing boundary conditions are related to the 2D Yang--Mills theory on a sphere with $\mathrm{U}(n)$, $\mathrm{O}(n)$, and $\mathrm{Sp}(n)$ gauge groups, respectively \cite{Gross-Matytsin95, Forrester-Majumdar-Schehr11}. The earlier paper \cite{Liechty-Wang14-2} dealt mainly with probabilistic aspects of the nonintersecting Brownian motion model with periodic boundary conditions, but has found applications in the study of 2D Yang--Mills theory, see \cite{Levy-Maida15, Gorsky-Milekhin-Nechaev16}. Similarly the current paper deals primarily with  the limiting local correlations for the nonintersecting Brownian motions on an interval with absorbing and reflecting boundary conditions, but we hope that the method and results obtained in this paper can shed light on the mathematico-physical aspect of the models as well.
On the other hand, the symmetry of the model revealed by the Yang--Mills theory gives a hint of the universality class that our probabilistic models belong to, see the discussion in Section \ref{subsubsec:univ_Pearcey}.

A precursor of $\NIBMabs$ is the nonintersecting Brownian excursion model studied in \cite{Tracy-Widom07}, which is equivalent to our $\NIBMabs$ with the two absorbing walls placed at $0$ and $+\infty$, or by a change of scale, the $T \to 0_+$ limit of our $\NIBMabs$. 

\subsection{Limiting correlation kernels}

\subsubsection{The extended Pearcey and tacnode kernels, and their odd and even parts}

In order to state our main results, we first present the explicit forms of the limiting correlation kernels that our nonintersecting Brownian motion models converge to. Hence we must define the Pearcey and tacnode processes.

The Pearcey process arises as a scaling limit when two separate groups of nonintersecting paths merge into a single group. It is a determinantal process and is thus defined by the (extended) Pearcey kernel \cite[Section 3]{Tracy-Widom06},
\begin{equation} \label{eq:Pearcey_kernel}
  K^{\Pearcey}_{s, t}(\xi, \eta) = \widetilde{K}^{\Pearcey}_{s, t}(\xi, \eta) - 1_{s<t}\phi_{s, t}(\xi, \eta),
\end{equation}
where 
\begin{equation} \label{eq:nonessential_Pearcey}
  \phi_{s, t}(\xi, \eta) =
    \frac{1}{\sqrt{2\pi(t - s)}} e^{-\frac{(\xi - \eta)^2}{2(t - s)}},
\end{equation}
and
\begin{equation} \label{eq:int_formula_Pearcey}
  \widetilde{K}^{\Pearcey}_{s, t}(\xi, \eta) = \frac{i}{(2\pi i)^2} \int_X dz \int_{\realR} dw \frac{e^{\frac{z^4}{4} + \frac{s z^2}{2} + i\xi z}}{e^{\frac{w^4}{4} + \frac{t w^2}{2} + i\eta w}} \frac{1}{z - w},
\end{equation}
where $X$ consists of four rays: one from $e^{\pi i/4} \cdot \infty$ to $0$, one from $e^{5\pi i/4} \cdot \infty$ to $0$, one from $0$ to $e^{3\pi i/4} \cdot \infty$, and one from $0$ to $e^{7\pi i/4} \cdot \infty$, see Figure \ref{fig:X}. Our definition of the Pearcey kernel is the same as that in \cite[Formula 1.2]{Adler-Orantin-van_Moerbeke10} up to a change of variables.

\begin{figure}[ht]
  \begin{minipage}[t]{0.3\linewidth}
    \centering
    \includegraphics{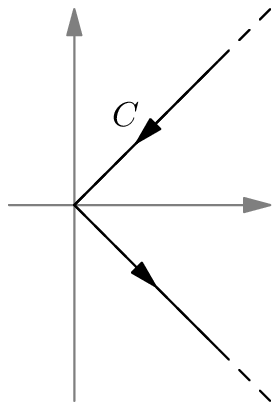}
    \caption{The shape of $C$.}
    \label{fig:C}
  \end{minipage}
  \begin{minipage}[t]{0.35\linewidth}
    \centering
    \includegraphics{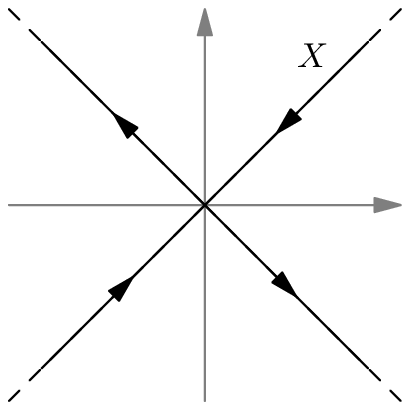}
    \caption{The shape of $X$.}
    \label{fig:X}
  \end{minipage}
  \hspace{\stretch{1}}
  \begin{minipage}[t]{0.35\linewidth}
    \centering
    \includegraphics{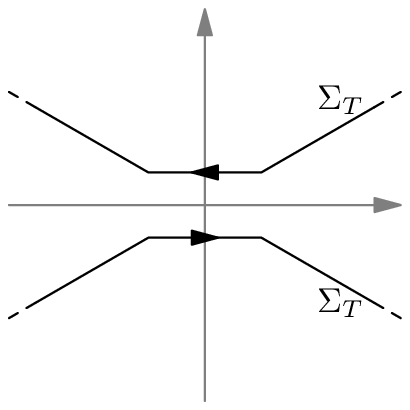}
    \caption{The shape of $\Sg_T$.}
    \label{fig:Sigma_T}
  \end{minipage}
\end{figure}

The tacnode process arises as a scaling limit when two separate groups of nonintersecting paths come together to meet at a single point in space-time, and then separate again. This process is also determinantal. It appears in several other models as the limiting process, and the correlation kernel has several equivalent definitions \cite{Delvaux-Kuijlaars-Zhang11, Johansson13, Adler-Ferrari-van_Moerbeke13, Ferrari-Veto12, Liechty-Wang14-2, Liechty-Wang16} with various generalities. Here we use the definition in terms of the Flaschka--Newell Lax pair for the Hastings--McLeod solution to the homogeneous \Painleve\ II equation, following \cite{Liechty-Wang14-2}. We only define the symmetric tacnode kernel, which corresponds to the case that the two groups of nonintersecting paths are of the same size. A straightforward generalization to the asymmetric form is given in \cite{Liechty-Wang16}. In order to define the kernel, we first must define some special functions which appear in the formula.

The homogeneous \Painleve\ II equation (PII) is the second order nonlinear ODE
\begin{equation}\label{cr2}
q''(s)=sq(s)+2q(s)^3\,,
\end{equation}
and the Hastings--McLeod solution \cite{Hastings-McLeod80} to PII
is the unique one that satisfies
\begin{equation}\label{cr3}
q(s)=\Ai(s)(1+o(1))\,, \qquad \textrm{as} \ s\to +\infty\,,
\end{equation}
where $\Ai(s)$ is the Airy function. Throughout this paper we let $q(s)$ be this particular solution to PII. The $2 \times 2$ matrix-valued differential equation
\begin{equation}\label{cr4}
  \frac{d}{d\z} \mathbf \Psi(\z;s) =
  \begin{pmatrix}
    -4i\z^2-i(s+2q(s)^2) & 4\z q(s)+2iq'(s) \\
    4\z q(s)-2iq'(s) & 4i\z^2 +i(s+2q(s)^2)
  \end{pmatrix}
  \mathbf \Psi(\z;s),
\end{equation}
was originally studied by Flaschka and Newell \cite{Flaschka-Newell80} as part of a Lax pair for PII, meaning that the compatibility of this equation with another differential equation given in \eqref{equiv2} implies that $q(s)$ solves PII. Throughout this paper we let $\mathbf{\Psi}(\zeta; s)$ be the solution to \eqref{cr4} that satisfies
\begin{equation}\label{cr5}
  \mathbf \Psi(\z; s)e^{i(\frac{4}{3} \z^3 +s \z)\sg_3} = I+O(\z^{-1})\,, \quad \z \to \pm \infty\, .
\end{equation}  
The asymptotics \eqref{cr5} extend into the sectors $-\pi/3< \arg \z < \pi/3$, and $2\pi/3 < \arg \z < 4\pi/3$, see e.g. \cite{Fokas-Its-Kapaev-Novokshenov06}.  Denote by $\Psi_{ij}(\z;s)$ the $(i,j)$ entry of the matrix $\mathbf \Psi(\z;s)$ defined in \eqref{cr4} and \eqref{cr5}.
It is convenient to also define the functions
\begin{equation}\label{tac18}
  f(u; s) :=
  \begin{cases}
    -\Psi_{12}(u;s) & \text{if $\Im u>0$}, \\
    \Psi_{11}(u;s) & \text{if $\Im u<0$}, 
  \end{cases}
  \qquad
  g(u,s) :=
  \begin{cases}
    -\Psi_{22}(u;s) & \text{if $\Im u>0$}, \\
    \Psi_{21}(u;s) & \text{if $\Im u<0$}.
  \end{cases}
\end{equation}
The extended tacnode kernel is now defined as
\begin{equation} \label{eq:tacnode_kernel}
  K^{\tac}_{s, t}(\xi, \eta; \sg) = \widetilde{K}^{\tac}_{s, t}(\xi, \eta; \sg) - \phi_{s, t}(\xi, \eta),
\end{equation}
where $\phi_{s, t}(\xi, \eta)$ is as in \eqref{eq:nonessential_Pearcey}, and
\begin{equation}\label{eq:essential_tacnode}
  \widetilde{K}^{\tac}_{s, t}(\xi, \eta; \sg):=\frac{1}{2\pi} \int_{\Sg_T} du \int_{\Sg_T} dv\, e^{\frac{s u^2}{2} - \frac{t v^2}{2}} e^{-i(u\xi - v\eta)} \frac{f(u;\sg)g(v;\sg)-g(u;\sg)f(v;\sg)}{2\pi i (u-v)}.
\end{equation}
Here $\Sg_T$ is a contour consisting of two pieces. One piece of $\Sg_T$ lies entirely above the real line, and goes from $e^{\pi i/6} \cdot \infty$ to $e^{5\pi i/6} \cdot \infty$.  The other piece lies entirely below the real line and goes from $e^{7\pi/6} \cdot \infty$ to $e^{11\pi/6} \cdot \infty$, see Figure \ref{fig:Sigma_T}. The convergence of the integrals in \eqref{eq:essential_tacnode} follows from the asymptotics \eqref{cr5}.

The symmetric tacnode process depends on one real parameter $\sg\in \realR$, whereas the Pearcey process contains no parameters.
If we view the tacnode process as the limit of two groups of particles in nonintersecting Brownian motions, the real parameter $\sigma$ which appears in $K^{\tac}$ measures the strength of interaction between the two groups. As $\sg\to -\infty$ the two groups become indistinguishable and the tacnode kernel has the sine kernel, which is the usual bulk scaling limit for nonintersecting paths, as a scaling limit. As $\sg\to +\infty$ the two groups of particles separate and there is no interaction between them \cite{Geudens-Zhang15, Girotti14a}.

We define the even versions of the Pearcey and tacnode kernels as
\begin{equation}\label{eq:def_even_part}
\begin{aligned}
  {K}^{\Pearcey, \even}_{s, t}(\xi, \eta) &:=  {K}^{\Pearcey}_{s, t}(\xi, \eta) +   {K}^{\Pearcey}_{s, t}(\xi, -\eta), \\
    {K}^{\tac, \even}_{s, t}(\xi, \eta; \sg)&:=  {K}^{\tac}_{s, t}(\xi, \eta; \sg)+  {K}^{\tac}_{s, t}(\xi, -\eta; \sg),
    \end{aligned}
    \end{equation}
and their odd versions as
\begin{equation}\label{eq:def_even_part}
\begin{aligned}
  {K}^{\Pearcey, \odd}_{s, t}(\xi, \eta) &:=  {K}^{\Pearcey}_{s, t}(\xi, \eta) -   {K}^{\Pearcey}_{s, t}(\xi, -\eta), \\
    {K}^{\tac, \odd}_{s, t}(\xi, \eta; \sg)&:=  {K}^{\tac}_{s, t}(\xi, \eta; \sg)- {K}^{\tac}_{s, t}(\xi, -\eta; \sg).
    \end{aligned}
    \end{equation}
Note that these are just the odd and even parts of the kernels with respect to the spatial variables, up to a factor of $2$. They describe symmetrized versions of the Pearcey and tacnode processes which we call the even (odd) Pearcey process and the even (odd) tacnode process. Similarly we refer to the kernels as the even (odd) Pearcey kernel and the even (odd) tacnode kernel. The even and odd Pearcey processes have appeared previously as limiting processes in the Plancherel growth models with $\mathrm{O}(\infty)$ symmetry \cite{Borodin-Kuan10, Kuan13}, and  with $\mathrm{Sp}(\infty)$ symmetry \cite{Cerenzia15}, respectively. In \cite{Borodin-Kuan10, Kuan13}, the kernel $K^{\Pearcey, \even}_{s, t}(\xi, \eta)$ appears under the name \emph{symmetric Pearcey kernel}, and in \cite{Cerenzia15} the kernel $K^{\Pearcey, \odd}_{s, t}(\xi, \eta)$ is simply referred to as \emph{Pearcey kernel}. We prefer to call them the even and odd Pearcey kernels because it clarifies the relation to the usual Pearcey kernel.

For convenience we write the following explicit formulas for the even and odd Pearcey kernels in terms of sine and cosine:
\begin{equation} \label{eq:Pearcey_kernel_sym}
\begin{aligned}
  K^{\Pearcey,\even}_{s, t}(\xi, \eta) &= \widetilde{K}^{\Pearcey, \even}_{s, t}(\xi, \eta) -  1_{s<t}(\phi_{s, t}(\xi, \eta) + \phi_{s, t}(\xi, -\eta)), \\
   K^{\Pearcey,\odd}_{s, t}(\xi, \eta) &= \widetilde{K}^{\Pearcey, \odd}_{s, t}(\xi, \eta) -  1_{s<t}(\phi_{s, t}(\xi, \eta) - \phi_{s, t}(\xi, -\eta)),
  \end{aligned}
\end{equation}
where $\phi_{s, t}(\xi, \eta)$ is defined in \eqref{eq:nonessential_Pearcey} and
\begin{equation} \label{eq:int_formula_Pearcey_even}
\begin{aligned}
  \widetilde{K}^{\Pearcey,\even}_{s, t}(\xi, \eta) &= \frac{i}{(2\pi i)^2} \int_X dz \int_{\realR} dw \frac{e^{\frac{z^4}{4} + \frac{s z^2}{2}}}{e^{\frac{w^4}{4} + \frac{t w^2}{2}}} \frac{2 e^{i\xi z}\cos(\eta w)}{z - w} \\
 &= \frac{1}{2\pi^2 i } \int_X dz \int_{\realR} dw \frac{e^{\frac{z^4}{4} + \frac{s z^2}{2}}}{e^{\frac{w^4}{4} + \frac{t w^2}{2}}} \frac{z\cos(\xi z)\cos(\eta w)}{z^2 - w^2},
    \end{aligned}
\end{equation}
and
\begin{equation} \label{eq:int_formula_Pearcey_odd}
\begin{aligned}
  \widetilde{K}^{\Pearcey,\odd}_{s, t}(\xi, \eta) &= \frac{i}{(2\pi i)^2} \int_X dz \int_{\realR} dw \frac{e^{\frac{z^4}{4} + \frac{s z^2}{2}}}{e^{\frac{w^4}{4} + \frac{t w^2}{2}}} \frac{2 ie^{i\xi z}\sin(\eta w)}{z - w} \\
   &= \frac{i}{2\pi^2} \int_X dz \int_{\realR} dw \frac{e^{\frac{z^4}{4} + \frac{s z^2}{2}}}{e^{\frac{w^4}{4} + \frac{t w^2}{2}}} \frac{z\sin(\xi z)\sin(\eta w)}{z^2 - w^2}.
      \end{aligned}
\end{equation}
The second lines of \eqref{eq:int_formula_Pearcey_even} and \eqref{eq:int_formula_Pearcey_odd} follow in a straightforward way from some symmetries of the integrand, and are in the same form as the symmetric Pearcey kernels given in \cite{Borodin-Kuan10} and \cite{Cerenzia15}, respectively.
Similarly the odd and even tacnode kernels are given as
\begin{equation} \label{eq:tacnode_kernel_sym}
\begin{aligned}
  K^{\tac, \even}_{\tau_i, \tau_j}(\xi,\eta; \sg) = {}& \widetilde{K}^{\tac, \even}_{\tau_i, \tau_j}(\xi,\eta; \sg) - 1_{s<t}(\phi_{s, t}(\xi, \eta) + \phi_{s, t}(\xi, -\eta)), \\
  K^{\tac, \odd}_{\tau_i, \tau_j}(\xi,\eta; \sg) = {}& \widetilde{K}^{\tac, \odd}_{\tau_i, \tau_j}(\xi,\eta; \sg) - 1_{s<t}(\phi_{s, t}(\xi, \eta) - \phi_{s, t}(\xi, \eta)),
\end{aligned}
\end{equation} 
where $\phi_{s, t}(\xi, \eta)$ is defined in \eqref{eq:nonessential_Pearcey}, $f$ and $g$ are defined in \eqref{tac18}, and
\begin{equation}
\widetilde{K}^{\tac, \even}_{s, t}(\xi,\eta; \sg)=\frac{1}{2\pi} \int_{\Sg_T}\,du\int_{\Sg_T}\, dv\, e^{\frac{su^2}{2}-\frac{tv^2}{2}} \frac{f(u; \sg)g(v; \sg)-g(u; \sg)f(v; \sg)}{2\pi i(u-v)} e^{-iu\xi} \cos(v\eta),
\end{equation}
and 
\begin{equation}
\widetilde{K}^{\tac, \odd}_{s, t}(\xi,\eta; \sg)=\frac{1}{2\pi} \int_{\Sg_T}\,du\int_{\Sg_T}\, dv\, e^{\frac{su^2}{2}-\frac{tv^2}{2}} \frac{f(u; \sg)g(v; \sg)-g(u; \sg)f(v; \sg)}{2\pi (u-v)} e^{-iu\xi} \sin(v\eta).
\end{equation}

\subsubsection{Main results}

The even (odd) Pearcey process and the even (odd) tacnode process appear as scaling limits in $\NIBMref$ ($\NIBMabs$). Namely, in the supercritical case $T>T_c$, the determinantal process defined by the kernel ${K}^{\Pearcey, \even}$ (resp. ${K}^{\Pearcey, \odd}$) appears as a scaling limit for $\NIBMref$  (resp. $\NIBMabs$) when space is scaled close to the upper wall $x=\pi$ and time is scaled close to $t^c$, the time at which the group of particles reaches the upper wall. Similarly, in the critical case $T=T_c=\pi^2/2$, the determinantal process defined by the kernel ${K}^{\tac, \even}$ (resp. ${K}^{\tac, \odd}$) appears as a scaling limit for $\NIBMref$  (resp. $\NIBMabs$) when space is scaled close to the upper wall $x=\pi$ and time is scaled close to $T/2$. 

The precise statement of these convergences is given the following theorem. In the statement of this theorem, we need the following notations: $T_c = \pi^2/2$ is the critical value of the total time; $t^c \in (0, T/2)$ depending on $T$ is the time when the limiting Pearcey process occurs if $T > T_c$; $d$ in part \ref{enu:thm:main_a} is a scaling parameter depending on $T$ and $t^c$ if $T > T_c$, while $d = 2^{-5/3} \pi$ in part \ref{enu:thm:main_b} has a similar role if $T$ is equal or close to $T_c$. The exact formulas for $t^c$ and $d$ in part \ref{enu:thm:main_a} will be given in Appendix \ref{sec:formulas_t^c_d}.

\begin{thm} \label{thm:main}
  Both \NIBMref\ and \NIBMabs\ are determinantal processes. Their multi-time correlation kernels, $K^{\reflect}_{t_i, t_j}(x, y; n, T)$ and $K^{\absorb}_{t_i, t_j}(x, y; n, T)$, have the following convergence properties:
  \begin{enumerate}[label=(\alph*)]
  \item \label{enu:thm:main_a}
    Assume $T > T_c=\pi^2/2$, and let $d$ be the constant that is specified in \eqref{eq:d_defn}. With the scalings
    \begin{equation}
      t_i = t^c + \frac{d^2}{2^{3/2} n^{1/2}} \tau_i, \quad t_j = t^c + \frac{d^2}{2^{3/2} n^{1/2}} \tau_j, \quad x = \pi - \frac{d}{(2n)^{3/4}}\xi, \quad y = \pi - \frac{d}{(2n)^{3/4}}\eta,
    \end{equation}
    the correlation kernels have the limits
    \begin{align}
      \lim_{n \to \infty} K^{\reflect}_{t_i, t_j}(x, y; n, T) \left\lvert \frac{dy}{d\eta} \right\rvert = {}&  K^{\Pearcey, \even}_{-\tau_j, -\tau_i}(\xi, \eta), \\
      \lim_{n \to \infty} K^{\absorb}_{t_i, t_j}(x, y; n, T) \left\lvert \frac{dy}{d\eta} \right\rvert = {}& K^{\Pearcey, \odd}_{-\tau_j, -\tau_i}(\xi, \eta).
    \end{align}
  \item  \label{enu:thm:main_b}
Fix $\sg\in \realR$, and let $T$ be scaled close to $T_c=\pi^2/2$ as 
    \begin{equation} \label{eq:T_scaling_tacnode}
     T = \frac{\pi^2}{2} \left( 1 - 2^{-\frac{2}{3}} \sigma (2n)^{-\frac{2}{3}} \right), \quad d = 2^{-5/3} \pi.
    \end{equation}
 With the scalings
    \begin{equation} \label{scaling_main_thm_crit}
      t_i = \frac{T}{2} + \frac{d^2}{2^{4/3} n^{1/3}} \tau_i, \quad t_j = \frac{T}{2} + \frac{d^2}{2^{4/3} n^{1/3}} \tau_j, \quad x = \pi - \frac{d}{(2n)^{2/3}} \xi, \quad y = \pi - \frac{d}{(2n)^{2/3}} \eta,
    \end{equation}
    the correlation kernels have the limits
    \begin{align}
      \lim_{n\to\infty} K^{\reflect}_{t_i, t_j}(x, y; n, T) \left\lvert \frac{dy}{d\eta} \right\rvert = {}& K^{\tac, \even}_{\tau_i, \tau_j}(\xi,\eta; \sg), \\
      \lim_{n\to\infty} K^{\absorb}_{t_i, t_j}(x, y; n, T) \left\lvert \frac{dy}{d\eta} \right\rvert = {}& K^{\tac, \odd}_{\tau_i, \tau_j}(\xi,\eta; \sg).
    \end{align}
  \end{enumerate}
\end{thm}

In this paper we consider only the limiting behavior of the top particles in the critical and supercritical phases. For the bottom particles, our method can show that they converge to the limiting Bessel process with parameter $-1/2$ for $\NIBMref$ and $1/2$ for $\NIBMabs$, in all phases, like the result in \cite{Tracy-Widom07}. The universal sine and Airy processes should also occur when we take limit at appropriate places. We omit further discussion on these limits.

\subsection{Hard-edge Pearcey process and relation to Pearcey process} \label{subsubsec:Pearcey_kernels}

In this and next subsections we discuss the relations among various limiting kernels. For Pearcey kernels (this subsection) we simply summarize known results, while for tacnode kernels (next subsection) some results presented here are new.

The even and odd Pearcey processes are special cases of a more general family of determinantal processes which depends on a real parameter $\al>-1$, and arises as a scaling limit in a model of nonintersecting (squared) Bessel paths. We refer to this family as the {\it hard-edge Pearcey process}. The single time version of this kernel for general $\al$ was first derived by Desrosiers and Forrester \cite{Desrosiers-Forrester08} in the context of random matrix theory, and by Kuijlaars, Mart\'inez-Finkelshtein, and Wielonsky \cite{Kuijlaars-Martinez_Finkelshtein-Wielonsky11} from nonintersecting squared Bessel paths. These two groups of authors gave slightly different formulations of the kernel and did not use the {\it hard-edge Pearcey} nomenclature.

The multi-time extended kernel was derived recently by Delvaux and \Veto\ \cite{Delvaux-Veto14}, also from nonintersecting squared Bessel paths with parameter $\al>-1$, and they call the limiting process the hard-edge Pearcey process. For general $\alpha$ the multi-time extended kernel is
\begin{equation} \label{eq:hard_Pearcey_defn}
  K^{\Pearcey, (\alpha)}_{s, t}(\xi, \eta) = \widetilde{K}^{\Pearcey, (\alpha)}_{s, t}(\xi, \eta) - \phi^{(\alpha)}_{s, t}(x, y) 1_{s < t},
\end{equation}
where
\begin{equation}
  \widetilde{K}^{\Pearcey, (\alpha)}_{s, t}(\xi, \eta) = \left( \frac{y}{x} \right)^{\frac{\alpha}{2}} \frac{2}{\pi i} \int_C dv \int^{\infty}_0 du \left( \frac{u}{v} \right)^{\alpha} \frac{uv}{v^2 - u^2} \frac{e^{\frac{v^4}{2} + sv^2}}{e^{\frac{u^4}{2} + tu^2}} J_{\alpha}(2\sqrt{y} u) J_{\alpha}(2\sqrt{x} v),
\end{equation}
and
\begin{equation}
  \phi^{(\alpha)}_{s, t}(x, y) = \frac{1}{t - s} \left( \frac{y}{x} \right)^{\frac{\alpha}{2}} e^{-\frac{x + y}{t - s}} I_{\alpha} \left( \frac{2\sqrt{xy}}{t - s} \right).
\end{equation}
Here $J_{\alpha}$ is the Bessel function of the first kind, $I_{\alpha}$ is the modified Bessel function of the first kind \cite{Abramowitz-Stegun64}, and the contour $C$ consists of two rays, one is from $e^{i\pi/4} \cdot \infty$ to $0$, and the other from $0$ to $e^{-i\pi/4} \cdot \infty$, as shown in Figure \ref{fig:C}.

The Bessel process reduces to Brownian motion on a half-line with a reflecting wall at the origin in the special case $\al=-1/2$, and is connected to Brownian motion on a a half-line with an absorbing wall at the origin in the special case $\al=1/2$, see e.g. \cite{Katori15}. Thus it is reasonable to guess that for $\al=\pm 1/2$, the hard-edge tacnode kernel reduces to the kernels $K^{\Pearcey, \odd}$ and $K^{\Pearcey, \even}$ which are scaling limits of nonintersecting paths in the presence of absorbing/reflecting walls.
 Using the explicit formulas of (modified) Bessel functions with $\alpha = \pm 1/2$, we see that in these cases the hard-edge Pearcey kernel reduces to
 \begin{subequations}\label{eq:HE_kernel_reduct}
\begin{align}
  \widetilde{K}^{\Pearcey, (1/2)}_{s, t}(\xi, \eta) = {}& \frac{2}{\pi^2 i} \frac{1}{\sqrt{x}} \int_C dv \int^{\infty}_0 du \frac{e^{\frac{v^4}{2} + sv^2}}{e^{\frac{u^4}{2} + tu^2}} \sin(2\sqrt{y} u) \sin(2\sqrt{x} v) \frac{u}{v^2 - u^2}, \\
  \widetilde{K}^{\Pearcey, (-1/2)}_{s, t}(\xi, \eta) = {}& \frac{2}{\pi^2 i} \frac{1}{\sqrt{y}} \int_C dv \int^{\infty}_0 du \frac{e^{\frac{v^4}{2} + sv^2}}{e^{\frac{u^4}{2} + tu^2}} \cos(2\sqrt{y} u) \cos(2\sqrt{x} v) \frac{v}{v^2 - u^2}, \\
  \phi^{(1/2)}_{s, t}(x, y) = {}& \frac{1}{2\sqrt{\pi}} \frac{1}{\sqrt{t - s}} \frac{1}{\sqrt{x}} \left( e^{-\frac{(\sqrt{x} - \sqrt{y})^2}{t - s}} - e^{-\frac{(\sqrt{x} + \sqrt{y})^2}{t - s}} \right), \\
  \phi^{(-1/2)}_{s, t}(x, y) = {}& \frac{1}{2\sqrt{\pi}} \frac{1}{\sqrt{t - s}} \frac{1}{\sqrt{y}} \left( e^{-\frac{(\sqrt{x} - \sqrt{y})^2}{t - s}} + e^{-\frac{(\sqrt{x} + \sqrt{y})^2}{t - s}} \right).
\end{align}
\end{subequations}

Comparing \eqref{eq:Pearcey_kernel_sym}--\eqref{eq:int_formula_Pearcey_odd} with \eqref{eq:HE_kernel_reduct}, it is straightforward to obtain the following proposition.

\begin{prop} \label{lem:Pearcey_relation}
  The hard-edge Pearcey kernel with $\alpha = \pm 1/2$ have the following relationship to the odd and even Pearcey kernels:
  \begin{align}
    K^{\Pearcey, \even}_{s, t}(\xi, \eta) = {}& \frac{\eta}{\sqrt{2}} K^{\Pearcey, (-1/2)}_{s/\sqrt{2}, t/\sqrt{2}}(2^{-\frac{3}{2}} \xi^2, 2^{-\frac{3}{2}} \eta^2), \label{eq:Pearcey_relation_ref} \\
    K^{\Pearcey, \odd}_{s, t}(\xi, \eta) = {}& \frac{\xi}{\sqrt{2}} K^{\Pearcey, (1/2)}_{s/\sqrt{2}, t/\sqrt{2}}(2^{-\frac{3}{2}} \xi^2, 2^{-\frac{3}{2}} \eta^2). \label{eq:Pearcey_relation_abs}
  \end{align}
\end{prop}

\subsection{Hard-edge tacnode process and relation to tacnode process}

The hard-edge tacnode process appears as a scaling limit of the lowest particles in nonintersecting (squared) Bessel processes when they just touch a hard-edge barrier \cite{Delvaux13}. As with the hard-edge Pearcey process, it is a determinantal process parametrized by the Bessel parameter $\alpha>-1$. Delvaux in \cite{Delvaux13} obtained a formula for a single-time limiting kernel depending on the parameter $\alpha>-1$, which we denote by $K^{\tac, (\al)}(x,y; s, \tau)$. His formulation is in terms of a certain $4\times 4$ Lax pair solution to the inhomogeneous Painlev\'{e} II equation, which is the same as \eqref{cr2} but with a additive constant term,
\begin{equation}\label{PII}
  q''(\sg)=2q(\sg)^3 + \zeta q(\sg) -\nu,
\end{equation}
where $\nu$ is a fixed constant. The particular solution to \eqref{PII} which appears in the Delvaux's tacnode kernel is also called the \emph{Hastings--McLeod} solution, and is defined to be the one which satisfies
\begin{equation} \label{eq:PII_BC}
  q(\sg) \sim \nu/ \sg, \quad \text{as $\sg \to +\infty$}, \quad q(\sg) \sim \sqrt{\frac{-\sg}{2}} \quad \text{as $\sg \to -\infty$}.
\end{equation}
The Bessel parameter $\al$ in the hard-edge tacnode kernel and the parameter $\nu$ in the PII equation are related by $\nu = \alpha+1/2$. Since the explicit formula for the hard-edge tacnode kernel is complicated, we relegate it to Section \ref{sec:kernels}. Currently the multi-time extended kernel for general $\alpha$ is not available. For integer $\alpha\ge 0$, a multi-time extended kernel was later obtained by Delvaux and \Veto\ \cite{Delvaux-Veto14} using a different approach to the same model. In this paper we are not going to use their extended kernel, so we omit the explicit formula, but only remark that their extended kernel is increasingly complicated as the integer $\alpha$ grows larger. 

Inspired by Proposition \ref{lem:Pearcey_relation}, it is natural to conjecture that $K^{\tac, \even}_{\tau_i, \tau_j}(\xi,\eta; \sg)$ and $K^{\tac, \odd}_{\tau_i, \tau_j}(\xi,\eta; \sg)$ are the correlation kernels for the hard-edge tacnode process with $\alpha = -1/2$ and $1/2$ respectively. However, the multi-time correlation kernel of the hard-edge tacnode process is not in literature, so we can only state a modest result on the one-time correlation kernel.
\begin{prop} \label{lem:tacnode_relation}
  The odd and even tacnode kernels have the following relations to the and the hard-edge tacnode kernel with $\alpha = \pm 1/2$:
  \begin{align}
    K^{\tac, (-1/2)}(x,y; s, \tau) = {}& 2^{5/3} \widetilde{K}^{\tac, \even}_{t, t}(2^{2/3}x,2^{2/3}y; \sg), \label{eq:ref_tac_relation} \\
    K^{\tac, (1/2)}(x,y; s, \tau) = {}& 2^{5/3} \widetilde{K}^{\tac, \odd}_{t, t}(2^{2/3}x,2^{2/3}y; \sg), \label{eq:abs_tac_relation}
  \end{align}
  where
  \begin{equation} \label{eq:relation_t_sigma_s_tau}
    t=2^{4/3} \tau, \quad \sg=2^{5/3} s - 2^{2/3}\tau^2.
  \end{equation}
\end{prop}
The proof of this proposition, especially the proof of \eqref{eq:abs_tac_relation}, is much more involved than the proof of Proposition \ref{lem:Pearcey_relation}, and it is given in Section \ref{sec:kernels}. Theorem \ref{thm:main}\ref{enu:thm:main_b} together with Proposition \ref{lem:tacnode_relation} shows that the \NIBMref\ and \NIBMabs\ converge to hard-edge tacnode process with parameters $\mp 1/2$ respectively, at the one-time correlation level. We conjecture that the convergence holds for multi-time correlations as well.

\subsubsection{Hard-edge tacnode kernel with half-integer $\alpha$ by Schlesinger transformation}

As we mentioned above, the hard-edge tacnode kernels with integer $\alpha$ are special in that the multi-time extended version of the kernel is available in the literature. They are also special in that they have Airy resolvent formulas that do not generalize to non-integer $\alpha$ in a direct way \cite{Delvaux-Veto14}. On the other hand, from Proposition \ref{lem:tacnode_relation} and formulas \eqref{eq:tacnode_kernel} and \eqref{eq:essential_tacnode}, we see that the hard-edge tacnode kernels with half-integer $\alpha$ are also special in that they, at least the first two in the sequence of infinitely many, can be expressed by the Lax pair associated to the Hastings--McLeod to the homogeneous ($\nu = 0$) PII equation. Actually Proposition \ref{lem:tacnode_relation} is not incidental, but gives the first two cases of a general result.
\begin{thm} \label{thm:half-integer}
  For any $\alpha = k - 1/2$ with $k = 0, 1, 2, \dotsc$, the one-time hard-edge tacnode kernel $K^{\tac, (\alpha)}(x, y; s, \tau)$ can be expressed in the form
  \begin{equation}
    K^{\tac, (k - 1/2)}(x, y; s, \tau) = \frac{1}{\pi} \int_s^\infty F_k(x; \widetilde{s}, \tau) F_k(y; \widetilde{s}, -\tau)\,d\widetilde{s},
  \end{equation}
  where the function $F_{\nu}(x; s, \tau)$ is defined in \eqref{eq:def_Fnu} for general $\nu > -1$ and inductively by \eqref{eq:integral_F_0} and \eqref{eq:inductive_F_nu+k} for integer-valued $\nu$. For positive integer $k$, $F_k(x; s, \tau)$ is expressed as a contour integral where the integrand is a linear combination of the entries of $\mathbf \Psi(\z; s)$ and $\z \mathbf \Psi(\z; s)$, where $\mathbf \Psi(\z; s)$ is defined by the Lax pair \eqref{cr4} and \eqref{cr5} associated to the Hastings--McLeod solution of the homogeneous ($\nu = 0$) PII equation. The coefficients are polynomials in $q_0(\sg), q_1(\sg), \dots, q_{k - 1}(\sg)$ and $q_0'(\sg), q_1'(\sg), \dots, q_{k - 1}'(\sg)$, where $q_\nu(\sg)$ is the Hastings--McLeod solution to the inhomogeneous PII equation \eqref{PII}, and $\sg=2^{2/3}(2s-\tau^2)$.
\end{thm}
The explicit formulas for $F_k(x; s, \tau)$ for $k = 0, 1, 2\dots$  are given in \eqref{eq:integral_F_0}, \eqref{eq:integral_F_1}, and \eqref{eq:integral_F_2}. The formulas become more complicated and less useful as $k$ increases. We remark that since the Hastings--McLeod solution to the PII equation with integer-valued $\nu$ can be expressed by the Hastings--McLeod solution of the homogeneous ($\nu = 0$) PII equation via the \Backlund\ transformation \eqref{eq:Backlund_q}, $F_k(x; s, \tau)$ can be expressed purely in terms of the Hastings--McLeod solution to the homogeneous PII equation, and the associated Lax pair \eqref{cr4} and \eqref{cr5}. To be precise, the statement that the coefficients of the entries of $\mathbf \Psi(\z; s)$ and $\z \mathbf \Psi(\z; s)$ in the integrand are polynomial in $q_0(\sg), q_1(\sg), \dots, q_{k - 1}(\sg)$ and $q_0'(\sg), q_1'(\sg), \dots, q_{k - 1}'(\sg)$ may be replaced with the statement that the coefficients are rational functions of $q_0(\sg)$ and $q_0'(\sg)$. See Sections \ref{sec:M_nu_defn} and \ref{sec:kernels} for detail.

The key ingredient of the proof of Theorem \ref{thm:half-integer} and \eqref{eq:abs_tac_relation} in Proposition \ref{lem:tacnode_relation} is the Schlesinger transformation of the $4 \times 4$ Lax pair discovered in \cite{Delvaux13a}. It is not surprising for experts on integrable differential equations that Lax pairs associated to PII equation with parameters $\nu$ and $\nu+1$ are related by a e Schlesinger transformation, just as the Painlev\'{e} functions themselves are related by an explicit \Backlund\ transformation \eqref{eq:Backlund_q}, but the explicit construction for the $4 \times 4$ Lax pair is new, see Section \ref{sec:M_nu_defn}.

For $\al=\pm 1/2$ the multi-time version of Theorem \ref{thm:half-integer} is given in an explicit form in Proposition \ref{prop:triple_int}. From results in Proposition \ref{lem:tacnode_relation} and Theorem \ref{thm:half-integer}, we can formally conjecture the generalization of Theorem \ref{thm:half-integer} to the multi-time extended kernel for the hard-edge tacnode kernel with general half-integer $\alpha$. But since the extended kernel is still missing, we do not pursue it in the current paper.


%


\subsection{Universality}

The (usual) Pearcey and tacnode processes occur as limiting processes in various models, and have established their statuses as universal limiting processes. The Pearcey kernel originally appeared in the papers \cite{Brezin-Hikami98a, Brezin-Hikami98} in the context of random matrix theory, see also \cite{Hachem-Hardy-Najim16, Mo10, Bleher-Kuijlaars07}, but also appears in models of nonintersecting paths \cite{Adler-Orantin-van_Moerbeke10,  Tracy-Widom06, Liechty-Wang14-2}, random growth models \cite{Okounkov-Reshetikhin07}, and random polymers \cite{Adler-van_Moerbeke-Wang11}. The tacnode kernel is newer so its scope in the literature is somewhat more limited, but it appears in various models of nonintersecting paths \cite{Johansson13, Ferrari-Veto12, Adler-Ferrari-van_Moerbeke13, Delvaux-Kuijlaars-Zhang11, Liechty-Wang14-2} and random tilings \cite{Adler-Johansson-van_Moerbeke14}. Here we comment briefly on the universal character of the hard-edge versions of the Pearcey and tacnode processes.

\subsubsection{Universality for hard-edge Pearcey process} \label{subsubsec:univ_Pearcey}

The hard-edge Pearcey process, as discussed in Section \ref{subsubsec:Pearcey_kernels}, has been mostly studied as the limit of nonintersecting squared Bessel process, while \cite{Desrosiers-Forrester08} considered the matrix model equivalent to the nonintersecting squared Bessel process. Very recently the hard-edge Pearcey process for general $\al>-1$ appeared as a scaling limit in an interacting particle system \cite{Cerenzia-Kuan16}. It is notable that as a limiting process, the $\alpha = -1/2$ hard-edge Pearcey process occurs in the Plancherel growth model with $\mathrm{O}(\infty)$ symmetry \cite{Borodin-Kuan10, Kuan13},  and the $\alpha = 1/2$ hard-edge Pearcey process occurs in the Plancherel growth model with $\mathrm{Sp}(\infty)$ symmetry \cite{Cerenzia15}. Here we remind the reader that our nonintersecting Brownian motion model between reflecting walls corresponds to the 2D Yang--Mills theory with $\mathrm{O}(n)$ symmetry, and the model between absorbing walls corrsponds to the 2D Yang--Mills theory with $\mathrm{Sp}(n)$ symmetry. The result in our paper is additional evidence that the hard-edge Pearcey process with $\alpha = \pm 1/2$ is universal, and indicates that they are features of symmetry classes.

As pointed out in \cite{Borodin-Kuan10} and \cite{Cerenzia15}, the hard-edge Pearcey processes with $\alpha = \pm 1/2$ are in the class of 2D anisotropic KPZ with a wall. It is generally acknowledged that the limiting behaviors of nonintersecting Brownian motions, without a wall, are in the 1D KPZ universality class \cite{Corwin-Hammond11}, and also in the 2D anisotropic KPZ class, without a wall \cite{Borodin-Ferrari08a}. The results in this paper shows the relation between nonintersecting Brownian motions with walls and class of 2D anisotropic KPZ with a wall.

\subsubsection{Universality for hard-edge tacnode process} 

For the more recent and more complicated hard-edge tacnode process, the results are scarcer. Outside of the current paper and the very recent preprint \cite{Ferrari-Veto16} (see below), the literature is all based on the model of nonintersecting squared Bessel process. Although it is too early to call the hard-edge tacnode process universal based on available results, the result in our paper is a hint that the hard-edge tacnode processes with $\alpha = \pm 1/2$ should be in the class of anisotropic KPZ with a wall, since they occur in the same models as the hard-edge Pearcey process with $\alpha = \pm 1/2$.

When the current paper was almost complete, we found the very recent preprint \cite{Ferrari-Veto16}. In this paper, the authors consider nonintersecting Brownian motions with one absorbing wall. They call the limiting process, which turns out to be exactly our limiting process with kernel $K^{\tac, \odd}$, \emph{hard-edge tacnode process}, without identifying it with the existing hard-edge tacnode process with $\alpha = 1/2$.
In Appendix \ref{sec:equivalence_FV}, we prove:
\begin{prop} \label{prop:FV_equiv}
  The multi-time correlation kernel $\widehat{K}^{\ext}(T_1, U_1; T_2, U_2)$ defined in \cite[Theorem 2.6]{Ferrari-Veto16} that depends on a parameter $R$ implicitly, satisfies
  \begin{equation}
    \widehat{K}^{\ext}(T_1, U_1; T_2, U_2) = 2^{2/3} K^{\tac, \odd}_{2^{7/3} T_1, 2^{7/3} T_2}(2^{2/3} U_1, 2^{2/3} U_2; 2^{2/3} R).
  \end{equation}
\end{prop}
The proof of Proposition \ref{prop:FV_equiv} is fairly involved. But a variation of our $\NIBMabs$ is very similar to the model in \cite{Ferrari-Veto16} and the equivalence of the limiting behaviors is easier to see, at least heuristically. We explain it in Appendix \ref{subsec:similar_model}.

Although it is not proven, the authors of \cite{Ferrari-Veto16} conjecture that the odd tacnode process occurs as a scaling limit in the six-vertex model with domain wall boundary conditions, or equivalently in random domino tilings, defined on a pentagonal domain, see \cite[Section 3]{Ferrari-Veto16}. This is of course further evidence of its universal character, and is deserving of further study.

\subsection{Plan of the rest of the paper}

The rest of the paper is organized as follows. In Section \ref{algebraic}, the determinantal structure of the processes $\NIBMref$ and $\NIBMabs$ is analyzed using the Eynard--Mehta formula, and the correlation kernels are described in terms of discrete Gaussian orthogonal polynomials. In Section \ref{sec:asymptotics}, these kernels are related to the model of nonintersecting Brownian paths on the unit circle studied in \cite{Liechty-Wang14-2} and Theorem \ref{thm:main} is proved based on the asymptotic results obtained in that paper. In Section \ref{sec:rank1} we prove that the derivative of the kernels $K^{\tac, \even}$ and $K^{\tac, \odd}$ with respect to the parameter $\sg$ is a rank-1 kernel, which is used later in the proof of Proposition \ref{lem:tacnode_relation}.

 The remaining sections  \ref{sec:M_nu_defn} and \ref{sec:kernels} deal with the hard-edge tacnode kernel as defined by Delvaux in \cite{Delvaux13}, and its relationship to the kernels $K^{\tac, \even}$ and $K^{\tac, \odd}$. In Section \ref{sec:M_nu_defn} we present the Lax system which describes the special functions appearing in Delvaux's formula for the hard-edge tacnode kernel. These special functions are given in the form of a $4\times 4$ matrix denoted $M_\nu$ which depends on a parameter $\nu>-1/2$ related to the parameter $\al>-1$ in the corresponding Bessel process by $\nu= \al+1/2$. In this section we formulate the Schlesinger transformation which enables us to write $M_{\nu+1}$ explicity in terms of $M_\nu$. Finally in Section \ref{sec:kernels} we define the hard-edge tacnode kernel of Delvaux and use the results obtained in Sections \ref{sec:M_nu_defn} and \ref{sec:kernels} to prove Proposition \ref{lem:tacnode_relation} and Theorem \ref{thm:half-integer}.

 In Appendix \ref{sec:formulas_t^c_d} we give explicit formulas for constants in Theorem \ref{thm:main}\ref{enu:thm:main_a}. In Appendix \ref{sec:equivalence_FV} we prove Proposition \ref{prop:FV_equiv}.
 
\section{Orthogonal polynomial formulas for the processes $\NIBMref$ and $\NIBMabs$}\label{algebraic}
In this section we derive a formula for the correlation kernels for the processes  $\NIBMref$ and $\NIBMabs$ in terms of a system of discrete Gaussian orthogonal polynomials, similar to the kernel derived in \cite{Liechty-Wang14-2} for nonintersecting paths on the unit circle. To describe the kernels we first define the lattice $L_n$ of mesh $1/n$,
\begin{equation}\label{def:lattice}
  L_n = \left\{ \frac{k}{n} \mid k \in \intZ \right\}.
\end{equation}
Now define the monic polynomials $p^T_{n, k}(s) = s^k + \cdots$ as the discrete Gaussian orthogonal polynomials satisfying the orthogonality condition
\begin{equation} \label{eq:defn_disc_Gaussian_poly}
  \frac{1}{n} \sum_{s \in L_n} p^T_{n, j}(s) p^T_{n, k}(s) e^{-\frac{nT s^2}{2}} =
    \begin{cases}
      0 & \text{if $j \neq k$}, \\
      h^T_{n, j} & \text{if $j = k$},
    \end{cases}
\end{equation}
where $\{h^T_{n, j}\}_{j=0}^\infty$ is a sequence of normalizing constants. These polynomials depend on two parameters: $n\in \N$, which is the number of particles in the $\NIBMref$ and $\NIBMabs$; and $T>0$ which is the total time of the processes $\NIBMref$ and $\NIBMabs$. Note that since the weight $e^{-\frac{nT s^2}{2}}$ is an even function of $s$ and the lattice $L_n$ is symmetric about the origin, the polynomials $p^T_{n, k}(s)$ are even for $k$ even and odd for $k$ odd. For $n,k \in \N$ and $a>0$ define also the discrete Fourier transform 
\begin{equation}\label{eq:S_ka_defn1}
  S_{k, a}(x; n, T) = \frac{1}{n} \sum_{s \in L_n} p^T_{n, k}(s) e^{-\frac{na s^2}{2}} e^{insx}.
\end{equation}
We then have the following proposition.

\begin{prop}\label{prop:op_formulas}
The processes $\NIBMref$ and $\NIBMabs$ are determinantal. 
  \begin{enumerate}[label=(\alph*)]
  \item \label{enu:prop_finite_n_ref}
The multi-time correlation kernel for $\NIBMref$ is given as
\begin{equation} \label{eq:kernel_ref}
 K^{\reflect}_{t_i, t_j}(x, y; n, T) = \widetilde{K}^{\reflect}_{t_i, t_j}(x, y; n, T) - \W^{\reflect}_{[i, j)}(x, y),
\end{equation}
where
\begin{equation} \label{eq:tilde_K_full_expansion}
  \widetilde{K}^{\reflect}_{t_i, t_j}(x, y; n, T) = \frac{n}{\pi} \sum^{n - 1}_{k = 0} \frac{1}{h^T_{n, 2k}} S_{2k, T - t_i}(x; n, T) S_{2k,  t_j}(-y; n, T),
\end{equation}
and
\begin{equation} \label{eq:W_circ_and_W}
    \W^{\reflect}_{[i, j)}(x, y) 
    =  \frac{\sqrt{n}}{\sqrt{2\pi t}} \sum^{\infty}_{k = -\infty} \left( e^{-\frac{n(y - x + 2k\pi)^2}{2(t_j - t_i)}} + e^{-\frac{n(y + x + 2k\pi)^2}{2(t_j - t_i)}} \right).
\end{equation}
 \item \label{enu:prop_finite_n_abs}
The multi-time correlation kernel for $\NIBMabs$ is given as
\begin{equation} \label{eq:kernel_ref}
 K^{\absorb}_{t_i, t_j}(x, y; n, T) = \widetilde{K}^{\absorb}_{t_i, t_j}(x, y; n, T) - \W^{\absorb}_{[i, j)}(x, y),
\end{equation}
where
\begin{equation} \label{eq:tilde_K_full_expansion_abs}
    \widetilde{K}^{\absorb}_{t_i, t_j}(x, y; n, T) =  \frac{n}{\pi} \sum^{n - 1}_{k = 0} \frac{1}{h^T_{n, 2k + 1}} S_{2k+1, T - t_i}(x; n, T)S_{2k+1, t_j}(-y; n, T),
\end{equation}
and
\begin{equation} \label{eq:W_circ_and_W_abs}
    \W^{\absorb}_{[i, j)}(x, y) = 
 \frac{\sqrt{n}}{\sqrt{2\pi t}} \sum^{\infty}_{k = -\infty} \left( e^{-\frac{n(y - x + 2k\pi)^2}{2(t_j - t_i)}} - e^{-\frac{n(y + x + 2k\pi)^2}{2(t_j - t_i)}} \right).
\end{equation}
\end{enumerate}
\end{prop}

The starting point for the proof of Proposition \ref{prop:op_formulas} is the formulas \eqref{eq:reflecting_tp} and \eqref{eq:absorbing_tp} for the transition probability for Brownian motion between a pair of walls at $0$ and $\pi$. Let $P^{\reflect}(x, y; t)$ and $P^{\absorb}(x, y; t)$ denote the the transition probabilities \eqref{eq:reflecting_tp} and \eqref{eq:absorbing_tp}, respectively, with diffusion parameter $\sg=n^{-1/2}$: 
\begin{align}
  P^{\reflect}(x, y; t) = {}& \frac{\sqrt{n}}{\sqrt{2\pi t} } \sum^{\infty}_{k = -\infty} e^{-\frac{n(y - x + 2k\pi)^2}{2t}} + e^{-\frac{n(y + x + 2k\pi)^2}{2t }}, \\
  P^{\absorb}(x, y; t) = {}& \frac{\sqrt{n}}{\sqrt{2\pi t} } \sum^{\infty}_{k = -\infty} e^{-\frac{n(y - x + 2k\pi)^2}{2t}} - e^{-\frac{n(y + x + 2k\pi)^2}{2t}}.
\end{align}

Fix $n$ starting points $A = \{ a_1, \dotsc, a_n \}$ where $0 < a_1 < \dotsb < a_n < \pi$, and $n$ ending points $B = \{ b_1, \dotsc, b_n \}$ where $0 < b_1 < \dotsb < b_n < \pi$. Consider the transition probability density for $n$ particles $X_1(t), \dotsc, X_n(t)$ in Brownian motion between the pair of (reflecting or absorbing) walls at $0$ and $\pi$, such that they start at positions $A_n$, end at positions $B_n$ after time $t > 0$, and the paths of the particles do not intersect during the time. 
By the Karlin--McGregor theorem \cite{Karlin-McGregor59} this transition probability density is given by the determinant
\begin{equation}
  P^{\xxx}(A; B; n; t) = \det \left( P^{\xxx}(a_i; b_j; n; t) \right)^n_{i, j = 1}, \quad \xxx = \text{$\reflect$ or $\absorb$}.
\end{equation}
Now consider the model of $n$ particles $X_1(t), \dotsc, X_n(t)$ in Brownian motion with a pair of (reflecting or absorbing) walls at $0$ and $\pi$, starting at $A^{(0)} = \{ a^{(0)}_1, \dotsc, a^{(0)}_n \}$ at time $t_0 = 0$ and conditioned so that:
\begin{enumerate}
\item They end at $A^{(m + 1)} = \{ a^{(m + 1)}_1 , \dotsc, a^{(m + 1)}_n \}$ at time $t_{m + 1} =  T > 0$, and
\item Their paths do not intersect during the time $[0, T]$. 
\end{enumerate}
Then the joint probability density for the location of the particles that at the times $t_1 < t_2 < \dotsb < t_m$ in $(0, T)$ such that
\begin{equation}
  X_i(t_j) = a^{(j)}_i, \quad \text{where $0 < a^{(j)}_1  < a^{(j)}_2 < \dotsb < a^{(j)}_n < \pi$ for all $j = 0, 1, \dotsc, m, m + 1$},
\end{equation}
is given by
\begin{multline}
  P^{\xxx}(A^{(1)}; \dotsc; A^{(m)}; n; t_1, \dotsc, t_m) = \\
  P^{\xxx}(A^{(0)}; A^{(m + 1)}; n; T)^{-1} \prod^{m + 1}_{j = 1} P^{\xxx}(A^{(j - 1)}; A^{(j)}; n; t_{j + 1} - t_j),
\end{multline}
where $A^{(j)} = \{ a_1^{(j)}, \dotsc, a_n^{(j)} \}$.
In order to arrive at the models $\NIBMref$ and $\NIBMabs$ we must consider the degenerate case that $a^{(0)}_i \to 0_+$ and $a^{(m + 1)} \to 0_+$ for all $i = 1, \dotsc, n$. 

\subsection{Proof of Proposition \ref{prop:op_formulas}\ref{enu:prop_finite_n_ref}}\label{subsec_reflecting_wbc}

Note that $P^{\reflect}(a, b; t)$ is an even function of both $a$ and $b$. By \lHopital's rule, as $a^{(0)}_1, \dotsc, a^{(0)}_n \to 0$ we have
\begin{multline} \label{eq:ref_wall_first_trans}
  P^{\reflect}(A^{(0)}; A^{(1)}; n; t_1 - t_0) = \frac{\prod_{1 \leq i < j \leq n} \left( (a^{(0)}_j)^2 - (a^{(0)}_i)^2 \right)}{\prod^{n - 1}_{j = 0} (2j)!} \\
  \times \det \left. \left( \frac{d^{2(i- 1)}}{dx^{2(i- 1)}} P^{\reflect}(x; a^{(1)}_j; t_1 - t_0) \right\rvert_{x = 0} \right)^n_{i, j = 1} (1 + \bigO(\max \lvert a^{(0)}_j \rvert));
\end{multline}
as $a^{(m+1)}_1, \dotsc, a^{(m+1)}_n \to 0$,
\begin{multline} \label{eq:ref_wall_last_trans}
  P^{\reflect}(A^{(m)}; A^{(m + 1)}; n; t_{m + 1} - t_m) = \frac{\prod_{1 \leq i < j \leq n} \left( (a^{(m + 1)}_j)^2 - (a^{(m + 1)}_i)^2 \right)}{\prod^{n - 1}_{j = 0} (2j)!} \\
  \times \det \left. \left( \frac{d^{2(i- 1)}}{dy^{2(i - 1)}} P^{\reflect}(a^{(m)}_j; y; t_{m + 1} - t_m) \right\rvert_{y = 0} \right)^n_{i, j = 1} (1 + \bigO(\max \lvert a^{(m + 1)}_j \rvert));
\end{multline}
and as both $a^{(0)}_1, \dotsc, a^{(0)}_n \to 0$ and $a^{(m+1)}_1, \dotsc, a^{(m+1)}_n \to 0$,
\begin{multline} \label{eq:ref_wall_total_trans}
  P^{\reflect}(A^{(0)}; A^{(m + 1)}; n; T) = \frac{\prod_{1 \leq i < j \leq n} \left( (a^{(0)}_j)^2 - (a^{(0)}_i)^2 \right) \left( (a^{(m + 1)}_j)^2 - (a^{(m + 1)}_i)^2 \right)}{\prod^{n - 1}_{j = 0} ((2j)!)^2} \\
  \times R^{\reflect}_n(T) (1 + \bigO(\max_{j = 1, \dotsc, n} (\lvert a^{(0)}_j \rvert, \lvert a^{(m + 1)}_j) \rvert)),
\end{multline}
where
\begin{equation} \label{eq:R^ref_defn}
  \begin{split}
    R^{\reflect}_n(T) = {}& \det \left( \left. \frac{d^{2(j + k - 2)}}{dx^{2(j - 1)} y^{2(k - 1)}} P^{\reflect}(x; y; T) \right\rvert_{x = y = 0} \right)^n_{j, k = 1} \\
    = {}& \det \left( \left. \frac{d^{2(j + k - 2)}}{dx^{2(j + k - 2)}} P^{\reflect}(x; 0; T) \right\rvert_{x = 0} \right)^n_{j, k = 1} \\
    = {}& \det \left( \left. \frac{d^{2(j + k - 2)}}{dz^{2(j + k - 2)}} \left( e^{-\frac{n z^2}{2T}} \vartheta_3 \left( \frac{i\pi n z}{T}, e^{-\frac{2\pi^2 n}{T}} \right) \right) \right\rvert_{z = 0} \right)^n_{j, k = 1},
  \end{split}
\end{equation}
where $\vartheta_3$ is the Jacobi theta function
\begin{equation}
  \vartheta_3(z; q) = \sum^{\infty}_{k = -\infty} e^{2ki z} q^{k^2}.
\end{equation}
Thus in the degenerate case, the joint probability density function of the particles at times $t_1 < \dotsb < t_m \in (0, T)$ is
\begin{multline} \label{eq:jpdf_ref}
  P^{\reflect}(A^{(1)}; \dotsc; A^{(m)}; n; t_1, \dotsc, t_m) = \frac{1}{R^{\reflect}_n(T)} \det \left( \widetilde{\phi}^{\reflect}_j(a^{(1)}_k) \right)^n_{j, k = 1} \det \left( \widetilde{\psi}^{\reflect}_j(a^{(m)}_k) \right)^n_{j, k = 1} \\
  \times \prod^{m - 1}_{j = 1} P^{\reflect}(A^{(j - 1)}; A^{(j)}; n; t_{j + 1} - t_j),
\end{multline}
where
\begin{equation}
\widetilde{  \phi}^{\reflect}_j(a) = \left. \frac{d^{2(j - 1)}}{dx^{2(j - 1)}} P^{\reflect}(x, a; t_1) \right\rvert_{x = 0}, \quad \widetilde{\psi}^{\reflect}_j(a) = \left. \frac{d^{2(j - 1)}}{dy^{2(j - 1)}} P^{\reflect}(a, y; T - t_m) \right\rvert_{y = 0}.
\end{equation}

By the Poisson summation formula, we have
\begin{equation} \label{eq:Fourier_of_P(x;y;t)}
  \begin{split}
    P^{\reflect}(y; x; t) = P^{\reflect}(x; y; t) = {}& \frac{\sqrt{n}}{\sqrt{2\pi t}} \sum^{\infty}_{k = -\infty} e^{-\frac{n(y - x + 2k\pi)^2}{2t}} + e^{-\frac{n(y + x + 2k\pi)^2}{2t}} \\
    = {}& \frac{\sqrt{n}}{\sqrt{2\pi t}} \sum^{\infty}_{k = -\infty} \frac{\sqrt{t}}{\sqrt{2\pi n}} \left( e^{ikx - \frac{tk^2}{2n}} + e^{-ikx - \frac{tk^2}{2n}} \right) e^{iky} \\
    = {}& \frac{1}{\pi} \sum_{k \in \intZ} e^{ikx - \frac{tk^2}{2n}} \cos(ky) = \frac{1}{\pi} + \frac{2}{\pi} \sum^{\infty}_{k = 1} \cos(kx) e^{-\frac{tk^2}{2n}} \cos(ky) \\
    = {}& \frac{1}{\pi} \sum_{s \in L_n} e^{insx - \frac{ts^2 n}{2}} \cos(nsy),
  \end{split}
\end{equation}
where $L_n$ is the lattice with mesh $1/n$ defined in \eqref{def:lattice}.
Then it is straightforward to calculate that for all $j = 0, 1, 2, \dotsc$, 
\begin{equation} \label{eq:derivative_of_P^ref}
  \left. \frac{d^{2j}}{dx^{2j}} P^{\reflect}(x; a; t) \right\rvert_{x = 0} = \left. \frac{d^{2j}}{dy^{2j}} P^{\reflect}(a; y; t) \right\rvert_{y = 0} = \frac{(n i)^{2j}}{\pi} \sum_{s \in L_n} s^{2j} e^{- \frac{n ts^2}{2}} \cos(nsa).
\end{equation}
Hence, letting $p^T_{n, 2k}(x)$ be the even monic polynomials of degree $2k$ defined in \eqref{eq:defn_disc_Gaussian_poly}, we have
\begin{align}
  \det \left( \widetilde{\phi}^{\reflect}_j(a^{(1)}_k) \right)^n_{j, k = 1} = {}& \frac{(n i)^{n(n - 1)}}{\pi^n} \det \left( \phi^{\reflect}_j(a^{(1)}_k) \right)^n_{j, k = 1}, \\
  \det \left( \widetilde{\psi}^{\reflect}_j(a^{(m)}_k) \right)^n_{j, k = 1} = {}& \frac{(n i)^{n(n - 1)}}{\pi^n} \det \left( \psi^{\reflect}_j(a^{(m)}_k) \right)^n_{j, k = 1}, 
\end{align}
where
\begin{align}
  \phi^{\reflect}_j(x) = {}& \sum_{s \in L_n} p^T_{n, 2(j - 1)}(s) e^{-\frac{n t_1 s^2}{2}} \cos(nsx) \notag \\
  = {}& p^T_{n, 2(j - 1)}(0) + 2 \sum^{\infty}_{k = 1} p^T_{n, 2(j - 1)} \left( \frac{k}{n} \right) e^{-\frac{t_1 k^2}{2n}} \cos(kx), \label{eq:phi_ref} 
\end{align}
\begin{align}
  \psi^{\reflect}_j(x) = {}& \sum_{s \in L_n} p^T_{n, 2(j - 1)}(s) e^{-\frac{n (T - t_m) s^2}{2}} \cos(nsx) \notag \\
  = {}& p^T_{n, 2(j - 1)}(0) + 2 \sum^{\infty}_{k = 1} p^T_{n, 2(j - 1)} \left( \frac{k}{n} \right) e^{-\frac{(T - t_m) k^2}{2n}} \cos(kx), \label{eq:psi_ref}
\end{align}
and
\begin{equation} \label{eq:R^ref_in_p}
  \begin{split}
    R^{\reflect}_n(T) = {}& \frac{(ni)^{2n(n - 1)}}{\pi^n} \det \left( \sum_{s \in L_n} s^{2(j - 1)} s^{2(k - 1)} e^{-\frac{nT s^2}{2}} \right)^n_{j, k = 1} \\
    = {}& \frac{n^{n(2n - 1)}}{\pi^n} \det \left( \frac{1}{n} \sum_{s \in L_n} p^T_{n, 2(j - 1)}(s) p^T_{n, 2(k - 1)}(s) e^{-\frac{nT s^2}{2}} \right)^n_{j, k = 1}.
  \end{split}
\end{equation}
By the orthogonality we then obtain
\begin{equation} \label{eq:R^ref_orthogonality}
  R^{\reflect}_n(T) = \frac{n^{n(2n - 1)}}{\pi^n} \prod^{n - 1}_{j = 0} h^T_{n, 2j}.
\end{equation}

We now use the Eynard--Mehta theorem \cite{Eynard-Mehta98} to derive the correlation kernel of the determinantal process. We follow the notational conventions in \cite{Liechty-Wang14-2}. Since we want to reuse the argument for the absorbing wall case, we use notation $\xxx$ throughout the derivation, and read $\xxx = \reflect$ here.

Define (with $\xxx = \reflect$) for $j = 1, \dotsc, n$, $\phi^{\xxx}_j(x)$ and $\psi^{\xxx}_j(x)$ as in \eqref{eq:phi_ref} and \eqref{eq:psi_ref}, and for $k = 1, \dotsc, m - 1$,
\begin{equation} \label{eq:W_l(x,y)_kernel}
  W^{\xxx}_k(x, y) := P^{\xxx}(x; y; t_{k + 1} - t_k).
\end{equation}
Then we define the operator $\Phi^{\xxx}: L^2[0, \pi] \to \ell^2(n)$ as
\begin{equation} \label{eq:defn_Phi}
  \Phi^{\xxx}(f(\theta)) = \left( \int^{\pi}_0 f(\theta) \phi^{\xxx}_1(\theta) d\theta, \dotsc, \int^{\pi}_0 f(\theta) \phi^{\xxx}_n(\theta) d\theta \right)^T,
\end{equation}
the operator $\Psi^{\xxx}: \ell^2(n) \to L^2[0, \pi]$ as
\begin{equation} \label{eq:defn_Psi}
  \Psi^{\xxx}((v_1, \dotsc, v_n)^T) = \sum^n_{k = 1} v^{\xxx}_k \psi_k(\theta),
\end{equation}
and define the integral operator $W^{\xxx}_k: L^2[0, \pi] \to L^2[0, \pi]$ by the kernel function \eqref{eq:W_l(x,y)_kernel}. Furthermore, we define the operators
\begin{equation} \label{eq:defn_W_ij_and_circ_W_ij}
  W^{\xxx}_{[i, j)} :=
  \begin{cases}
    W^{\xxx}_i \dotsm W^{\xxx}_{j - 1} & \text{for $i < j$}, \\
    1 & \text{for $i = j$}, \\
    0 & \text{for $i > j$},
  \end{cases}
  \quad \text{and} \quad
  \W^{\xxx}_{[i, j)} :=
  \begin{cases}
    W^{\xxx}_i \dotsm W^{\xxx}_{j - 1} & \text{for $i < j$}, \\
    0 & \text{for $i \geq j$}.
  \end{cases}
\end{equation}
We also define the operator $M^{\xxx}: \ell^2(n) \to \ell^2(n)$ as
\begin{equation} \label{eq:defn_M_compact}
  M^{\xxx} = \Phi^{\xxx} W^{\xxx}_{[1, m)} \Psi^{\xxx},
\end{equation}
which is represented by the $n \times n$ matrix
\begin{equation} \label{eq:defn_entries_M}
  M^{\xxx}_{ij} = \int^{\pi}_0 \dots \int^{\pi}_0 \phi^{\xxx}_i(\theta_1) W^{\xxx}_1(\theta_1, \theta_2) \dotsm W^{\xxx}_{m - 1}(\theta_{m - 1}, \theta_m) \psi^{\xxx}_j(\theta_m) d\theta_1 \dotsm d\theta_m.
\end{equation}
Then by the Eynard--Mehta theorem the correlation kernel is given  as
\begin{equation} \label{eq:general_formula_E_M}
  K^{\xxx}_{t_i, t_j}(x, y; n, T) = \widetilde{K}^{\xxx}_{t_i, t_j}(x, y; n, T) - \W^{\xxx}_{[i, j)}(x, y), 
\end{equation}
where
\begin{equation} \label{eq:gneral_formula_E_M_auxiliary}
  \widetilde{K}^{\xxx}_{t_i, t_j}(x, y; n, T) = \left( W^{\xxx}_{[i, m)} \Psi (M^{\xxx})^{-1} \Phi^{\xxx} W^{\xxx}_{[1, j)} \right)(x, y).
\end{equation}

Up to here we have given a general construction of the correlation kernel. In order to find an explicit expression, we first define the Fourier coefficients of a function $f \in L^2[0, \pi]$ over the basis $1, \cos x, \cos 2x, \dotsc$, as
\begin{equation} \label{eq:Fourer_coeff_even}
  f(x) = \frac{1}{2} \hat{f}(0) + \sum^{\infty}_{k = 1} \hat{f}(k) \cos(kx), \quad \text{where} \quad \hat{f}(k) = \frac{2}{\pi} \int^{\pi}_0 f(x) \cos(kx) dx.
\end{equation}
The Fourier expansion of $W^{\reflect}_j(x, y)$ with respect to $y$ is given in \eqref{eq:Fourier_of_P(x;y;t)}. Then we have
\begin{equation}
  (W^{\reflect}_jf)(x) = \int^{\pi}_0 W^{\reflect}_j(x, y) f(y) dy = \frac{1}{2} \hat{f}(0) + \sum^{\infty}_{k = 1} e^{-\frac{(t_{j + 1} - t_j)k^2}{2n}} \hat{f}(k) \cos(kx),
\end{equation}
or equivalently,
\begin{equation} \label{eq:W_j_in_dual_space}
  \widehat{W^{\reflect}_j f}(k) = e^{-\frac{(t_{j + 1} - t_j)k^2}{2n}} \hat{f}(k), \quad \text{for all $k = 0, 1, 2, \dotsc$.}
\end{equation}
By the definition of $W^{\reflect}_{[i, j)}$ and using \eqref{eq:W_j_in_dual_space} successively, we have
\begin{equation} \label{eq:W_ij_in_dual_space}
  \widehat{W^{\reflect}_{[i, j)} f}(k) = e^{-\frac{(t_j - t_i)k^2}{2n}} \hat{f}(k), \quad \text{for all $k = 0, 1, 2, \dotsc$.}
\end{equation}
Hence by the Poisson summation formula, for $i < j$,
\begin{equation} \label{eq:W_circ_and_W}
  \begin{split}
    \W^{\reflect}_{[i, j)}(x, y) = W^{\reflect}_{[i, j)}(x, y) = {}& \frac{1}{\pi} + \frac{2}{\pi} \sum^{\infty}_{k = 1} \cos(kx) e^{-\frac{(t_j - t_i)k^2}{2n}} \cos(ky) \\
    = {}& \frac{1}{2\pi} \sum^{\infty}_{k = -\infty} \left( e^{-ikx - \frac{(t_j - t_i)k^2}{2n}} + e^{ikx - \frac{(t_j - t_i)k^2}{2n}} \right) e^{iky} \\
    = {}& \frac{\sqrt{n}}{\sqrt{2\pi t}} \sum^{\infty}_{k = -\infty} \left( e^{-\frac{n(y - x + 2k\pi)^2}{2(t_j - t_i)}} + e^{-\frac{n(y + x + 2k\pi)^2}{2(t_j - t_i)}} \right).
  \end{split}
\end{equation}
As a specialization of \eqref{eq:W_ij_in_dual_space}, consider $W^{\reflect}_{[i, m)} \Psi^{\reflect}$, an operator from $\ell^2(n)$ to $L^2[0, \pi]$. It is represented by an $n$ dimensional row vector, whose $l$-th component has the Fourier coefficients
\begin{equation}
  \widehat{(W^{\reflect}_{[i, m)} \Psi)_l}(k) = 2 p_{2(l - 1)}\left( \frac{k}{n} \right) e^{-\frac{(T - t_i)k^2}{2n}},
\end{equation}
and then we have the formula for the $l$-th component
\begin{equation} \label{eq:formula_WPsi}
  \begin{split}
    (W^{\reflect}_{[i, m)} \Psi^{\reflect})_l(x) = {}& \int^{\pi}_0 W^{\reflect}_{[i, m)}(x, y) \psi^{\reflect}_l(y) dy \\
    = {}& p_{2(l- 1)}(0) + 2 \sum^{\infty}_{k = 1} p_{2(l- 1)} \left( \frac{k}{n} \right) e^{-\frac{(T - t_i) k^2}{2n}} \cos(kx) \\
    = {}& \sum_{s \in L_n} p_{2(l - 1)}(s) e^{-\frac{n(T - t_i)s^2}{2}} e^{insx}.
  \end{split}
\end{equation}
Similarly, $\Phi^{\reflect} W^{\reflect}_{[1, j)}$ is an operator from $L^2[0, \pi]$ to $\ell^2(n)$, and is represented by an $n$ dimensional column vector. Its $l$-th component is
\begin{equation} \label{eq:formula_PhiW}
  \begin{split}
    (\Phi^{\reflect} W^{\reflect}_{[1, j)})(x) = {}& \int^{\pi}_0 \phi^{\reflect}_l(y) W^{\reflect}_{[1, j)}(y, x) dy = \int^{\pi}_0 W^{\reflect}_{[1, j)}(x, y) \phi^{\reflect}_l(y) dy \\
    = {}& p_{2(l- 1)}(0) + 2 \sum^{\infty}_{k = 1} p_{2(l- 1)} \left( \frac{k}{n} \right) e^{-\frac{t_j k^2}{2n}} \cos(kx) \\
    = {}& \sum_{s \in L_n} p_{2(l - 1)}(s) e^{-\frac{n t_j s^2}{2}} e^{-insx}.
  \end{split}
\end{equation}
Hence the $(i, j)$ entry of the matrix $M^{\reflect}$ defined in \eqref{eq:defn_entries_M} is
\begin{equation}
  \begin{split}
    M^{\reflect}_{ij} = {}& \int^{\pi}_0 \int^{\pi}_0 \phi^{\reflect}_i(x) W^{\reflect}_{[1, m)}(x, y) \psi^{\reflect}_j(y) dy dx \\
    = {}& \int^{\pi}_0 \phi^{\reflect}_i(x) (W^{\reflect}_{[1, m)}\psi_j)(x) dx \\
    = {}& \int^{\pi}_0 \left( p_{2(i - 1)}(0) + 2 \sum^{\infty}_{k = 1} p_{2(i - 1)} \left( \frac{k}{n} \right) e^{-\frac{t_1 k^2}{2n}} \cos(kx) \right) \\
    & \phantom{\int^{\pi}_0} \times \left( p_{2(j - 1)}(0) + 2 \sum^{\infty}_{k = 1} p_{2(j - 1)} \left( \frac{k}{n} \right) e^{-\frac{(T - t_m) k^2}{2n}} \cos(kx) \right) dx.
  \end{split}
\end{equation}
By the orthogonality of the basis $1, \cos x, \cos 2x, \dotsc$, we have
\begin{equation}
  \begin{split}
    M^{\reflect}_{ij} = {}& \pi \left( p_{2(i - 1)}(0) p_{2(j - 1)}(0) + 2 \sum^{\infty}_{k = 1} p_{2(i - 1)} \left( \frac{k}{n} \right) p_{2(j - 1)} \left( \frac{k}{n} \right) e^{-\frac{T k^2}{2n}} \right) \\
    = {}& \pi \sum_{s \in L_n} p_{2(i - 1)}(s) p_{2(j - 1)}(s) e^{-\frac{nT s^2}{2}}.
  \end{split}
\end{equation}
By the orthogonality of $p_{2(i - 1)}$ and $p_{2(j - 1)}$ defined in \eqref{eq:defn_disc_Gaussian_poly}, we have
\begin{equation} \label{eq:entries_M}
  M^{\reflect}_{ij} =
  \begin{cases}
    0 & \text{if $i \neq j$}, \\
    \pi n h^T_{n, 2(i - 1)} & \text{if $i = j$}.
  \end{cases}
\end{equation}
Thus by plugging \eqref{eq:formula_WPsi}, \eqref{eq:formula_PhiW}, and \eqref{eq:entries_M} into \eqref{eq:general_formula_E_M} and \eqref{eq:gneral_formula_E_M_auxiliary}, we obtain Proposition \ref{prop:op_formulas}\ref{enu:prop_finite_n_ref}.


\subsection{Proof of Proposition \ref{prop:op_formulas}\ref{enu:prop_finite_n_abs}}

The derivation of the correlation kernel in the case of absorbing walls is very similar. Since $P^{\absorb}(a, b; t)$ is an odd function of both $a$ and $b$ we have, analogous to \eqref{eq:ref_wall_first_trans}, \eqref{eq:ref_wall_last_trans}, and \eqref{eq:ref_wall_total_trans}, as $a^{(0)}_1, \dotsc, a^{(0)}_n \to 0$,
\begin{multline} \label{eq:abs_wall_first_trans}
  P^{\absorb}(A^{(0)}; A^{(1)}; n; t_1 - t_0) = \frac{\prod_{1 \leq i < j \leq n} \left( (a^{(0)}_j)^2 - (a^{(0)}_i)^2 \right)}{\prod^{n - 1}_{j = 0} (2j + 1)!} \prod^n_{j = 1} a^{(0)}_j \\
  \times \det \left. \left( \frac{d^{2i- 1}}{dx^{2i- 1}} P^{\absorb}(x; a^{(1)}_j; t_1 - t_0) \right\rvert_{x = 0} \right)^n_{i, j = 1} (1 + \bigO(\max \lvert a^{(0)}_j \rvert));
\end{multline}
as $a^{(m+1)}_1, \dotsc, a^{(m+1)}_n \to 0$,
\begin{multline} \label{eq:abs_wall_last_trans}
  P^{\absorb}(A^{(m)}; A^{(m + 1)}; n; t_{m + 1} - t_m) = \frac{\prod_{1 \leq i < j \leq n} \left( (a^{(m + 1)}_j)^2 - (a^{(m + 1)}_i)^2 \right)}{\prod^{n - 1}_{j = 0} (2j + 1)!} \prod^n_{j = 1} a^{(m + 1)}_j \\
  \times \det \left. \left( \frac{d^{2i- 1}}{dy^{2i - 1}} P^{\absorb}(a^{(m)}_j; y; t_{m + 1} - t_m) \right\rvert_{y = 0} \right)^n_{i, j = 1} (1 + \bigO(\max \lvert a^{(m + 1)}_j \rvert));
\end{multline}
and as both $a^{(0)}_1, \dotsc, a^{(0)}_n \to 0$ and $a^{(m+1)}_1, \dotsc, a^{(m+1)}_n \to 0$,
\begin{multline} \label{eq:abs_wall_total_trans}
  P^{\absorb}(A^{(0)}; A^{(m + 1)}; n; T) = \frac{\prod_{1 \leq i < j \leq n} \left( (a^{(0)}_j)^2 - (a^{(0)}_i)^2 \right) \left( (a^{(m + 1)}_j)^2 - (a^{(m + 1)}_i)^2 \right)}{\prod^{n - 1}_{j = 0} ((2j + 1)!)^2} \prod^n_{j = 1} a^{(0)}_j a^{(m + 1)}_j \\
  \times R^{\absorb}_n(T) (1 + \bigO(\max_{j = 1, \dotsc, n} (\lvert a^{(0)}_j \rvert, \lvert a^{(m + 1)}_j) \rvert)),
\end{multline}
where $R^{\absorb}$ is defined like $R^{\reflect}$ in \eqref{eq:R^ref_defn},
\begin{equation} \label{eq:R^abs_defn}
  \begin{split}
    R^{\absorb}_n(T) = {}& \det \left( \left. \frac{d^{2(j + k - 1)}}{dx^{2j - 1} y^{2k - 1}} P^{\absorb}(x; y; T) \right\rvert_{x = y = 0} \right)^n_{j, k = 1} \\
    = {}& \det \left( \left. \frac{d^{2(j + k - 1)}}{dx^{2(j + k - 1)}} P^{\reflect}(x; 0; T) \right\rvert_{x = 0} \right)^n_{j, k = 1} \\
    = {}& \det \left( \left. \frac{d^{2(j + k - 1)}}{dz^{2(j + k - 1)}} \left( e^{-\frac{n z^2}{2T}} \vartheta_3 \left( \frac{i\pi n z}{T}, e^{-\frac{2\pi^2 n}{T}} \right) \right) \right\rvert_{z = 0} \right)^n_{j, k = 1}.
  \end{split}
\end{equation}
Note that the $P^{\reflect}$ in  \eqref{eq:R^abs_defn} is not a typo.
Thus in the degenerate case, the joint probability density function of the particles at times $t_1 < \dotsb < t_m \in (0, T)$ is, analogous to \eqref{eq:jpdf_ref},
\begin{multline} \label{eq:jpdf_abs}
  P^{\absorb}(A^{(1)}; \dotsc; A^{(m)}; t_1, \dotsc, t_m) = \frac{1}{R^{\absorb}_n(T)} \det \left(\widetilde{ \phi}^{\absorb}_j(a^{(1)}_k) \right)^n_{j, k = 1} \det \left( \widetilde{\psi}^{\absorb}_j(a^{(m)}_k) \right)^n_{j, k = 1} \\
  \times \prod^m_{j = 2} P^{\absorb}(A^{(j - 1)}; A^{(j)}; t_{j + 1} - t_j),
\end{multline}
where
\begin{equation}
  \widetilde{\phi}^{\absorb}_j(a) = \left. \frac{d^{2j - 1}}{dx^{2j - 1}} P^{\absorb}(x; a; t_1) \right\rvert_{x = 0}, \quad \widetilde{\psi}^{\absorb}_j(a) = \left. \frac{d^{2j - 1}}{dy^{2j - 1}} P^{\absorb}(a; y; T - t_m) \right\rvert_{y = 0}.
\end{equation}

It is now clear that the model \NIBMabs\ is a determinantal process like \NIBMref. We derive its correlation kernel parallel to the derivation for \NIBMref. Like \eqref{eq:Fourier_of_P(x;y;t)}, 
\begin{equation} \label{eq:Fourier_of_P(x;y;t)_abs}
  \begin{split}
    P^{\absorb}(y; x; t) = P^{\absorb}(x; y; t) = {}& \frac{\sqrt{n}}{\sqrt{2\pi t}} \sum^{\infty}_{k = -\infty} e^{-\frac{n(y - x + 2k\pi)^2}{2t}} - e^{-\frac{n(y + x + 2k\pi)^2}{2t}} \\
    = {}& \frac{1}{\pi} \sum_{s \in L_n} e^{insx - \frac{ts^2n}{2}} \sin(nsy).
  \end{split}
\end{equation}
Then like \eqref{eq:derivative_of_P^ref}, for all $j = 0, 1, 2, \dotsc$
\begin{equation}
  \left. \frac{d^{2j + 1}}{dx^{2j + 1}} P^{\absorb}(x; a; t) \right\rvert_{x = 0} = \left. \frac{d^{2j + 1}}{dy^{2j + 1}} , \quad P^{\absorb}(a; y; t) \right\rvert_{y = 0} = \frac{i(ni)^{2j + 1}}{\pi} \sum_{s \in L_n} s^{2j + 1} e^{-\frac{n ts^2}{2}} \sin(nsa).
\end{equation}
Hence, letting $p^T_{n, 2k + 1}(x)$ be the discrete Gaussian orthogonal polynomial of degree $2k + 1$, we have,
\begin{align}
  \det \left( \widetilde{\phi}^{\absorb}_j(a^{(1)}_k) \right)^n_{j, k = 1} = {}& \frac{i^n (n i)^{n^2}}{\pi^n} \det \left( \phi^{\absorb}_j(a^{(1)}_k) \right)^n_{j, k = 1}, \\
  \det \left( \widetilde{\psi}^{\absorb}_j(a^{(m)}_k) \right)^n_{j, k = 1} = {}& \frac{i^n (n i)^{n^2}}{\pi^n} \det \left( \psi^{\absorb}_j(a^{(m)}_k) \right)^n_{j, k = 1}, 
\end{align}
where
\begin{align}
  \phi^{\absorb}_j(x) = {}& \sum_{s \in L_n} p^T_{n, 2j - 1}(s) e^{-\frac{n t_1 s^2}{2}} \sin(nsx) = 2 \sum^{\infty}_{k = 1} p^T_{n, 2j - 1} \left( \frac{k}{n} \right) e^{-\frac{t_1 k^2}{2n}} \sin(kx), \label{eq:phi_abs} \\
  \psi^{\absorb}_j(x) = {}& \sum_{s \in L_n} p^T_{n, 2j - 1}(s) e^{-\frac{n (T - t_m) s^2}{2}} \sin(nsx) = 2 \sum^{\infty}_{k = 1} p^T_{n, 2j - 1} \left( \frac{k}{n} \right) e^{-\frac{(T - t_m) k^2}{2n}} \sin(kx), \label{eq:psi_abs}
\end{align}
and analogous to \eqref{eq:R^ref_in_p} and \eqref{eq:R^ref_orthogonality}, we use \eqref{eq:R^abs_defn} and \eqref{eq:Fourier_of_P(x;y;t)} (but not \eqref{eq:Fourier_of_P(x;y;t)_abs})
\begin{equation} \label{eq:R^abs_in_p}
  \begin{split}
    R^{\absorb}_n(T) = {}& \frac{(ni)^{2n^2}}{\pi^n} \det \left( \sum_{s \in L_n} s^{2j - 1} s^{2k - 1} e^{-\frac{nT s^2}{2}} \right)^n_{j, k = 1} \\
    = {}& \frac{(-1)^n n^{n(2n + 1)}}{\pi^n} \det \left( \frac{1}{n} \sum_{s \in L_n} p^T_{n, 2j - 1}(s) p^T_{n, 2k - 1}(s) e^{-\frac{nT s^2}{2}} \right)^n_{j, k = 1} \\
    = {}& \frac{(-1)^n n^{n(2n + 1)}}{\pi^n} \prod^{n - 1}_{j = 0} h^T_{n, 2j + 1}.
  \end{split}
\end{equation}

Now we apply the Eynard--Mehta theorem. Define for $j = 1, \dotsc, n$, $\phi^{\absorb}_j(x)$ and $\psi^{\absorb}_j(x)$ as in \eqref{eq:phi_abs} and \eqref{eq:psi_abs}, and for $k = 1, \dotsc, m - 1$
\begin{equation} \label{eq:W_l(x,y)_kernel_abs}
  W^{\absorb}_k(x, y) := P^{\absorb}(x; y; t_{k + 1} - t_k).
\end{equation}
Then we define $\Phi^{\absorb}$, $\Psi^{\absorb}$, $W^{\absorb}_{[i, j)}$, $\W^{\absorb}_{[i, j)}$, and $M^{\absorb}$ by \eqref{eq:defn_Phi}, \eqref{eq:defn_Psi}, \eqref{eq:defn_W_ij_and_circ_W_ij}, and \eqref{eq:defn_M_compact} respectively, with $\xxx = \absorb$. All the arguments between \eqref{eq:W_l(x,y)_kernel} and \eqref{eq:gneral_formula_E_M_auxiliary} remain valid, and the correlation kernel is given by the Eynard--Mehta formula \eqref{eq:general_formula_E_M}, with $\xxx = \absorb$.

To find the explicit expression of the correlation kernel, we define the Fourier expansion of a function $g \in L^2[0, \pi]$ over the basis $\sin x, \sin 2x, \sin 3x, \dotsc$, as
\begin{equation}
  g(x) = \sum^{\infty}_{k = 1} \hat{g}(k) \sin(kx), \quad \text{where} \quad \hat{g}(k) = \frac{2}{\pi} \int^{\pi}_0 g(x) \sin(kx) dx.
\end{equation}
Note that the Fourier coefficients of an $L^2[0, \pi]$ function is defined differently from those in \eqref{eq:Fourer_coeff_even} in Section \ref{subsec_reflecting_wbc}, since the orthogonal basis is changed. The Fourier expansion of $W^{\absorb}_j(x, y)$ with respect to $y$ is given in \eqref{eq:Fourier_of_P(x;y;t)_abs}. Then we have
\begin{equation}
  (W^{\absorb}_j g)(x) = \int^{\pi}_0 W^{\absorb}_j(x, y) g(y) dy = \sum^{\infty}_{k = 1} e^{=\frac{(t_{j + 1} - t_j)k^2}{2n}} \hat{g}(k) \sin(kx),
\end{equation}
or equivalently
\begin{equation} \label{eq:W_j_in_dual_space_abs}
  \widehat{W^{\absorb}_j g}(k) = e^{=\frac{(t_{j + 1} - t_j)k^2}{2n}} \hat{g}(k), \quad \text{for all $k = 1, 2, 3, \dotsc$.}
\end{equation}
By the definition of $W^{\absorb}_{[i, j)}$ in \eqref{eq:defn_W_ij_and_circ_W_ij} and using \eqref{eq:W_j_in_dual_space_abs} successively, we have
\begin{equation} \label{eq:W_ij_in_dual_space_abs}
  \widehat{W^{\absorb}_{[i, j)} g}(k) = e^{-\frac{(t_j - t_i)k^2}{2n}} \hat{g}(k), \quad \text{for all $k = 1, 2, 3, \dotsc$.}
\end{equation}
Hence by the Poisson summation formula, for $i < j$,
\begin{equation} \label{eq:W_circ_and_W_abs}
  \begin{split}
    \W^{\absorb}_{[i, j)}(x, y) = W^{\absorb}_{[i, j)}(x, y) = {}& \frac{2}{\pi} \sum^{\infty}_{k = 1} \sin(kx) e^{-\frac{(t_j - t_i)k^2}{2n}} \sin(ky) \\
    = {}& \frac{1}{2\pi} \sum^{\infty}_{k = -\infty} \left( e^{-ikx - \frac{(t_j - t_i)k^2}{2n}} - e^{ikx - \frac{(t_j - t_i)k^2}{2n}} \right) e^{iky} \\
    = {}& \frac{\sqrt{n}}{\sqrt{2\pi t}} \sum^{\infty}_{k = -\infty} \left( e^{-\frac{n(y - x + 2k\pi)^2}{2(t_j - t_i)}} - e^{-\frac{n(y + x + 2k\pi)^2}{2(t_j - t_i)}} \right).
  \end{split}
\end{equation}
As a specialization of \eqref{eq:W_ij_in_dual_space_abs}, consider $W^{\absorb}_{[i, m)} \Psi^{\absorb}$, an operator from $\ell^2(n)$ to $L^2[0, \pi]$. It is represented by an $n$ dimensional row vector, whose $l$-th component has the Fourier coefficients
\begin{equation}
  \widehat{(W^{\absorb}_{[i, m)} \Psi^{\absorb})_l}(k) = 2 p^T_{n, 2l - 1}\left( \frac{k}{n} \right) e^{-\frac{(T - t_i)k^2}{2n}},
\end{equation}
and then we have the formula for the $l$-th component
\begin{equation} \label{eq:formula_WPsi_abs}
  \begin{split}
    (W^{\absorb}_{[i, m)} \Psi^{\absorb})_l(x) = {}& \int^{\pi}_0 W^{\absorb}_{[i, m)}(x, y) \psi^{\absorb}_l(y) dy \\
    = {}& 2 \sum^{\infty}_{k = 1} p^T_{n, 2l - 1} \left( \frac{k}{n} \right) e^{-\frac{(T - t_i) k^2}{2n}} \sin(kx) \\
    = {}& -i \sum_{s \in L_n} p^T_{n, 2l - 1}(s) e^{-\frac{n(T - t_i)s^2}{2}} e^{insx}.
  \end{split}
\end{equation}
Similarly, $\Phi^{\absorb} W^{\absorb}_{[1, j)}$ is an operator from $L^2[0, \pi]$ to $\ell^2(n)$, and is represented by an $n$ dimensional column vector. Its $l$-th component is
\begin{equation} \label{eq:formula_PhiW_abs}
  \begin{split}
    (\Phi^{\absorb} W^{\absorb}_{[1, j)})(x) = {}& \int^{\pi}_0 \phi^{\absorb}_l(y) W^{\absorb}_{[1, j)}(y, x) dy = \int^{\pi}_0 W^{\absorb}_{[1, j)}(x, y) \phi^{\absorb}_l(y) dy \\
    = {}& 2 \sum^{\infty}_{k = 1} p^T_{n, 2l - 1} \left( \frac{k}{n} \right) e^{-\frac{t_j k^2}{2n}} \sin(kx) \\
    = {}& i\sum_{s \in L_n} p^T_{n, 2l - 1}(s) e^{-\frac{n t_j s^2}{2}} e^{-insx}.
  \end{split}
\end{equation}
Hence the $(i, j)$ entry of the matrix $M^{\absorb}$ defined in \eqref{eq:defn_entries_M} is
\begin{equation}
  \begin{split}
    M^{\absorb}_{ij} = {}& \int^{\pi}_0 \int^{\pi}_0 \phi^{\absorb}_i(x) W^{\absorb}_{[1, m)}(x, y) \psi^{\absorb}_j(y) dy dx \\
    = {}& \int^{\pi}_0 \phi^{\absorb}_i(x) (W^{\absorb}_{[1, m)}\psi^{\absorb}_j)(x) dx \\
    = {}& \int^{\pi}_0 \left( 2 \sum^{\infty}_{k = 1} p^T_{n, 2i - 1} \left( \frac{k}{n} \right) e^{-\frac{t_1 k^2}{2n}} \sin(kx) \right) \left( 2 \sum^{\infty}_{k = 1} p^T_{n, 2j - 1} \left( \frac{k}{n} \right) e^{-\frac{(T - t_m) k^2}{2n}} \sin(kx) \right) dx.
  \end{split}
\end{equation}
By the orthogonality of the basis $1, \cos x, \cos 2x, \dotsc$, we have
\begin{equation}
  M^{\absorb}_{ij} = \pi \left( 2 \sum^{\infty}_{k = 1} p^T_{n, 2i - 1} \left( \frac{k}{n} \right) p^T_{n, 2j - 1} \left( \frac{k}{n} \right) e^{-\frac{T k^2}{2n}} \right) = \pi \sum_{s \in L_n} p^T_{n, 2i - 1}(s) p^T_{n, 2j - 1}(s) e^{-\frac{nT s^2}{2}}.
\end{equation}
By the orthogonality of $p^T_{n, 2i - 1}$ and $p^T_{n, 2j - 1}$ defined in \eqref{eq:defn_disc_Gaussian_poly}, we have
\begin{equation} \label{eq:entries_M_abs}
  M^{\absorb}_{ij} =
  \begin{cases}
    0 & \text{if $i \neq j$}, \\
    \pi n h^T_{n, 2i - 1} & \text{if $i = j$}.
  \end{cases}
\end{equation}
Thus by plugging \eqref{eq:formula_WPsi_abs}, \eqref{eq:formula_PhiW_abs}, and \eqref{eq:entries_M_abs} into \eqref{eq:general_formula_E_M} and \eqref{eq:gneral_formula_E_M_auxiliary}, we obtain Proposition \ref{prop:op_formulas}\ref{enu:prop_finite_n_abs}.


\section{Relation to nonintersecting Brownian motions on the unit circle and proof of Theorem \ref{thm:main}} \label{sec:asymptotics}

The correlation kernels for \NIBMref\  and \NIBMabs\  given in Proposition \ref{prop:op_formulas} are closely related to the correlation kernel for the nonintersecting Brownian motions on the unit circle, which is studied in \cite{Liechty-Wang14-2} and denoted \NIBMT\ there. The model  \NIBMT\ is a model of $n$ nonintersecting Brownian motions on the unit circle, all of which start at the point $e^{i\cdot 0}$ at time $t=0$ and return to the same point at time $t=T$. For that model we have the following formulas for the correlation kernel.

\begin{prop} \cite[Formulas (117), (131) and (133)]{Liechty-Wang14-2} \label{enu:prop:circle_a}
  The \NIBMT\ with $n$ particles is a determinantal process, with the multi-time correlation kernel
  \begin{equation}
    K_{t_i, t_j}(x, y; n, T) = \widetilde{K}_{t_i, t_j}(x, y; n, T) - \W_{[i, j)}(x, y),
  \end{equation}
  where
  \begin{equation} \label{eq:W_circ_and_W_circle}
    \W_{[i, j)}(x, y) =
    \begin{cases}
      \frac{1}{2\pi} \sum_{s \in L_n} e^{-\frac{(t_j - t_i)ns^2}{2} - in(y - x)s} & \text{if $t_j > t_i$}, \\
      0 & \text{otherwise},
    \end{cases}
  \end{equation}
  and
  \begin{equation}
    \widetilde{K}_{t_i, t_j}(x, y; n, T) = \frac{n}{2\pi} \sum^{n - 1}_{k = 0} \frac{1}{h^T_{n, k}} S_{k, T - t_i}(x; n, T) S_{k, t_j}(-y; n, T),
  \end{equation}
  such that $h^T_{n, k}$ and $S_{k, a}(x; n, T)$ are defined in \eqref{eq:defn_disc_Gaussian_poly} and \eqref{eq:S_ka_defn1} respectively.
\end{prop}
We remark that in \cite{Liechty-Wang14-2}, the correlation kernel is stated with a phase factor $\tau$, and here we only need to $\tau = 0$ case. In \cite{Liechty-Wang14-2} it is proven that the kernel $K_{t_i, t_j}(x, y; n, T)$ converges to the Pearcey and tacnode kernels in certain scaling limits. This result is essential for our proof of Theorem \ref{thm:main} and we repeat it below.

\begin{prop} \label{prop:circle}
  The \NIBMT\ is a determinantal process. The multi-time correlation kernel, $K_{t_i, t_j}(x, y; n, T)$, has the following properties:
  \begin{enumerate}[label=(\alph*)]
  \item \cite[Theorem 1.3(a)]{Liechty-Wang14-2} \label{enu:prop:circle_b}
    Assume $T > \pi^2$, and let $n \to \infty$. The \NIBMT\ converges to the Pearcey process at the time around $t_c$ that is given in \eqref{eq:T_c_in_introduction} and position around $\pi$. To be precise, let $d$ be the positive constant given in \eqref{eq:d_defn}, and
    \begin{equation}
      t_i = 2t^c + \frac{d^2}{n^{1/2}} \tau_i, \quad t_j = 2t^c + \frac{d^2}{n^{1/2}} \tau_j, \quad x = \pi - \frac{d}{n^{3/4}} \xi, \quad y = \pi - \frac{d}{n^{3/4}} \eta,
    \end{equation}
    the correlation kernel has the limit
    \begin{equation}
      \lim_{n \to \infty} K_{t_i, t_j}(x, y; n, T) \left\lvert \frac{dy}{d\eta} \right\rvert = K^{\Pearcey}_{-\tau_j, -\tau_i}(\eta, \xi).
    \end{equation}
  \item \cite[Theorem 1.3(b)]{Liechty-Wang14-2} \label{enu:prop:circle_c}
    Assume $T$ is close to $\pi^2$ and let $n \to \infty$. The \NIBMT\ converges to the tacnode process at the time around $T/2$ and position around $\pi$. To be precise, let $\sigma \in \realR$, 
    \begin{equation}
     T = \pi^2(1 - 2^{-2/3} \sigma n^{-2/3}), \quad d = 2^{-5/3} \pi,
    \end{equation}
    and
    \begin{equation}
      t_i = \frac{T}{2} + \frac{d^2}{n^{1/3}} \tau_i, \quad t_j = \frac{T}{2} + \frac{d^2}{n^{1/3}} \tau_j, \quad x = \pi - \frac{d}{n^{2/3}} \xi, \quad y = \pi - \frac{d}{n^{2/3}} \eta.
    \end{equation}
    The correlation kernel has the limit
    \begin{equation}
      \lim_{n \to \infty} K_{t_i, t_j}(x, y; n, T) \left\lvert \frac{dy}{d\eta} \right\rvert = K^{\tac}_{\tau_i, \tau_j}(\xi, \eta; \sigma) = K^{\tac}_{-\tau_j, -\tau_i}(\eta, \xi; \sigma).
    \end{equation}
  \end{enumerate}
\end{prop}
Our Proposition \ref{prop:circle}\ref{enu:prop:circle_b} and \ref{enu:prop:circle_c} are stated slightly different from \cite[Theorem 1.3]{Liechty-Wang14-2}, since we make use of the periodicity of the \NIBMT. 
\begin{rmk}\label{rem:T_difference}
The process \NIBMT\ takes place on the circle of radius 1, or equivalently the interval $(-\pi, \pi]$ with periodic boundary conditions, whereas the processes \NIBMref\  and \NIBMabs\ are defined on the interval $[0,\pi]$. Therefore the critical time which separates the subcritical and supercritical regimes is $T=\pi^2$ for \NIBMT, compared with $T=T_c=\pi^2/2$ for \NIBMref\  and \NIBMabs. Similarly, the time $t^c$ at which the Pearcey process occurs in this paper differs from the time $t^c$ in \cite{Liechty-Wang14-2} by a factor of 2.
\end{rmk}

From the formulas presented in Propositions \ref{prop:op_formulas} and \ref{enu:prop:circle_a} we see that the kernels $K_{t_i, t_j}^{\reflect}(x,y; n, T)$ and $K_{t_i, t_j}^{\absorb}(x,y; n, T)$ are closely related to the kernel $K_{t_i, t_j}(x,y; n, T)$. We state this relation in the following lemma.
\begin{lem}\label{lem:kernels_relation}
The correlations kernels for the processes  \NIBMref\  and \NIBMabs\ are related to the correlation kernel for the process \NIBMT in the following explicit way.
\begin{align}
  K_{t_i, t_j}^{\reflect}(x,y; n, T) = {}& K_{2t_i, 2t_j}(x,y; 2n, 2T)+K_{2t_i, 2t_j}(x,-y; 2n, 2T) \notag \\
  = {}& K_{2t_i, 2t_j}(x,y; 2n, 2T)+K_{2t_i, 2t_j}(x, 2\pi - y; 2n, 2T), \label{eq:even_part} \\
  K_{t_i, t_j}^{\absorb}(x,y; n, T) = {}& K_{2t_i, 2t_j}(x,y; 2n, 2T)-K_{2t_i, 2t_j}(x,-y; 2n, 2T) \notag \\
  = {}& K_{2t_i, 2t_j}(x,y; 2n, 2T)-K_{2t_i, 2t_j}(x, 2\pi - y; 2n, 2T). \label{eq:odd_part}
\end{align}
\end{lem}
In other words, after rescaling time and the number of particles by a factor of 2, the kernels for \NIBMref\  and \NIBMabs\ are precisely the even and odd parts, respectively, of the kernel for \NIBMT.
Inserting the scalings described in parts \ref{enu:thm:main_a} and \ref{enu:thm:main_b} of Theorem \ref{thm:main} into both sides of \eqref{eq:even_part} and \eqref{eq:odd_part}, and applying the asymptotics given in Proposition \ref{prop:circle}, it is immediate to obtain the result of Theorem \ref{thm:main}, noting also that the Pearcey kernel is a symmetric function of its two spatial variables. Therefore the only thing which remains in the proof of Theorem \ref{thm:main} is to prove Lemma \ref{lem:kernels_relation}, which we do below.

\begin{proof}[Proof of Lemma \ref{lem:kernels_relation}]
Recall that the function $S_{k, a}(x; n, T)$ is defined in \eqref{eq:S_ka_defn1} by the Gaussian discrete orthogonal polynomials $p^T_{n, k}(x)$ and their inner product $h^T_{n, k}$, which are defined by the orthogonality \eqref{eq:defn_disc_Gaussian_poly}. From their definition, we have
\begin{equation} \label{eq:OPs_rescaling}
  p^T_{n, k}(x) = (-1)^k p^T_{n, k}(-x), \quad p_{n,k}^T(x)=2^k p_{2n, k}^{2T}(x/2), \quad \text{and} \quad h^T_{n,k}=2^{2k+1}h^{2T}_{2n, k},
\end{equation}
and then 
\begin{equation} \label{eq:S_k_rescaling}
  S_{k,a}(-x; n, T) =
  \begin{cases}
    S_{k,a}(x; n, T) & \text{if $k$ is even}, \\
    -S_{k,a}(x; n, T) & \text{if $k$ is odd},
  \end{cases}
  \quad \text{and} \quad S_{k,a}(x; n, T)= 2^{k+1}S_{k, 2a}(x; 2n, 2T).
\end{equation}
and then can rewrite \eqref{eq:tilde_K_full_expansion} and \eqref{eq:tilde_K_full_expansion_abs} into
\begin{align}
  \widetilde{K}_{t_i, t_j}^{\reflect}(x,y; n, T) = {}& \frac{2n}{\pi} \sum^{n - 1}_{k = 0} \frac{1}{ h^{2T}_{2n, 2k}} S_{2k, 2(T - t_i)}(x; 2n, 2T) S_{2k, 2t_j}(-y; 2n, 2T) \notag \\
  = {}& \widetilde{K}_{2t_i, 2t_j}(x,y; 2n, 2T)+\widetilde{K}_{2t_i, 2t_j}(x,-y; 2n, 2T), \label{eq:even_part_rec} \\
  \widetilde{K}_{t_i, t_j}^{\absorb}(x,y; n, T) = {}& \frac{2n}{\pi} \sum^{n - 1}_{k = 0} \frac{1}{h^{2T}_{2n, 2k + 1}}S_{2k+1, 2(T - t_i)}(x; 2n, 2T) S_{2k+1, 2t_j}(-y; 2n, 2T) \notag \\
  = {}& \widetilde{K}_{2t_i, 2t_j}(x,y; 2n, 2T)-\widetilde{K}_{2t_i, 2t_j}(x,-y; 2n, 2T). \label{eq:odd_part_rec}
\end{align}
On the other hand, comparing the formulas \eqref{eq:W_circ_and_W} for $\W^{\reflect}_{[i, j)}(x, y)$ and \eqref{eq:W_circ_and_W_abs} for $\W^{\absorb}_{[i, j)}(x, y)$ with \eqref{eq:W_circ_and_W_circle} for $\W_{[i, j)}(x, y)$, we have (recalling that they depend on parameters $t_i$, $t_j$ and $n$)
\begin{align}
  \W^{\reflect}_{[i, j)}(x, y) = {}& \W_{[i, j)}(x, y) \bigg\rvert_{\substack{n \mapsto 2n, \\ t_i \mapsto 2t_i, t_j \mapsto 2t_j}} +  \W_{[i, j)}(x, -y) \bigg\rvert_{\substack{n \mapsto 2n, \\ t_i \mapsto 2t_i, t_j \mapsto 2t_j}}, \\ 
  \W^{\absorb}_{[i, j)}(x, y) = {}& \W_{[i, j)}(x, y) \bigg\rvert_{\substack{n \mapsto 2n, \\ t_i \mapsto 2t_i, t_j \mapsto 2t_j}} -  \W_{[i, j)}(x, -y) \bigg\rvert_{\substack{n \mapsto 2n, \\ t_i \mapsto 2t_i, t_j \mapsto 2t_j}}.
\end{align}
Furthermore, the kernel $K_{s, t}(x, y; n, T)$ is periodic with respect to both $x$ and $y$, with period $2\pi$. Thus we obtain \eqref{eq:even_part_rec} and \eqref{eq:odd_part_rec}.


\end{proof}

\section{Rank 1 property of the kernels $\widetilde{K}^{\tac, \even}$ and $\widetilde{K}^{\tac, \odd}$}\label{sec:rank1}

In this section we prove that the kernels $\widetilde{K}^{\tac, \even}$ and $\widetilde{K}^{\tac, \odd}$ satisfy the property that their $\sg$-derivative is a rank 1 kernel, and find triple integral representations for them.  Namely we have the following proposition, which is an important part of the proof of Proposition \ref{lem:tacnode_relation} but may also be of independent interest.
\begin{prop}\label{prop:triple_int}
  The kernels $\widetilde{K}^{\tac, \even}$ and $\widetilde{K}^{\tac, \odd}$ have the following triple integral representations:
  \begin{align}
    \widetilde{K}^{\tac, \even}_{s, t}(\xi,\eta; \sg) = {}& \frac{1}{4\pi^2} \int_\sg^\infty \,d\widetilde{\sg}\int_{\Sg_T}\,du\int_{\Sg_T}\, dv\, \frac{e^{\frac{su^2}{2} - iu \xi}}{e^{\frac{tv^2}{2}-iv\eta}} \big(f(u; \widetilde{\sg})+g(u; \widetilde{\sg})\big)\big(f(v; \widetilde{\sg})+g(v; \widetilde{\sg})\big), \label{rank7a} \\
    \widetilde{K}^{\tac, \odd}_{s, t}(\xi,\eta; \sg) = {}& \frac{1}{4\pi^2} \int_\sg^\infty \,d\widetilde{\sg}\int_{\Sg_T}\,du\int_{\Sg_T}\, dv\, \frac{e^{\frac{su^2}{2} - iu \xi}}{e^{\frac{tv^2}{2}-iv\eta}} \big(f(u; \widetilde{\sg})-g(u; \widetilde{\sg})\big) \big(f(v; \widetilde{\sg})-g(v; \widetilde{\sg})\big), \label{rank8a}
  \end{align}
  where the functions $f(u; {\sg})$ and $g(u; {\sg})$ are defined in \eqref{tac18}.
\end{prop}

\begin{rmk} For the odd tacnode process a similar result appeared very recently in \cite[Proposition 2.8]{Ferrari-Veto16}.
\end{rmk}

Since the functions $f(u; {\sg})$ and $g(u; {\sg})$ are defined by the Hastings--McLeod solution $\mathbf{\Psi}(\zeta; s)$ in \eqref{cr4} and \eqref{cr5}, we recall some properties of $\mathbf{\Psi}(\zeta; s)$ before proceeding with the proof of Proposition \ref{prop:triple_int}.
First, $\mathbf{\Psi}(\zeta; s)$ satisfies the differential equation
\begin{equation}\label{equiv2}
  \frac{\d}{\d s} \mathbf \Psi(\z; s) =
  \begin{pmatrix}
    -i\z & q(s) \\
    q(s) & i\z
  \end{pmatrix}
  \mathbf \Psi(\z; s),
\end{equation}
where $q(s)$ is the Hastings--McLeod solution to the Painlev\'{e} II equation. Equation \eqref{equiv2} is the second differential equation of the Flaschka--Newell Lax pair (the first one is \eqref{cr4}). Then for the functions $f(\zeta; s)$ and $g(\zeta; s)$ defined in \eqref{tac18}, we have
\begin{equation}\label{rank5}
  \frac{\d}{\d s}f(\zeta; s)=-i\zeta f(\zeta; s)+q(s)g(\zeta; s),\qquad \frac{\d}{\d s}g(\zeta; s)=q(s)f(\zeta; s)+i\zeta g(\zeta; s).
\end{equation}
We also note that the uniqueness of the boundary value problem \eqref{cr4} and \eqref{cr5} implies
\begin{equation} \label{eq:symmetry_of_H-ML_solution}
  \Psi_{i, j}(-\zeta; s) = \Psi_{3 - i, 3 - j}(\zeta; s), \quad i, j = 1, 2, 
\end{equation}
so for the functions $f(\zeta; s)$ and $g(\zeta; s)$ defined in \eqref{tac18}, we have
\begin{equation}\label{rank3}
  f(-\zeta; s)=-g(\zeta; s).
\end{equation}

\begin{proof}[Proof of Proposition \ref{prop:triple_int}]
  We prove \eqref{rank7a} in detail, and the proof of \eqref{rank8a} is analogous and omitted.
  
  The kernel $\widetilde{K}^{\tac, \even}_{s, t}(\xi,\eta; \sg)$ can be written as
  \begin{multline}\label{rank1}
    \widetilde{K}^{\tac, \even}_{s, t}(\xi,\eta; \sg) = \frac{1}{2\pi} \int_{\Sg_T}\,du\int_{\Sg_T}\, dv\, e^{\frac{su^2}{2}-\frac{tv^2}{2}} \frac{f(u; \sg)g(v; \sg)-g(u; \sg)f(v; \sg)}{2\pi i(u-v)} e^{-i(u\xi-v\eta)} \\
    + \frac{1}{2\pi} \int_{\Sg_T}\,du\int_{\Sg_T}\, dv\, e^{\frac{su^2}{2}-\frac{tv^2}{2}} \frac{f(u; \sg)g(v; \sg)-g(u; \sg)f(v; \sg)}{2\pi i(u-v)} e^{-i(u\xi+v\eta)}.
  \end{multline}
  In the second integral, we make the change of variable $v\mapsto (-v)$, and make use of \eqref{rank3} to change $f(-v; \sigma)$ and $g(-v; \sigma)$ into $-g(v; \sigma)$ and $-f(v; \sigma)$ respectively. Notice that the contour $\Sg_T$ is invariant under this transformation, and so we obtain
\begin{equation}\label{rank4}
  \begin{aligned}
    \widetilde{K}^{\tac, \even}_{s, t}(\xi,\eta; \sg)&=\frac{1}{4\pi} \int_{\Sg_T}\,du\int_{\Sg_T}\, dv\, \frac{e^{\frac{su^2}{2} - iu \xi}}{e^{\frac{tv^2}{2}-iv\eta}} \bigg[\frac{f(u; \sg)g(v; \sg)-g(u; \sg)f(v; \sg)}{2\pi i(u-v)}  \\
    &\hspace{7cm}+ \frac{f(u; \sg)f(v; \sg)-g(u; \sg)g(v; \sg)}{2\pi i(u+v)} \bigg] \\
    &=\frac{1}{4\pi} \int_{\Sg_T}\,du\int_{\Sg_T}\, dv\, \frac{e^{\frac{su^2}{2} - iu \xi}}{e^{\frac{tv^2}{2}-iv\eta}} \\
    &\quad\times \frac{(u+v)(f(u; \sg)g(v; \sg)-g(u; \sg)f(v; \sg))+(u-v)(f(u; \sg)f(v; \sg)-g(u; \sg)g(v; \sg))}{2\pi i(u^2-v^2)}.  \\
  \end{aligned}
\end{equation}
By \eqref{rank5}, we have
\begin{multline}\label{rank6}
  \frac{\d}{\d\sg} \left[ \frac{(u+v)(f(u; \sg)g(v; \sg)-g(u; \sg)f(v; \sg))+(u-v)(f(u; \sg)f(v; \sg)-g(u; \sg)g(v; \sg))}{2\pi i(u^2-v^2)}\right]= \\
  -\frac{(f(u; \sg)+g(u; \sg))(f(v; \sg)+g(v; \sg))}{2\pi}.
\end{multline}
This derivative identity, together with the property that $f(\zeta; s)$ and $g(\zeta; s)$ vanish exponentially fast as $s \to +\infty$ uniformly for all $\zeta \in \Sigma_T$, implies \eqref{rank7a}.

The vanishing property of $f(\zeta; s)$ and $g(\zeta; s)$ is implied by \cite[Lemma 5.2]{Liechty-Wang16}, see \cite[Formula (5.40)]{Liechty-Wang16}. We note that this vanishing property was used implicitly also in \cite[Formula (346)]{Liechty-Wang14-2}.
\end{proof}

\section{The Schlesinger transformation of a Lax pair associated to \Painleve\ II} \label{sec:M_nu_defn}

In the paper \cite{Delvaux13a}, a family of hard-edge tacnode kernels is derived from the nonintersecting (squared) Bessel processes. The kernels are given in terms of a $4 \times 4$ Riemann--Hilbert problem associated to the nonhomogeneous \Painleve\ II (PII) equation. Equivalently, they can be uniquely determined by a $4\times 4$ Lax pair associated to the same solution to PII. In this section we give the definition of the Lax pair, following \cite{Delvaux13a}, and then derive the Schlesinger transformation formula for the Lax pair which preserves the Hastings--McLeod solutions to PII. This transformation is a key ingredient in the proofs of Proposition \ref{lem:tacnode_relation} and Theorem \ref{thm:half-integer}. The kernels will be defined and discussed in Section \ref{sec:kernels}.

\subsection{Definition of $M_{\nu}$ by Lax pair}
 
 Recall the general PII equation \eqref{PII} and the Hastings--McLeod solution defined in \eqref{eq:PII_BC}. Throughout this section we  denote the Hastings--McLeod solution to PII as $q_\nu(\zeta)$. We also define the related Hamiltonian, following \cite[Formula (21)]{Delvaux13a} \footnote{We note that $K(\zeta; q, p) = \frac{1}{2}(p^2 - \zeta q^2 - q^4 - \nu q)$ is a Hamiltonian for PII, in the sense that $\frac{\partial K}{\partial p} = q'$ and $-\frac{\partial K}{\partial q} = p'$ implies the PII equation for $q$.}
\begin{equation} \label{eq:defn_u_nu}
  u_\nu(\zeta)=q_\nu'(\zeta)^2 - \zeta q_\nu(\zeta)^2-q_\nu(\zeta)^4 + 2\nu q_\nu(\zeta).
\end{equation}

To state the Lax pair, we introduce variables $s, \tau$, and let
\begin{equation} \label{eq:sigma_relation}
  \sg=2^{2/3}(2s-\tau^2).
\end{equation}
Following the notational convention in \cite[Formula (23)]{Delvaux13a}, we define the following quantities in terms of $s$ and $\tau$:
\begin{align}
  c = {}& -2^{-1/3} u_{\nu}(\sigma) + s^2, \quad d = 2^{-1/3} q_{\nu}(\sigma), \\
  g + a = {}& -c^2 + d^2 + s = 2^{-2/3} (q_{\nu}(\sigma)^2 - u_{\nu}(\sigma)^2) + 2^{2/3} s^2 u_{\nu}(\sigma) - s^4 + s, \\
  b - h = {}& 2\tau d = 2^{2/3} \tau q_{\nu}(\sigma), \\
  b + h = {}& \frac{1}{2\tau} \frac{\partial d}{\partial \tau} + 2cd = -2^{1/3}(q'_{\nu}(\sigma) + q_{\nu}(\sigma) u_{\nu}(\sigma)) + 2^{2/3} s^2 q_{\nu}(\sigma). \label{eq:notation_b+h}
\end{align}
Then for $\nu>-1/2$, we define the $4 \times 4$ matrix-valued function $M_\nu(z; s, \tau)$, for $z \in \compC \setminus (-\infty, 0]$ and $s, \tau$ in a neighbourhood of $\realR$, as the solution to the differential equation
\begin{equation} \label{eq:diff_for_M_nu}
  \frac{\partial}{\partial z} M_\nu(z; s, \tau)=U_\nu M_\nu(z; s, \tau), \quad U_{\nu} =
  \begin{pmatrix}
    -c + \tau & d + \frac{\nu}{z} & i & 0 \\
    -d + \frac{\nu}{z} & c - \tau & 0 & i \\
    -i(-z + g + a + s) & -i(b + h) & c + \tau & d - \frac{\nu}{z} \\
    -i(b + h) & -i(z + g + a + s) & -d - \frac{\nu}{z} & -c - \tau
  \end{pmatrix},
\end{equation}
which satisfies the asymptotics as $z\to+\infty$ in the sector $-\pi/12 < \arg z < 7\pi/12$,
\begin{multline} \label{eq:asy_M_nu}
  M_\nu(z) = \left( I + \frac{M_{\nu, 1}}{z} + \bigO \left( \frac{1}{z^2} \right) \right) \diag \left( (-z)^{-1/4}, z^{-1/4}, (-z)^{1/4}, z^{1/4} \right) \\
  \times \A \diag \left( e^{-\theta_1(z) + \tau z}, e^{-\theta_2(z) - \tau z}, e^{\theta_1(z) + \tau z}, e^{\theta_2(z) - \tau z} \right),
\end{multline}
where $M_{\nu, 1}$ is a constant $4 \times 4$ matrix whose explicit value we are not interested in, all power functions are taken as the principal branch, and 
\begin{equation}
  \A = \frac{1}{\sqrt{2}}
  \begin{pmatrix}
    1 & 0 & -i & 0 \\
    0 & 1 & 0 & i \\
    -i & 0 & 1 & 0 \\
    0 & i & 0 & 1
  \end{pmatrix}, \quad
  \theta_1(z) = \frac{2}{3} (-z)^{3/2} + 2s(-z)^{1/2}, \quad \theta_2(z) = \frac{2}{3} z^{3/2} + 2s z^{1/2}.
\end{equation}
The Lax pair consists of \eqref{eq:diff_for_M_nu} together with another differential equation for $M_{\nu}(z; s, \tau)$ with respect to $s$:
\begin{equation} \label{eq:Lax_V_part}
  \frac{\partial}{\partial s}M_\nu(z; s, \tau)=V_{\nu} M_\nu(z; s, \tau), \quad V_{\nu} = 2
  \begin{pmatrix}
    c & d & -i & 0 \\
    d & c & 0 & i \\
    i(-z + g + a) & i(b - h) & -c & -d \\
    i(h - b) & -i(z + g + a) & -d & -c
  \end{pmatrix}.
\end{equation}
$M_{\nu}(z; s, \tau)$ also satisfies a differential equation with respect to $\tau$, but we omit it here, see \cite[Proposition 5]{Delvaux13a}.

We note that there are other Lax pairs associated to the PII equation, like the Flaschka--Newell and the Jimbo--Miwa Lax pairs which are $2 \times 2$ \cite[Section 4.2]{Fokas-Its-Kapaev-Novokshenov06}.

In \cite{Delvaux13a}, $M_{\nu}$ is defined first by a Riemann--Hilbert problem \cite[RH Problem 1]{Delvaux13a}, with notation $M$, and our $M_{\nu}$ is the solution to the Riemann--Hilbert problem in the sector $0 < \arg z < \varphi_1$ that is part of the first quadrant, see \cite[Formula (9)]{Delvaux13a}. The asymptotics \eqref{eq:asy_M_nu} is part of the Riemann--Hilbert problem. The existence of $M_{\nu}$ on $\compC \setminus (-\infty, 0]$, which is equivalent to the solvability of the Riemann--Hilbert problem, is established in \cite[Theorem 2]{Delvaux13a}.
It is shown that $M_{\nu}$ defined by the Rieman--Hilbert problem also satisfies the differential equation \eqref{eq:diff_for_M_nu} in \cite[Propositions 3 and 4]{Delvaux13a}. The entries of $U_{\nu}$ are given in \cite[Fomulas (24), (25), (149), (150), (166)]{Delvaux13a}. By the theory of differential equations, for example \cite[Chapters IV and V]{Wasow87}, the solution to \eqref{eq:diff_for_M_nu} that satisfies \eqref{eq:asy_M_nu} is unique if it exists.

In the homogeneous case $\nu = 0$, the Riemann--Hilbert problem for $M_{\nu}$ as well as the Lax pair were first defined and analysed in \cite{Delvaux-Kuijlaars-Zhang11}. Later it was found that entries of $\left. M_\nu \right\rvert_{\nu = 0}$ have Airy resolvent formulas, see \cite{Delvaux13, Kuijlaars14}. It was also shown in \cite{Liechty-Wang16} that the entries of $\left. M_\nu \right\rvert_{\nu = 0}$ have integral representations in terms with the solution to the  $2 \times 2$ Flaschka--Newell Lax pair associated to the Hastings--McLeod solution to the homogeneous PII equation. But for the general $\nu \neq 0$ case, analogous results are not known.

\subsection{Schlesinger transformation of the Lax pair} \label{subsec:Backlund_trans}


The Hastings--McLeod solutions $q_\nu(\sg)$ and $q_{\nu+1}(\sg)$ to PII are related by the following \Backlund\ transformation (see \cite[Section 6.1]{Fokas-Its-Kapaev-Novokshenov06}, noting that $q_{-\nu}(\sg) = -q_{\nu}(\sg)$):
\begin{equation} \label{eq:Backlund_q}
  q_{\nu+1}(\sg) + q_\nu(\sg) = \frac{2\nu+1}{2q_\nu(\sg)^2 - 2q_\nu'(\sg) + \sg} = u_{\nu+1}(\sg) - u_\nu(\sg).
\end{equation}
The first identity can be verified by checking that the function
\begin{equation}
  -q_\nu(\sg) + \frac{2\nu+1}{2q_\nu(\sg)^2 - 2q_\nu'(\sg) + \sg},
\end{equation}
satisfies PII with parameter $\nu+1$, and has the correct asymptotics at both $\sg = +\infty$ and $\sg = -\infty$. Then the second identity is an exercise. Accordingly, the Lax pairs associated to the Hastings--McLeod solutions should have corresponding Schlesinger transformations. The Schlesinger transformation for the $2 \times 2$ Flaschka--Newell Lax pair is well known, see \cite[Section 6.1]{Fokas-Its-Kapaev-Novokshenov06}. In this section we derive the Schlesinger transformation for the $4 \times 4$ Lax pair given by \eqref{eq:diff_for_M_nu}--\eqref{eq:Lax_V_part}.

\begin{prop}\label{Schlesinger}
  The matrix-valued functions $M_\nu(z)$ and $M_{\nu+1}(z)$ satisfy the relation
  \begin{equation}\label{Schles_def}
    M_{\nu+1}(z;s,\tau)= \left(I + \frac{\alpha_\nu}{z} R_\nu\right)\Sg M_\nu(z;s,\tau)\Sg\,,
  \end{equation}
  where 
  \begin{equation} \label{eq:formula_R_nu}
    R_\nu :=
    \begin{pmatrix}
      \beta_\nu & \beta_\nu & -i & i \\
      -\beta_\nu & -\beta_\nu & i & -i \\
      i\gamma_\nu & i\gamma_\nu &  -\delta_\nu & \delta_\nu \\
      i\gamma_\nu & i\gamma_\nu & -\delta_\nu & \delta_\nu
    \end{pmatrix},
    \qquad \Sg= \diag(1, -1, 1, -1),
  \end{equation}
  and, using notations in \eqref{eq:sigma_relation} -- \eqref{eq:notation_b+h}, and with $q_{\nu} = q_{\nu}(\sigma)$ and $u_{\nu} = u_{\nu}(\sigma)$,
  \begin{align}
    \alpha_\nu = {}& 2^{-1/3}(q_{\nu + 1}(\sigma) + q_{\nu}(\sigma)) = 2^{-1/3}(u_{\nu + 1}(\sigma) - u_{\nu}(\sigma)) = \frac{2^{-1/3}(2\nu+1)}{2q_\nu^2-2q_\nu'+2^{5/3}s-2^{2/3}\tau^2}, \label{Backlund} \\
    \beta_\nu = {}& c - d + \tau = s^2+\tau - 2^{-1/3}(q_\nu+u_\nu), \label{abcd_def} \\
    \gamma_\nu = {}& b + h + g + a + s = 2^{-2/3}(q_\nu^2-u_\nu^2)+2^{2/3}s^2(q_\nu+u_\nu)-2^{1/3}(q_\nu u_\nu+q_\nu')+2s-s^4, \\
    \delta_\nu = {}& c - d - \tau = s^2-\tau - 2^{-1/3}(q_\nu+u_\nu).
  \end{align}
\end{prop}

\begin{proof}[Proof of Proposition \ref{Schlesinger}]

  First we show that right-hand side of \eqref{Schles_def} satisfies the differential equation \eqref{eq:diff_for_M_nu} with $\nu$ replaced by $\nu + 1$. In the proof we write $M_{\nu}(z; s, t)$ as $M_{\nu}$ if there is no chance of confusion. We need to check that
  \begin{equation}
    \frac{d}{dz}\left(I + \frac{\alpha_{\nu}}{z} R_\nu\right)\Sg M_{\nu}(z)\Sg = U_{\nu+1} \left(I + \frac{\alpha_\nu}{z} R_\nu \right)\Sg M_\nu(z)\Sg.
  \end{equation}
  Using the differential equation \eqref{eq:diff_for_M_nu} for $M_{\nu}(z)$, this amounts to checking that
  \begin{equation}
    \left( I + \frac{\alpha_\nu}{z} R_\nu \right) \Sg U_\nu M_{\nu}(z) \Sg - \frac{\alpha_\nu}{z^2} R_\nu \Sg M_{\nu}(z)\Sg = U_{\nu+1} \left(I + \frac{\alpha_\nu}{z} R_\nu \right)\Sg M_{\nu}(z)\Sg,
  \end{equation}
  or equivalently
  \begin{equation} \label{eq:formula_for_U_nu+1}
    U_{\nu+1} = \left[\left(I + \frac{\alpha_\nu}{z} R_\nu\right)\Sg U_\nu \Sg-\frac{\alpha_\nu}{z^2} R_\nu\right] \left(I + \frac{\alpha_\nu}{z} R_\nu\right)^{-1}.
  \end{equation}
  It is straightforward (although a little tedious) to check that
  \begin{equation} \label{eq:commutation_RU}
    \begin{split}
      & R_{\nu} \Sigma U_{\nu} \Sigma - \Sigma U_{\nu} \Sigma R_{\nu} \\
      = {}& z
      \begin{pmatrix}
        1 & 1 & 0 & 0 \\
        -1 & - 1& 0 & 0 \\
        -i(\beta_{\nu} + \delta_{\nu}) & -i(\beta_{\nu} + \delta_{\nu}) & -1 & 1 \\
        -i(\beta_{\nu} + \delta_{\nu}) & -i(\beta_{\nu} + \delta_{\nu}) & -1 & 1
      \end{pmatrix}
      + 2(\beta_{\nu} \delta_{\nu} + \gamma_{\nu})
      \begin{pmatrix}
        0 & 1 & 0 & 0 \\
        1 & 0 & 0 & 0 \\
        0 & 0 & 0 & -1 \\
        0 & 0 & -1 & 0
      \end{pmatrix} 
      - \frac{2\nu}{z} R_{\nu} \\
      = {}& \frac{z}{\alpha_{\nu}} W + W R_{\nu} + \frac{1}{z} R_{\nu},
    \end{split}
  \end{equation}
  where
  \begin{equation}
    W = \alpha_{\nu}
    \begin{pmatrix}
      1 & 1 & 0 & 0 \\
      -1 & -1 & 0 & 0 \\
      -i(\beta_{\nu} + \delta_{\nu}) & -i(\beta_{\nu} + \delta_{\nu}) & -1 & 1 \\
      -i(\beta_{\nu} + \delta_{\nu}) & -i(\beta_{\nu} + \delta_{\nu}) & -1 & 1
    \end{pmatrix}
    + \frac{2\nu + 1}{z}
    \begin{pmatrix}
      0 & 1 & 0 & 0 \\
      1 & 0 & 0 & 0 \\
      0 & 0 & 0 & -1 \\
      0 & 0 & -1 & 0
    \end{pmatrix}.
  \end{equation}
  Here for the second identity of \eqref{eq:commutation_RU}, we need the identity
  \begin{equation}
    \beta_{\nu} \delta_{\nu} + \gamma_{\nu} = \frac{2\nu + 1}{2\alpha_{\nu}}.
  \end{equation}
  Hence it follows that
  \begin{equation}
    \left[\left(I + \frac{\alpha_\nu}{z} R_\nu\right)\Sg U_\nu \Sg-\frac{\alpha_\nu}{z^2} R_\nu\right] \left(I + \frac{\alpha_\nu}{z} R_\nu\right)^{-1} = \Sigma U_{\nu} \Sigma + W.
  \end{equation}
  To check \eqref{eq:formula_for_U_nu+1}, we only need to show that
  \begin{equation} \label{eq:technical_comparison}
    \Sigma U_{\nu} \Sigma + W = U_{\nu + 1}.
  \end{equation}
  
  Although there are $16$ entries on both sides of \eqref{eq:technical_comparison} to be compared, it turns out most of them follows straightforwardly from \eqref{Backlund}, and only the four entries in the lower-left block require discussion. Namely, we need to check, after writing $\beta_{\nu}, \delta_{\nu}, g + a, b + h$ into formulas in $q_{\nu} = q_{\nu}(\sigma)$ and $u_{\nu} = u_{\nu}(\sigma)$,
  \begin{multline} \label{eq:comparison_ll_corner}
    i
    \begin{pmatrix}
      z & 0 \\
      0 & -z
    \end{pmatrix}
    + i
    \begin{pmatrix}
      - 2^{-2/3}(q_\nu^2-u_\nu^2)-2s+s^4-2^{2/3}s^2u_\nu & -2^{1/3}(q_\nu u_\nu+q_\nu')+2^{2/3}s^2 q_\nu \\
      -2^{1/3}(q_\nu u_\nu+q_\nu')+2^{2/3}s^2 q_\nu & - 2^{-2/3}(q_\nu^2-u_\nu^2)-2s+s^4-2^{2/3}s^2u_\nu
    \end{pmatrix} \\
    -i \alpha_\nu(2s^2-2^{2/3}(q_\nu+u_\nu))
    \begin{pmatrix}
      1 & 1 \\
      1 & 1
    \end{pmatrix} 
    = i
    \begin{pmatrix}
      z & 0 \\
      0 & -z
    \end{pmatrix} + \\
    i
    \begin{pmatrix}
      - 2^{-2/3}(q^2_{\nu + 1} - u^2_{\nu + 1}) - 2s + s^4 - 2^{2/3} s^2 u_{\nu + 1} & 2^{1/3}(q_{\nu + 1} u_{\nu + 1} + q'_{\nu + 1}) - 2^{2/3} s^2 q_{\nu + 1} \\
      2^{1/3} (q_{\nu + 1} u_{\nu + 1} - q'_{\nu + 1}) + 2^{2/3} s^2 q_{\nu + 1} & - 2^{-2/3} (q^2_{\nu + 1} - u^2_{\nu + 1}) - 2s + s^4 - 2^{2/3} s^2 u_{\nu + 1}
    \end{pmatrix}.
  \end{multline}

  Consider first the diagonal terms in \eqref{eq:comparison_ll_corner}. It suffices to show
  \begin{equation} \label{eq:technical_in_comparison}
    2\alpha_\nu s^2 - 2^{2/3}s^2(u_{\nu+1} - u_\nu) + 2^{-2/3}(u_{\nu+1}^2 - q_{\nu+1}^2 + q_\nu^2 - u_\nu^2) - 2^{2/3} \alpha_\nu (q_\nu + u_\nu) = 0.
  \end{equation}
  The first two terms in \eqref{eq:technical_in_comparison} cancel because of \eqref{Backlund}, and then the equation is simplified into
  \begin{equation} \label{eq:diagonal_Backlund}
    (u_{\nu+1} - u_\nu)(u_{\nu+1} + u_\nu)-(q_{\nu+1} - q_\nu)(q_{\nu+1} + q_\nu)-2^{4/3} \alpha_\nu (q_\nu + u_\nu) = 0
  \end{equation}
  by multiplying $2^{2/3}$ on both sides. The left-hand side of \eqref{eq:diagonal_Backlund} is simplified, by \eqref{Backlund}, into
  \begin{equation}
    2^{1/3} \alpha_\nu (u_{\nu+1} + u_\nu)-2^{1/3} \alpha_\nu (q_{\nu+1} - q_\nu)-2^{4/3} \alpha_\nu (q_\nu + u_\nu) = 2^{1/3}\alpha_{\nu} [(u_{\nu+1}-u_\nu)-(q_{\nu+1}+q_\nu)],
  \end{equation}
  and it vanishes by \eqref{eq:Backlund_q}. Thus we confirm \eqref{eq:technical_in_comparison}.
  
  Consider next the off-diagonal entries in \eqref{eq:comparison_ll_corner}. It suffices to show
  \begin{equation}
    2^{1/3}(q'_{\nu+1} + u_{\nu+1}q_{\nu+1}) - 2^{2/3} s^2q_{\nu+1} = -2\alpha_\nu s^2 + 2^{2/3} \alpha_\nu (q_\nu + u_\nu) - 2^{1/3}(q_\nu u_\nu + q'_\nu) + 2^{2/3}s^2 q_\nu.
  \end{equation}
  We can immediately cancel the $s^2$ terms using \eqref{Backlund}, and then by dividing $2^{1/3}$ on both sides, we are left with
  \begin{equation}\label{offdiag3}
    (q'_{\nu + 1} + q'_\nu) + (q_\nu u_\nu + u_{\nu+1}q_{\nu+1})-2^{1/3} \alpha_\nu (q_\nu + u_\nu) = 0.
  \end{equation}
  The first term in this expression is
  \begin{equation}
    q'_{\nu+1} + q'_\nu = 2^{1/3}\frac{d \alpha_\nu}{d \sg}.
  \end{equation}
Using the formula for $\alpha_\nu$ given in \eqref{Backlund}, along with the \Painleve\ II equation \eqref{PII} with parameter $\nu$, one finds that $\alpha_\nu$ satisfies
\begin{equation}
  \frac{d \alpha_\nu}{d \sg} = -\frac{(2\nu + 1)(4q_{\nu} q'_{\nu} - 2q''_{\nu} + 1)}{2^{1/3} (2q^2_{\nu} - 2q'_{\nu} + \sigma)^2} = -\frac{(2\nu + 1)(4q_{\nu} q'_{\nu} - 4q^3_{\nu} -2\sigma q_{\nu} + 2\nu + 1)}{2^{1/3}(2q^2_{\nu} - 2q'_{\nu} + \sigma)^2} = 2 \alpha_\nu q_\nu - 2^{1/3} \alpha_\nu^2,
\end{equation}
and so equation \eqref{offdiag3} can be written as
\begin{equation}
  2^{4/3} \alpha_\nu q_\nu-2^{2/3} \alpha_\nu ^2+(q_\nu u_\nu+u_{\nu+1}q_{\nu+1})-2^{1/3} \alpha_\nu (q_\nu+u_\nu)=0.
\end{equation}
In the $\alpha_\nu^2$ term, we now replace one of the factors of $\alpha_\nu$ with $2^{-1/3}(u_{\nu+1}-u_\nu)$, and the other with $2^{-1/3}(q_{\nu+1}+q_\nu)$, yielding 
\begin{equation} \label{eq:Backlund_off_diag}
  2^{4/3} \alpha_\nu q_\nu - (u_{\nu+1} - u_\nu)(q_{\nu+1} + q_\nu) + (q_\nu u_\nu+u_{\nu+1}q_{\nu+1}) - 2^{1/3} \alpha_\nu (q_\nu+u_\nu)=0.
\end{equation}
The identity \eqref{eq:Backlund_off_diag} can be written as
\begin{equation}
  -q_\nu(u_{\nu+1} - u_\nu-2^{1/3} \alpha_\nu) + u_\nu(q_{\nu+1} + q_\nu - 2^{1/3} \alpha_\nu) = 0,
\end{equation}
which holds by \eqref{Backlund}.

By the theory of differential equations, the solution to \eqref{eq:diff_for_M_nu} that satisfies \eqref{eq:asy_M_nu} is unique, where $\nu$ can be any real number greater than $-1/2$. We have shown that the right-hand side of \eqref{Schles_def} satisfies \eqref{eq:diff_for_M_nu} with $\nu$ replaced by $\nu + 1$. On the other hand, the conjugation by $\Sg$ and the left multiplication by the matrix $(I + \alpha z^{-1} R_\nu)$ to $M_{\nu}$ do not change the leading asymptotic behavior at infinity, so the right-hand side of \eqref{Schles_def} satisfies \eqref{eq:asy_M_nu}. Thus we verify \eqref{Schles_def} and prove the Proposition.
\end{proof}

\begin{rmk}
  The formula \eqref{eq:formula_R_nu} of $R_{\nu}$ seems to come out of the blue, and the proof suggests little on how the formula is found. Suppose we have obtained the formula \eqref{Schles_def} for the Schlesinger transformation with the help of guesswork, then the explicit formula of $R_{\nu}$ can be derived by matching the subleading coefficients for $M_{\nu}(z)$ and $M_{\nu + 1}(z)$ as $z \to \infty$, namely $M_{\nu, 1}$ in \eqref{eq:asy_M_nu} and its counterpart $M_{\nu + 1, 1}$. The explicit formula of $M_{\nu, 1}$ is given in \cite[Theorem 1]{Delvaux13a}, with the notation $M_1$ there.
\end{rmk}

\section{The hard-edge tacnode kernels of Delvaux and proofs of Proposition \ref{lem:tacnode_relation} and Theorem \ref{thm:half-integer}} \label{sec:kernels}

In this section, for notational convention we let
\begin{equation}
  \alpha = \nu - 1/2,
\end{equation}
so that $\nu$ is an integer when $\alpha$ is a half-integer.

Then the limiting hard-edge tacnode kernel $K^{\tac, (\al)}(x,y; s, \tau)$ for nonintersecting Bessel process with parameter $\alpha > -1$ is given in terms of $M_{\nu}(z; s, \tau)$ defined in Section \ref{sec:M_nu_defn} \cite[Theorem 4 and Remark 2]{Delvaux13a}. For notational convention, we denote the $4\times 4$ matrix $D$ 
\begin{equation}
  D =
  \begin{pmatrix}
    x+ y & -x+ y & 0 & 0 \\
    -x + y & x+ y & 0 & 0 \\
    0 & 0 & x+ y & x- y \\
    0 & 0 & x-y & x+ y
  \end{pmatrix},
\end{equation}
and we have the formula
\begin{equation}\label{eq:Delvaux_kernel}
  K^{\tac, (\al)}(x,y; s, \tau) = \frac{1}{2\pi i(x^2-y^2)}
  \left( -1, 0, 1, 0 \right)
  M_\nu(y; s, \tau)^{-1} D M_\nu(x; s, \tau)
  \left( 1, 0, 1, 0 \right)^T.
\end{equation}
Note that the entries of $M^{-1}_{\nu}$ are the same as the entries of $M_{\nu}(x; s, -\tau)$ up to permutation and $(-1)^{\pm}$ factors. More concretely, by \cite[Lemma 2]{Delvaux13a} we have
\begin{equation} \label{eq:inverse_conj_symmetry}
  M_{\nu}(z; s, \tau)^{-1} = K^{-1} M_{\nu}(z; s, -\tau)^T K, \quad K = 
  \begin{pmatrix}
    0 & -I_{2 \times 2} \\
    I_{2 \times 2} & 0
  \end{pmatrix}.
\end{equation}
Hence if we denote the $4$-dimensional column vector
\begin{equation} \label{eq:vect_n}
  \vec{n}_{\nu}(z; s, \tau) = \left( n_{\nu}(z; s, \tau)_i \right)^4_{i = 1}, \quad \text{where} \quad n_{\nu}(z; s, \tau)_i = M_{\nu}(z; s, \tau)_{i1} + M_{\nu}(z; s, \tau)_{i3},
\end{equation}
we have
\begin{equation} \label{eq:kernel_reflecting_expanded}
  \begin{split}
    K^{\tac, (\al)}(x,y; s, \tau) = {}& \frac{1}{2\pi i(x-y)} \Big( n_{\nu}(y; s, -\tau)_1 n_{\nu}(x; s, \tau)_3 + n_{\nu}(y; s, -\tau)_2 n_{\nu}(x; s, \tau)_4 \\
  & \phantom{\frac{1}{2\pi i(x-y)} \Big(} - n_{\nu}(y; s, -\tau)_3 n_{\nu}(x; s, \tau)_1 - n_{\nu}(y; s, -\tau)_4 n_{\nu}(x; s, \tau)_2 \Big) \\
  & + \frac{1}{2\pi (x + y)} \Big( n_{\nu}(y; s, -\tau)_1 n_{\nu}(x; s, \tau)_4 + n_{\nu}(y; s, -\tau)_2 n_{\nu}(x; s, \tau)_3 \\
  & \phantom{+ \frac{1}{2\pi (x + y)} \Big(} + n_{\nu}(y; s, -\tau)_3 n_{\nu}(x; s, \tau)_2 + n_{\nu}(y; s, -\tau)_4 n_{\nu}(x; s, \tau)_1 \Big).
  \end{split}
\end{equation}
\begin{rmk}
  In \cite[Theorem 4]{Delvaux13a}, the limiting hard-edge tacnode kernel is defined for the nonintersecting squared Bessel process. Our kernel defined in \eqref{eq:Delvaux_kernel} is for the limiting kernel of the nonintersecting Bessel process, which differs from the squared one by a quadratic change of variables, see \cite[Remark 2]{Delvaux13a}.
\end{rmk}

Proposition \ref{Schlesinger} implies that the tacnode kernel with Bessel parameter $\alpha$ can be expressed in terms of the Lax pair for the inhomogeneous PII equation with parameter $\nu - 1$ as well as $\nu$. Indeed we have that by \eqref{eq:Delvaux_kernel} and \eqref{Schles_def}
\begin{equation}
  \begin{split}
    K^{\tac, (\al + 1)}(x,y; s, \tau)
    = {}& \frac{1}{2\pi i(x^2-y^2)} \left( -1, 0, 1, 0 \right) \Sigma M_{\nu}(y; s, \tau)^{-1} \Sigma \left( I + \frac{\alpha_{\nu}}{y} R_{\nu} \right)^{-1} \\
    & \phantom{\frac{1}{2\pi i(x^2-y^2)}} \times D \left( I + \frac{\alpha_{\nu}}{x} R_{\nu} \right) \Sigma M_{\nu}(x; s, \tau) \Sigma \left( 1, 0, 1, 0 \right)^T \\ 
    = {}& \frac{1}{2\pi i(x^2-y^2)} \left( -1, 0, 1, 0 \right) M_{\nu}(y; s, \tau)^{-1}\Sg D \Sg M_{\nu}(x; s, \tau) \left( 1, 0, 1, 0 \right)^T,
  \end{split}
\end{equation} 
where the second equality follows from the following identities which are easily checked:
\begin{equation}
  \left( I + \frac{\alpha_\nu}{z}R_\nu \right)^{-1} = \left( I - \frac{\alpha_\nu}{z}R_\nu \right), \qquad R_\nu D R_\nu = 0, \qquad
  \frac{ R_\nu D }{2y} = \frac{D R_\nu }{2x} =  R_\nu .
\end{equation}
Then similar to \eqref{eq:kernel_reflecting_expanded}, we have
\begin{equation}
  \begin{split}
    K^{\tac, (\al + 1)}(x,y; s, \tau) = {}& \frac{1}{2\pi i(x-y)} \Big( n_{\nu}(y; s, -\tau)_1 n_{\nu}(x; s, \tau)_3 + n_{\nu}(y; s, -\tau)_2 n_{\nu}(x; s, \tau)_4 \\
  & \phantom{\frac{1}{2\pi i(x-y)} \Big(} - n_{\nu}(y; s, -\tau)_3 n_{\nu}(x; s, \tau)_1 - n_{\nu}(y; s, -\tau)_4 n_{\nu}(x; s, \tau)_2 \Big) \\
  & - \frac{1}{2\pi (x + y)} \Big( n_{\nu}(y; s, -\tau)_1 n_{\nu}(x; s, \tau)_4 + n_{\nu}(y; s, -\tau)_2 n_{\nu}(x; s, \tau)_3 \\
  & \phantom{+ \frac{1}{2\pi (x + y)} \Big(} + n_{\nu}(y; s, -\tau)_3 n_{\nu}(x; s, \tau)_2 + n_{\nu}(y; s, -\tau)_4 n_{\nu}(x; s, \tau)_1 \Big).
  \end{split}
\end{equation}

Similarly,
\begin{equation}
  \begin{aligned}
    K^{\tac, (\al+2)}(x,y; s, \tau) = {}& K^{\tac, (\al)}(x,y; s, \tau) \\
    &-\frac{a_\nu}{\pi ixy} \left( -1, 0, 1, 0 \right) M_{\nu}(y; s, \tau)^{-1} \Sigma R_\nu \Sigma M_{\nu}(x; s, \tau) \left( 1, 0, 1, 0 \right)^T,
\end{aligned}
\end{equation} 
and inductively
\begin{multline} \label{eq:general_hard_tac_formula}
  2\pi i(x^2 - y^2) K^{\tac, (\al + k)}(x,y; s, \tau) = \\
  \left( -1, 0, 1, 0 \right) \Sigma^k K^{-1} M_{\nu}(y; s, -\tau) K \left( \Sigma - \frac{\alpha_{\nu}}{y} \Sigma R_{\nu} \right) \dotsb \left( \Sigma - \frac{\alpha_{\nu + k - 1}}{y} \Sigma R_{\nu + k - 1} \right) \\
  \times D \left( \Sigma + \frac{\alpha_{\nu + k - 1}}{x} R_{\nu + k - 1} \Sigma \right) \dotsb \left( \Sigma + \frac{\alpha_{\nu + k - 1}}{x} R_{\nu + k - 1} \Sigma \right) M_{\nu}(x; s, \tau) \Sigma^k \left( 1, 0, 1, 0 \right)^T,
\end{multline} 
which is a linear combination of $n_{\nu}(y; s, -\tau)_i n_{\nu}(x; s, \tau)_j$ with $i, j = 1, 2, 3, 4$.

For general values of $\nu$, the kernel \eqref{eq:Delvaux_kernel} has the property that its derivative with respect to $s$ is a rank-1 kernel. Namely we have the following proposition.
\begin{prop} \label{prop:derivative_of_tac}
  The kernel \eqref{eq:Delvaux_kernel} satisfies
  \begin{equation}\label{eq:ds_rank1}
    \frac{\d}{\d s} K^{\tac, (\al)}(x,y; s, \tau)=-\frac{F_\nu(x; s, \tau) F_\nu(y; s, -\tau)}{\pi},
  \end{equation}
  where $F_\nu(z; s, \tau)$ is the following combination of the matrix entries of $M_\nu(z; s,\tau)$:
  \begin{equation}\label{eq:def_Fnu}
    \begin{split}
      F_\nu(z; s, \tau) := {}& \left( 1, 1, 0, 0 \right) M_{\nu}(z; s, \tau) \left( 1, 0, 1, 0 \right)^T \\
      = {}& M_\nu(z; s, \tau)_{11}+M_\nu(z; s, \tau)_{13}+M_\nu(z; s, \tau)_{21}+M_\nu(z; s, \tau)_{23} \\
      = {}& n_{\nu}(z; s, \tau)_1 + n_{\nu}(z; s, \tau)_3.
    \end{split}
  \end{equation}
\end{prop}

\begin{proof}
  The $s$-derivative of $K^{\tac, (\alpha)}$ is
  \begin{equation}\label{eq:diff_ker0}
    \begin{aligned}
      \frac{\d}{\d s}& K^{\tac, (\al)}(u,v; s, \tau)= \\
      &\frac{1}{2\pi i(x^2-y^2)} \left( -1, 0, 1, 0 \right) \frac{\partial}{\partial s}\left[M_\nu(y; s, \tau)^{-1} D M_\nu(x; s, \tau)\right] \left( 1, 0, 1, 0 \right)^T.
    \end{aligned}
  \end{equation}
  Writing for the moment $M(z) \equiv M_\nu(z; s, \tau),$ we have
  \begin{equation}\label{eq:diff_ker1}
    \frac{\d}{\d s}\left[M(y)^{-1} D M(x)\right]=M(y)^{-1}\left[D\left[\frac{\d}{\d s} M(x)\right]M(x)^{-1} - \left[\frac{\d}{\d s} M(y)\right] M(y)^{-1} D \right]M(x).
  \end{equation} 
  Using \eqref{eq:Lax_V_part}, we find that \eqref{eq:diff_ker1} can be written as
  \begin{equation}\label{eq:diff_ker2}
    \frac{\d}{\d s}\left[M(y)^{-1} D M(x)\right]=M(y)^{-1}\left[DV(x) -V(y) D \right]M(x).
  \end{equation} 
  Due to the special structure of $D$ and $V$, we have
  \begin{equation}
    DV(x) -V(y) D=-2i(x^2-y^2)
    \begin{pmatrix}
      0_{2 \times 2} & 0_{2 \times 2} \\
      J_2 & 0_{2 \times 2}
    \end{pmatrix},
    \quad J_2 =
    \begin{pmatrix}
      1 & 1 \\
      1 & 1
    \end{pmatrix},
  \end{equation}
  and therefore \eqref{eq:diff_ker2} is 
  \begin{equation}\label{eq:diff_ker4}
    \frac{\d}{\d s}\left[M(y)^{-1} D M(x)\right]=-2i(x^2-y^2)M(y)^{-1}
    \begin{pmatrix}
      0 & 0 \\
      J_2 & 0
    \end{pmatrix}
    M(x),
  \end{equation} 
  and \eqref{eq:diff_ker0} is
\begin{equation}\label{eq:diff_ker5}
    \frac{\d}{\d s} K^{\tac, (\al)}(x,y; s, \tau)= 
    -\frac{1}{\pi} \left( -1, 0, 1, 0 \right) M(y; s, \tau)^{-1}
    \begin{pmatrix}
      0 & 0 \\
      J_2 & 0
    \end{pmatrix}
    M(x; s, \tau) \left( 1, 0, 1, 0 \right)^T.
\end{equation}
Using \eqref{eq:inverse_conj_symmetry} to write the entries of $M_{\nu}(y; s, \tau)^{-1}$ into those of $M_{\nu}(y; s, -\tau)$, we obtain
\begin{equation}\label{eq:diff_ker6}
  \begin{aligned}
    \frac{\d}{\d s} K^{\tac, (\al)}(x,y; s, \tau)= {}& -\frac{1}{\pi} \left( -1, 0, 1, 0 \right) K^{-1} M_\nu(y; s, -\tau)^T K
    \begin{pmatrix}
      0 & 0 \\
      J_2 & 0
    \end{pmatrix}
    M_\nu(x; s,\tau) \left( 1, 0, 1, 0 \right)^T \\
    = {}& -\frac{1}{\pi} \left( 1, 0, 1, 0 \right) M_\nu(y; s, -\tau)^{T}
    \begin{pmatrix}
      J_2 & 0 \\
      0 & 0
    \end{pmatrix}
    M_\nu(x; s,\tau) \left( 1, 0, 1, 0 \right)^T,
  \end{aligned}
\end{equation}
which is \eqref{eq:ds_rank1}.
\end{proof}

The Schlesinger transformation for $M_{\nu}(z; s, \tau)$ implies identities for $F_\nu(z; s, \tau)$. First, Proposition \ref{Schlesinger} yields
\begin{equation}\label{eq_Fnuplus1}
  \begin{split}
    F_{\nu+1}(z; s, \tau) = {}& \left( 1, 1, 0, 0 \right) \left( I + \frac{\alpha_{\nu}}{z} R_{\nu} \right) \Sigma M_{\nu}(z; s, \tau) \Sigma \left( 1, 0, 1, 0 \right)^T \\
    = {}& \left( 1, -1, 0, 0 \right) M_\nu(z; s, \tau) \left( 1, 0, 1, 0 \right)^T \\
    = {}& M_\nu(z; s, \tau)_{11} + M_\nu(z; s, \tau)_{13} - M_\nu(z; s, \tau)_{21} - M_\nu(z; s, \tau)_{23} \\
    = {}& n_{\nu}(z; s, \tau)_1 - n_{\nu}(z; s, \tau)_2.
  \end{split}
\end{equation} 
Also
\begin{equation} \label{eq:Fnuplus2}
  \begin{split}
    F_{\nu + 2}(z; s, \tau) = {}& \left( 1, 1, 0, 0 \right) \left( I + \frac{\alpha_{\nu + 1}}{z} R_{\nu + 1} \right) \Sigma \left( I + \frac{\alpha_{\nu}}{z} R_{\nu} \right) \Sigma M_{\nu}(z; s, \tau) \left( 1, 0, 1, 0 \right)^T \\
    = {}& F_{\nu}(z; s, \tau) + \frac{2\alpha_{\nu}}{z} \left( \beta_{\nu} n_{\nu}(z; s, \tau)_1 - \beta_{\nu} n_{\nu}(z; s, \tau)_2 - in_{\nu}(z; s, \tau)_3 - in_{\nu}(z; s, \tau)_4 \right).
  \end{split}
\end{equation}
Inductively,
\begin{equation} \label{eq:inductive_F_nu+k}
  F_{\nu + k}(z; s, \tau) = \left( 1, 1, 0, 0 \right) \left( \Sigma + \frac{\alpha_{\nu + k - 1}}{z} R_{\nu + k - 1} \Sigma \right) \dotsb \left( \Sigma + \frac{\alpha_{\nu}}{z} R_{\nu} \Sigma \right) M_{\nu}(z; s, \tau) \left( 1, 0, 1, 0 \right)^T,
\end{equation}
is a linear combination of $n_{\nu}(z; s, \tau)_i$ ($i = 1, 2, 3, 4$).
We expect the following vanishing property:
\begin{conj} \label{conj:vanishing}
  Let $\alpha > -1$. For fixed $x, y \in \realR_+$ and $\tau \in \realR$,
  \begin{equation}
    \lim_{s \to \infty} K^{\tac, (\alpha)}(x, y; s, \tau) = 0.
  \end{equation}
\end{conj}
If Conjecture \ref{conj:vanishing} is true, then Proposition \ref{prop:derivative_of_tac} implies the formula
\begin{equation}\label{int_rank1}
  K^{\tac, (\al)}(x,y; s, \tau) = \frac{1}{\pi} \int_s^\infty F_\nu(x; \widetilde{s}, \tau) F_\nu(y; \widetilde{s}, -\tau)\,d\widetilde{s}.
\end{equation}

We now consider the case $\nu=0$. In the paper \cite{Kuijlaars14}, a matrix-valued function $M(z)$ is defined as a solution to a Riemann--Hilbert problem, in six sectors, with parameters $r_1, r_2, s_1, s_2, \tau$. It is clear that our matrix-valued function $M_{\nu=0}(z; s, \tau)$ agrees $M(z)$ defined in sector $\Omega_0 = \{ z \mid 0 < \arg z < \pi/3 \}$, with parameters $r_1 = r_2 = 1$, $s_1 = s_2 = s$, and the same $\tau$. In \cite{Liechty-Wang16}, a more general Riemann--Hilbert problem is considered, with the same parameters $r_1, r_2, s_1, s_2, \tau$ and more parameters $t_1, t_2, t_3$. If $(t_1, t_2, t_3) = (1, 0, -1)$, the Riemann--Hilbert problem in \cite{Liechty-Wang16} is essentially the same as that in \cite{Kuijlaars14}, up to a multiplication by a constant matrix in sector $\Omega_0$, see \cite[Section 1.4.3]{Liechty-Wang16}. The column vector $\vec{n}_{\nu}(z; s, \tau)$ defined in \eqref{eq:vect_n} is equal to the vector $m^{(0)} + m^{(3)}$ in \cite[Theorem 2]{Kuijlaars14}, and equal to the vector $n^{(0)} - n^{(3)}$ in \cite[Theorem 1.4]{Liechty-Wang16}, see \cite[Formulas (1.42) and (1.70)]{Liechty-Wang16}.

By \cite[Formula (1.36)]{Liechty-Wang16}, we have that
\begin{equation}\label{eq:n0_formulas}
  \begin{gathered}
    n_0(z; s, \tau)_1 = \frac{2^{1/6}}{\sqrt{\pi}} \int_{\Sg_T} e^{2^{4/3}\tau \zeta^2 + 2^{2/3} i z \zeta} f(\z; \sg) d\zeta, \quad n_0(z; s, \tau)_2 = \frac{2^{1/6}}{\sqrt{\pi}} \int_{\Sg_T} e^{2^{4/3}\tau \zeta^2 + 2^{2/3} i z \zeta} g(\z; \sg) d\zeta, \\
    n_0(z; s, \tau)_3 = \frac{2^{1/6}}{\sqrt{\pi}} \int_{\Sg_T} e^{2^{4/3}\tau \zeta^2 + 2^{2/3} i z \zeta} \left[ \left( i(\tau - s^2 + 2^{-1/3} u_0) + 2^{2/3} \zeta \right) f(\z; \sg) + 2^{-1/3}i q_0 g(\zeta; \sigma) \right] d\zeta, \\
    n_0(z; s, \tau)_4 = \frac{2^{1/6}}{\sqrt{\pi}} \int_{\Sg_T} e^{2^{4/3}\tau \zeta^2 + 2^{2/3} i z \zeta} \left[ \left( i(-\tau + s^2 - 2^{-1/3} u_0) + 2^{2/3} \zeta \right) g(\z; \sg) - 2^{-1/3}i q_0 f(\zeta; \sigma) \right] d\zeta,
  \end{gathered}
\end{equation}
where $q_0=q_0(\sg)$ and $u_0=u_0(\sg)$ are defined in \eqref{PII}, \eqref{eq:PII_BC}, and \eqref{eq:defn_u_nu}, with $\nu = 0$. Then by \eqref{eq:def_Fnu}, \eqref{eq_Fnuplus1}, and \eqref{eq:Fnuplus2}
\begin{align}
  F_0(z; s, \tau) = {}& \frac{2^{1/6}}{\sqrt{\pi}} \int_{\Sg_T} e^{2^{4/3}\tau \zeta^2 + 2^{2/3} i z \zeta} \big( f(\z; \sg) + g(\zeta; \sigma) \big) d\zeta, \label{eq:integral_F_0} \\
  F_1(z; s, \tau) = {}& \frac{2^{1/6}}{\sqrt{\pi}} \int_{\Sg_T} e^{2^{4/3}\tau \zeta^2 + 2^{2/3} i z \zeta} \big( f(\zeta; \sigma) - g(\z; \sg) \big) d\zeta, \label{eq:integral_F_1} \\
  F_2(z; s, \tau) = {}& F_0(z; s, \tau) + \frac{\alpha_0(4\tau - 2^{5/3} q_0)}{z} F_1(z; s, \tau) \notag \\
  & - \frac{2^{11/6} \alpha_0 i}{\sqrt{\pi} z} \int_{\Sg_T} e^{2^{4/3}\tau \zeta^2 + 2^{2/3} i z \zeta} \zeta \big( f(\z; \sg) + g(\zeta; \sigma) \big) d\zeta, \label{eq:integral_F_2}
\end{align}
where $\sg=2^{2/3}(2s-\tau^2)$ as in \eqref{eq:sigma_relation}, $f(\zeta; \sigma)$ and $g(\zeta; \sigma)$ are defined as in \eqref{tac18} and $\alpha_0$ is defined in \eqref{Backlund}. 

The formula \eqref{eq:n0_formulas}, along with \eqref{eq:inductive_F_nu+k}, prove Theorem \ref{thm:half-integer} provided that Conjecture \ref{conj:vanishing} holds. We will prove that the conjecture holds for half-integer $\al$ in the next subsection, thus completing the proof of Theorem \ref{thm:half-integer}. But first let us look specifically at the cases $\al=\pm1/2$ in order to prove Proposition \ref{lem:tacnode_relation}.

The $\alpha = -1/2$ case of \eqref{int_rank1} together with \eqref{eq:integral_F_0} implies (letting $\widetilde{\sigma} = 2^{2/3}(2\widetilde{s} - \tau^2)$)
\begin{equation}
  \begin{split}
    & K^{\tac, (-1/2)}(x,y; s, \tau) \\
    = {}& \frac{2^{1/3}}{\pi^2} \int_s^\infty\,d\widetilde{s}  \int_{\Sg_T}\,du\int_{\Sg_T}\,dv e^{2^{4/3}\tau (u^2-v^2) + 2^{2/3} i (x u+y v)} \big(f(u; \widetilde{\sg})+g(u; \widetilde{\sg})\big)  \big(f(v; \widetilde{\sg})+g(v; \widetilde{\sg})\big) \\
    = {}& \frac{1}{2^{4/3}\pi^2} \int_\sg^\infty\,d\widetilde{\sg}  \int_{\Sg_T}\,du\int_{\Sg_T}\,dv\, e^{2^{4/3}\tau (u^2-v^2) -2^{2/3} i (x u-y v)}  \big(f(u; \widetilde{\sg})+g(u; \widetilde{\sg})\big) \big(f(v; \widetilde{\sg})+g(v; \widetilde{\sg})\big),
  \end{split}
\end{equation}
where in the second identity we use the symmetry \eqref{rank3}, and make the change of variable $\widetilde{s}\mapsto \widetilde{\sg}$. Thus we find, by the comparison with \eqref{rank7a}, we obtain the identity \eqref{eq:ref_tac_relation}.

Analogously, the $\alpha = 1/2$ case of \eqref{int_rank1} together with \eqref{eq:integral_F_1} implies (letting $\widetilde{\sigma} = 2^{2/3}(2\widetilde{s} - \tau^2)$)
\begin{equation}
  \begin{split}
    & K^{\tac, (1/2)}(x,y; s, \tau) \\
    = {}& \frac{2^{1/3}}{\pi^2} \int_s^\infty\,d\widetilde{s}  \int_{\Sg_T}\,du\int_{\Sg_T}\,dv e^{2^{4/3}\tau (u^2-v^2) + 2^{2/3} i (x u+y v)} \big(f(u; \widetilde{\sg}) - g(u; \widetilde{\sg})\big)  \big(f(v; \widetilde{\sg}) - g(v; \widetilde{\sg})\big) \\
    = {}& \frac{1}{2^{4/3}\pi^2} \int_\sg^\infty\,d\widetilde{\sg}  \int_{\Sg_T}\,du\int_{\Sg_T}\,dv\, e^{2^{4/3}\tau (u^2-v^2) -2^{2/3} i (x u-y v)}  \big(f(u; \widetilde{\sg}) - g(u; \widetilde{\sg})\big) \big(f(v; \widetilde{\sg}) - g(v; \widetilde{\sg})\big).
  \end{split}
\end{equation}
Thus we find, by the comparison with \eqref{rank8a}, we obtain the identity \eqref{eq:abs_tac_relation}, with $t, \sigma$ related to $\tau, s$ by \eqref{eq:relation_t_sigma_s_tau}. This completes the proof of Proposition \ref{lem:tacnode_relation}.

\subsection{Proof of Conjecture \ref{conj:vanishing} when $\alpha = k + 1/2$}

By \eqref{eq:general_hard_tac_formula} with $\alpha = -1/2$, $2\pi i(x^2 - y^2)K^{\tac, (k + 1/2)}(x, y; s, \tau)$ can be expressed as a linear combination of $n_0(y; s, \tau)_i n_0(x; s, \tau)_j$ with $i, j = 1, 2, 3, 4$, so when $x \neq y$, we only need to show that for all $z \in (0, \infty)$ and $\tau \in \realR$, $n_0(z; s, \tau)_i \to \infty$ as $s \to +\infty$.

Since $\vec{n}_0(z; s, \tau) = m^{(0)} + m^{(3)}$ in the notation of \cite{Kuijlaars14}, with parameters $r_1 = r_2 = 1$ and $s_1 = s_2 = s$ and the same $\tau$. Let $\Ai(z)$ denote the Airy function, and for any real number $t$ let $Q_t(x)$ and $R_t(x, t)$ be the functions in $x$ defined by Airy resolvents, as in \cite[Formulas (2.16)--(2.20)]{Kuijlaars14}. Then we have, by \cite[Theorem 2.5]{Kuijlaars14}
\begin{multline}
  n_0(z; s, \tau)_1 = -\sqrt{2\pi} e^{-\tau z} \int^{\infty}_0 \Ai(z + 2s + 2^{1/3} w) e^{-2^{1/3}\tau w} Q_{\sigma}(w + \sigma) dw \\
  + \sqrt{2\pi} \Ai(-z + 2s) e^{\tau z} + \sqrt{2\pi} e^{\tau z} \int^{\infty}_0 \Ai(-z + 2s + 2^{1/3} w) e^{-2^{1/3}\tau w} R_{\sigma}(w + \sigma, \sigma) dw,
\end{multline}
\begin{multline}
  n_0(z; s, \tau)_2 = \sqrt{2\pi} \Ai(z + 2s) e^{-\tau z} + \sqrt{2\pi} e^{-\tau z} \int^{\infty}_0 \Ai(z + 2s + 2^{1/3} w) e^{-2^{1/3}\tau w} R_{\sigma}(w + \sigma, \sigma) dw \\
   - \sqrt{2\pi} e^{\tau z} \int^{\infty}_0 \Ai(-z + 2s + 2^{1/3} w) e^{-2^{1/3}\tau w} Q_{\sigma}(w + \sigma) dw,
\end{multline}
where $\sigma = 2^{2/3}(2s - \tau^2)$ as before, and by \cite[Formulas (2.14) and (2.15)]{Kuijlaars14}, $n_0(z; s, \tau)_3$ and $n_0(z; s, \tau)_4$ are linear combinations of $n_0(z; s, \tau)_1, n_0(z; s, \tau)_2$, and their derivatives with respect to $z$, such that the coefficients of the linear combinations are either independent of $s$ or polynomials in $s$.

It is well known that $\Ai(z)$ vanishes as $\exp(-(2/3)z^{3/2})$ as $z \to +\infty$ (see \cite[Formula (2.26)]{Kuijlaars14} for instance), and by the definitions of $Q_t(x)$ and $R_t(x, t)$, it is not hard to see that if $t$ is large enough, then
\begin{equation} \label{eq:estimate_Q_R}
 \lvert Q_t(x) \rvert < 1 \quad \text{and} \quad \lvert R_t(x, t) \rvert < 1 \quad \text{for all $x \in (t, \infty)$}. 
\end{equation}
(Actually stronger estimates of $Q_t(x)$ and $R_t(x, t)$ are possible, see \cite[Lemma 2.4]{Kuijlaars14} for the leading term in the asymptotic expansion of $Q_t(x)$ and $R_t(x, t)$ when $t$ is fixed and $x \to \infty$. But the crude estimate \eqref{eq:estimate_Q_R} suffices for us.)

Then we see that as $z, \tau$ are fixed and $s \to +\infty$, $n_0(z; s, \tau)_1$ and $n_0(z; s, \tau)_2$ vanish super-exponentially, and so do $n_0(z; s, \tau)_3$ and $n_0(z; s, \tau)_4$ who are the linear combinations of $n_0(z; s, \tau)_1$, $n_0(z; s, \tau)_2$ and their derivatives. Thus we prove Conjecture \ref{conj:vanishing} when $\alpha = k + 1/2$ and $x \neq y$.

For the remaining $x = y$ case, we use the property that $K^{\tac, (k + 1/2)}(x, y; \tau)$ is an analytic function in $x$ and $y$, although only positive real values of $x, y$ are meaningful in probability. Given $y > 0$, we consider $x = y + \epsilon e^{i\theta}$ for a small enough $\epsilon > 0$ and $\theta \in [0, 2\pi]$. Then by the argument above, for all $x$ on a small circle around $y$, $K^{\tac, (k + 1/2)}(x, y; s, \tau)$ vanishes uniformly as $s \to +\infty$. Thus by the analyticity of $K^{\tac, (k + 1/2)}(x, y; s, \tau)$ in $x$, we find that $K^{\tac, (k + 1/2)}(y, y; s, \tau)$ vanishes as $s \to +\infty$. Thus we complete the proof for Conjecture \ref{conj:vanishing} when $\alpha = k + 1/2$.

\appendix

\section{Formulas for $t^c$ and $d$ in Theorem \ref{thm:main}\ref{enu:thm:main_a}} \label{sec:formulas_t^c_d}

 Here we give the explicit, though not very simple, formula for $t^c$ and $d$ in Theorem \ref{prop:circle}\ref{enu:prop:circle_b}, in terms of $T$. Our formulas are taken from \cite{Liechty-Wang14-2}. First, we parametrize $T > \pi^2/2$ by $k \in (0, 1)$. For each $k$, we have the elliptic integrals
\begin{equation}\label{eq:complete_elliptic}
  \K :=  \K(k) = \int^1_0 \frac{ds}{\sqrt{(1 - s^2)(1 - k^2 s^2)}}, \qquad \E :=  \E(k) = \int^1_0 \frac{\sqrt{1 - k^2 s^2}}{\sqrt{1 - s^2}} ds.
\end{equation}
We further define
\begin{equation}\label{eq:def_tilde1}
  \widetilde{k} := \frac{2\sqrt{k}}{1 + k},
\end{equation}
and denote
\begin{equation}\label{eq:def_tilde2}
  \widetilde{\K} :=  \K(\widetilde{k}) = \int^1_0 \frac{ds}{\sqrt{(1 - s^2)(1 - \widetilde{k}^2 s^2)}}, \qquad \widetilde{\E} := \E(\widetilde{k}) = \int^1_0 \frac{\sqrt{1 - \widetilde{k}^2 s^2}}{\sqrt{1 - s^2}} ds.
\end{equation}
By \cite[Lemma 3.2]{Liechty-Wang14-2}\footnote{In \cite[Formula (149)]{Liechty-Wang14-2}, $\pi^2$ should be $\pi^2/4$.}, for all $T > T_c = \pi^2/2$, there is a unique $k$ such that
\begin{equation} \label{eq:T_parametrization}
  T = 2\widetilde{\K} \widetilde{\E} = 2 \int^1_0 \frac{ds}{\sqrt{(1 - s^2)(1 - \widetilde{k}^2 s^2)}} \int^1_0 \frac{\sqrt{1 - \widetilde{k}^2 s^2}}{\sqrt{1 - s^2}} ds.
\end{equation}
Then $t^c$ is expressed as \cite[Formula (20)]{Liechty-Wang14-2}
\begin{equation} \label{eq:T_c_in_introduction}
  t^c = \frac{2}{\widetilde{k}^2} \widetilde{\E} \left(\widetilde{\E} - (1 - \widetilde{k}^2) \widetilde{\K}\right) = 2 \int^1_0 \frac{\sqrt{1 - \widetilde{k}^2 s^2}}{\sqrt{1 - s^2}} ds \int^1_0 \frac{\sqrt{1 - s^2}}{\sqrt{1 - \widetilde{k}^2 s^2}} ds \in (0, T/2),
\end{equation}
and $d$ is expressed as \cite[Formulas (227), (236), and (239)]{Liechty-Wang14-2}
\begin{equation} \label{eq:d_defn}
  d = \left( \frac{1}{6\alpha^3}((1 + k^2) \E - (1 - k^2) \K) \right)^{\frac{1}{4}} = \left( \frac{k^2}{6 \alpha^3} \int^1_0 \frac{(1 - s^2) + (1 - k^2 s^2)}{\sqrt{(1 - s^2)(1 - k^2 s^2)}} ds \right)^{\frac{1}{4}} > 0.
\end{equation}
\begin{rmk}
 The parametrization of $T$ in \eqref{eq:T_parametrization} and the definition of $t^c$ in \eqref{eq:T_c_in_introduction} differ from those in \cite{Liechty-Wang14-2} by a factor of 2, see Remark \ref{rem:T_difference}.
\end{rmk}

\section{Equivalence with the kernel of Ferrari--Vet\H o} \label{sec:equivalence_FV}

We would like to show that the hard-edge tacnode kernel of Ferrari and \Veto\ appearing in \cite{Ferrari-Veto16} is the odd part of the symmetric tacnode kernel studied by several groups \cite{Delvaux-Kuijlaars-Zhang11, Adler-Ferrari-van_Moerbeke13, Johansson13} and culminated in their paper \cite{Ferrari-Veto12}. Since we only need the symmetric tacnode kernel, we take $K^{\tac}_{s, t}(\xi, \eta; \sigma)$ defined in \eqref{eq:tacnode_kernel} and \eqref{eq:essential_tacnode} as the standard form, as we do throughout the paper, and note that the kernel defined in \cite{Delvaux13} with the implicit parameter $\lambda = 1$,
\begin{equation}
  \lcal_{\tac}(u,v;\sg,\tau_1, \tau_2) = \widetilde{\lcal}_{\tac}(u,v;\sg,\tau_1, \tau_2) - 1_{\tau_1 < \tau_2} \phi_{2\tau_1, 2\tau_2}(u, v),
\end{equation}
where $\phi_{s, t}(u, v)$ is defined in \eqref{eq:nonessential_Pearcey} and $\widetilde{\lcal}_{\tac}(u,v;\sg,\tau_1, \tau_2)$ defined in \cite[Formula (2.29)]{Delvaux13} and \cite[Formula (45)]{Liechty-Wang14-2}, is equivalent to $K^{\tac}_{s, t}(\xi, \eta; \sigma)$ and satisfies \cite[Proposition 1.5]{Liechty-Wang14-2} 
\begin{equation}
  K^{\tac}_{\tau_1, \tau_2}(\xi, \eta; \sigma) = 2^{-2/3} \lcal_{\tac}(2^{-2/3} \xi, 2^{-2/3} \eta; \sigma, 2^{-7/3} \tau_1, 2^{-7/3} \tau_2).
\end{equation}
Thus to prove Proposition \ref{prop:FV_equiv}, we only need to show that
\begin{multline} \label{eq:ess_FV_equiv}
  \widehat{K}^{\ext}(\tau_1, u; \tau_2, v) + 1_{\tau_1 < \tau_2} [\phi_{2\tau_1, 2\tau_2}(u, v) - \phi_{2\tau_1, 2\tau_2}(u, -v)] = \\
  \widetilde{\lcal}_{\tac}(u,v; 2^{2/3}R,\tau_1, \tau_2) -\widetilde{\lcal}_{\tac}(u,-v; 2^{2/3}R, \tau_1, \tau_2),
\end{multline}
where $R$ is the parameter used implicitly in \cite[Formulas (2.34), (2.30) and (2.31)]{Ferrari-Veto16}. Our strategy to prove \eqref{eq:ess_FV_equiv} is to transform its right-hand side repeatedly, and at last show that it agrees with the left-hand side. The derivation below is based on formulas in \cite{Liechty-Wang14-2} and \cite{Baik-Liechty-Schehr12}.

By  \cite[Formula (2.29)]{Delvaux13}, the right-hand side of \eqref{eq:ess_FV_equiv} is expressed as
\begin{equation} \label{eq:integral_FV_to_be_equiv}
  \frac{1}{2^{2/3}} \int^{\infty}_{2^{2/3}R} [\widehat{p}_1(u; s, \tau_1) - \widehat{p}_1(-u; s, \tau_1)] [\widehat{p}_1(v; s, -\tau_2) - \widehat{p}_1(-v; s, -\tau_2)] ds,
\end{equation}
where $\widehat{p}_1$ is defined in \cite[Formula (2.26) and Lemma 4.3]{Delvaux13}. Note that since we only consider the symmetric tacnode kernel with $\lambda = 1$, $\widehat{p}_1$ and $\widehat{p}_2$ defined in \cite[Formula (2.26)]{Delvaux13} are equivalent, see \cite[Formulas (45) and (46)]{Liechty-Wang14-2}. Next, we recall the function $b_{\tau, z, \sigma}(x)$ defined in \cite[Formula (43)]{Liechty-Wang14-2}, the integral operators $\B_s$ and $\A_s = \B^2_s$ defined in \cite[Formula (40)]{Liechty-Wang14-2} (They were  defined in \cite{Baik-Liechty-Schehr12} and \cite{Delvaux13} with different notations). By \cite[Formulas (41) and (46)]{Liechty-Wang14-2}, we have that the integral \eqref{eq:integral_FV_to_be_equiv} can be expressed as
\begin{equation}\label{eq:formula1}
  \frac{1}{2^{2/3}} \int^\infty_{2^{2/3}R} \langle b_{\tau_1, u, s} - b_{\tau_1, -u, s}, \left( \1 - \B_s\right)^{-1} \de_0\rangle_0   \langle b_{-\tau_2, v, s} - b_{-\tau_2, -v, s}, \left( \1 - \B_s\right)^{-1} \de_0\rangle_0\, ds,
\end{equation}
where $\langle \cdot, \cdot \rangle_0$ is the inner product on $L_2[0,\infty)$, see \cite[Formula (42)]{Liechty-Wang14-2}. Now denote
\begin{equation}\label{def:psi}
  \psi(x; z, s, \tau):=e^{-\frac{2}{3}\tau^3 - \tau x - \tau s} \left[e^{-\tau z} \Ai(x+z+s+\tau^2) - e^{\tau z} \Ai(x-z+s+\tau^2)\right],
\end{equation}
and
\begin{equation}
  h(z, s, \tau):=\langle \psi(2^{1/3} x; z, 2^{-2/3} s, \tau) , \left( \1 - \B_s\right)^{-1} \de_0\rangle_0,
\end{equation}
where $x$ is the variable of integration in the inner product. Then \eqref{eq:formula1} is written as
\begin{equation}\label{eq:formula2}
    \frac{1}{2^{2/3}} \int^\infty_{2^{2/3}R} h(u, s, \tau_1) h(v, s, -\tau_2) \, ds
    = 2^{1/3} \int_0^\infty h(u, 2y + 2^{2/3}R, \tau_1) h(v, 2y + 2^{2/3}R, -\tau_2) \, dy.
\end{equation}
Note that our $\psi(x; z, s, \tau)$ is equal to $\widehat{\Phi}^x_{-\tau}(z)$ and $\widehat{\Psi}^x_{\tau}(z)$ defined in \cite[Formula (2.30)]{Ferrari-Veto16}.


Following  \cite[Section 2]{Baik-Liechty-Schehr12}, we now change to work in $L_2(\realR)$ instead of $L_2[0,\infty)$. We use $\widetilde{\B}_s$ to denote the operator with the same kernel as $\B_s$, but acting on $L_2(\realR)$ instead of $L_2[0,\infty)$. Also for any $r\in \realR$ introduce the operator $\Pi_r$ to be the projection onto $L_2[r,\infty)$ and $\T_r$ to be the translation operator, i.e. $(\T_r f)(x)= f(x+r)$.
Repeating the analysis in \cite[Section 2]{Baik-Liechty-Schehr12} we find that $h(u, 2y+\sg, \tau)$ can be written as ($\langle \cdot, \cdot \rangle$ is the inner product in $L^2(\realR)$)
\begin{equation} \label{eq:h_by_BLS}
\begin{aligned}
h(u, 2y+\sg, \tau) &=\langle \psi(2^{1/3} x; u, 2^{-2/3} (2y+\sg), \tau), \Pi_0(\1 - \Pi_0 \widetilde{\B}_{2y+\sg} \Pi_0)^{-1} \Pi_0 \de_0\rangle \\
&=\langle \psi(2^{1/3} x; u, 2^{-2/3} (2y+\sg), \tau),  \Pi_y(\1 - \Pi_y \widetilde{\B}_{\sg} \Pi_y)^{-1} \Pi_y \T_{-y}\de_0\rangle \\
&=\langle \T_{-y}\psi(2^{1/3} x; u, 2^{-2/3} (2y+\sg), \tau),  \Pi_y(\1 - \Pi_y \widetilde{\B}_{\sg} \Pi_y)^{-1} \Pi_y \de_y\rangle. \\
\end{aligned}
\end{equation}
Noticing that 
\begin{equation}
\psi(2^{1/3} (x-y); u, 2^{-2/3} (2y+\sg), \tau) = \psi(2^{1/3} x; u, 2^{-2/3} \sg, \tau),
\end{equation}
we find that with $\sigma = 2^{2/3}R$,
\begin{equation}
  h(u, 2y + 2^{2/3}R, \tau) = \langle \psi(2^{1/3} x; u, R, \tau),  \Pi_y(\1 - \Pi_y \widetilde{\B}_{2^{2/3}R} \Pi_y)^{-1} \Pi_y \de_y\rangle.
\end{equation}
Define the resolvent 
\begin{equation}
  {\bf R}_y:= (\1 - \Pi_y \B_{2^{2/3}R} \Pi_y)^{-1} - \1.
\end{equation}
Then we can write $h(u, 2y + 2^{2/3}R, \tau_1) $ and $h(v, 2y + 2^{2/3}R, -\tau_2) $ as
\begin{equation}\label{eq:shifted_h}
  \begin{aligned}
    h(u, 2y + 2^{2/3}R, \tau_1) &=  \Psi_1(y)+  \int_{-\infty}^\infty \Psi_1(x) {\bf R}_y(x,y)\,dx, \\
    h(v, 2y + 2^{2/3}R, -\tau_2) &=  \Psi_2(y)+ \int_{-\infty}^\infty \Psi_2(x) {\bf R}_y(x,y)\,dx, 
\end{aligned}
\end{equation}
where we set
\begin{equation} \label{eq:defn_Psi_12}
\Psi_1(y):=\psi(2^{1/3} y, u, R, \tau_1), \qquad \Psi_2(y):=\psi(2^{1/3} y, v, R, -\tau_2).
\end{equation}
 Thus by \eqref{eq:h_by_BLS} -- \eqref{eq:defn_Psi_12} and using the symmetry of the resolvent kernel, \eqref{eq:formula2} becomes
\begin{multline}\label{eq:four_terms}
2^{1/3}  \bigg[\int_0^\infty\,dy\, \Psi_1(y) \Psi_2(y) +  \int_0^\infty\,dy  \int_{-\infty}^\infty\,dx_2\, \Psi_1(y){\bf R}_y(y,x_2) \Psi_2(x_2) \\
+ \int_0^\infty\,dy \, \int_{-\infty}^\infty\,dx_1\, \Psi_1(x_1) {\bf R}_y(x_1,y)\Psi_2(y) \\
+\int_0^\infty\,dy \int_{-\infty}^\infty\,dx_1 \int_{-\infty}^\infty \,dx_2 \,\Psi_1(x_1) {\bf R}_y(x_1,y){\bf R}_y(y,x_2)\Psi_2(x_2) \bigg].
\end{multline}
Notice that the integrands in \eqref{eq:shifted_h} vanish for $x<y$ due to the resolvent kernel, so all integrals above may be taken over $[0,\infty)$. Replacing $y$ with $x_1$ in the first double integral in \eqref{eq:four_terms}, and $y$ with $x_2$ in the second double integral, we can then write \eqref{eq:four_terms}  as
\begin{multline}\label{eq:four_terms1}
2^{1/3}  \bigg[\int_0^\infty\,dy\, \Psi_1(y) \Psi_2(y) + \\ \int_0^\infty\,dx_1  \int_0^\infty\,dx_2\,  \Psi_1(x_1)\bigg({\bf R}_{x_1}(x_1,x_2)  + {\bf R}_{x_2}(x_1,x_2) 
+\int_0^\infty\,dy \, {\bf R}_y(x_1,y){\bf R}_y(y,x_2)\bigg)\Psi_2(x_2) \bigg].
\end{multline}
Applying \cite[Lemma 2.1]{Baik-Liechty-Schehr12}, this is
\begin{equation}\label{eq:two_terms1}
2^{1/3}  \bigg[\int_0^\infty\,dy\, \Psi_1(y) \Psi_2(y) + \\ \int_0^\infty\,dx_1  \int_0^\infty\,dx_2\,  \Psi_1(x_1){\bf R}_0(x_1, x_2)\Psi_2(x_2) \bigg],
\end{equation}
which is simply
\begin{equation}\label{eq:one_ters}
  2^{1/3} \int_0^\infty\,dx_1  \int_0^\infty\,dx_2\,  \Psi_1(x_1)(\1 -\Pi_0 \widetilde{\B}_{2^{2/3}R} \Pi_0)^{-1}(x_1, x_2)\Psi_2(x_2) .
\end{equation}
We note that $\Psi_1$ and $\Psi_2$ are expressed in $\psi(x; z, s, \tau)$ in \eqref{def:psi}, which is equal to $\widehat{\Phi}^x_{-\tau}(z)$ and $\widehat{\Psi}^x_{\tau}(z)$ defined in \cite[Formula (2.30)]{Ferrari-Veto16}. Hence \eqref{eq:one_ters} is expressed as (with $\zeta = 2^{1/3}x_1$ and $\xi = 2^{1/3} x_2$)
\begin{multline} \label{eq:final_FV_argument}
  2^{1/3} \int_0^\infty\,dx_1  \int_0^\infty\,dx_2\,  \widehat{\Psi}^{2^{1/3} x_1}_{\tau_1} (u)(\1 -\Pi_0 \widetilde{\B}_{2^{2/3}R} \Pi_0)^{-1}(x_1, x_2)\widehat{\Phi}^{2^{1/3} x_2}_{\tau_2} (v)  \\
  =2^{-1/3} \int_0^\infty\,d\xi  \int_0^\infty\,d\z\,  \widehat{\Psi}^{\xi}_{\tau_1} (u)(\1 -\B_{2^{2/3}R})^{-1}(2^{-1/3}\xi, 2^{-1/3}\z)\widehat{\Phi}^{\z}_{\tau_2} (v).
\end{multline}
To prove \eqref{eq:ess_FV_equiv}, we just need to show that the right-hand side of \eqref{eq:final_FV_argument} is equal to the second term of \cite[Equation (2.34)]{Ferrari-Veto16}. They are equal if we have the identity of the operators on $L^2[0, \infty)$
\begin{equation} \label{eq:last_eq}
  2^{-1/3} (\1 -\B_{2^{2/3}R})^{-1}(2^{-1/3}\xi, 2^{-1/3}\z) = (\1 - \widehat{K}_0)^{-1}(\zeta, \xi),
\end{equation}
where $\widehat{K}_0$ is defined in \cite[Formula (2.31)]{Ferrari-Veto16}. Since \eqref{eq:last_eq} is readily checked by the definitions of $\B_s$ and $\widehat{K}_0$, we finish the proof of \eqref{eq:ess_FV_equiv} and hence the proof of Proposition \ref{prop:FV_equiv}.

\subsection{A variation of $\NIBMabs$ similar to the model of Ferrari--\Veto} \label{subsec:similar_model}

Consider $2n$ particles in nonintersecting Brownian brideges between two absorbing walls placed at $\pi$ and $-\pi$, from the common start point $0$ to the common end point $0$, during the time span $[0, T]$. Denote the particles and their trajectories as $-\pi < x_1(t) < \dotsb < x_{2n}(t) < \pi$. Here we assume the particles have diffusion parameter $n^{-1/2}$. A heuristic symmetry argument implies that the trajectories $x_{n + 1}(t), x_{n + 2}(t), \dotsc, x_{2n}(t)$ have roughly the same behavior as the $n$ particles in $\NIBMabs$. If the total time $T$ is close to $T_c = \pi^2/2$, then we observe the limiting hard-edge tacnode process close to $\pi$. Similarly the same limiting process occurs close to $-\pi$. If we remove one absorbing wall, the model becomes the same as that in \cite{Ferrari-Veto16}. When $T$ is close to $T_c$, it is suggested by \cite{Baik-Liu14} that the two absorbing walls affect the nonintersecting Brownian bridges asymptotically independently, as $n \to \infty$. Thus the two-wall model should have the same hard-edge tacnode limit as the one-wall model in \cite{Ferrari-Veto16}.

The argument above is not rigorous. But by the method in this paper, we can solve the aforementioned variation of $\NIBMabs$ and derive the hard-edge tacnode limit rigorously, although our method does not apply directly for the model in \cite{Ferrari-Veto16}.

\subsubsection*{Acknowledgements}

K.L. is supported by a grant from the Simons Foundation (\#357872, Karl Liechty).

\noindent D.W. is supported by the Singapore AcRF Tier 1 grant R-146-000-217-112.

\noindent Both authors thank Patrik Ferrari and Balint \Veto\ for useful comments. We also thank Tom Claeys for discussion at the early stage of the project.


\end{document}